\documentclass[12pt]{amsart}
\usepackage{latexsym}
\usepackage{amssymb}
\usepackage{amscd}
\usepackage{mathrsfs}
\renewcommand\a{\alpha}
\renewcommand\b{\beta}
\newcommand\g{\gamma}

\newcommand\la{\lambda}
\newcommand\z{\zeta}

\renewcommand\th{\theta}

\newcommand\io{\iota}

\newcommand\s{\sigma}

\newcommand\f{\phi}
\newcommand\vf{\varphi}

\renewcommand\r{\rho}

\newcommand\w{\omega}

\newcommand\vG{\varGamma}
\newcommand\ve{\varepsilon}

\newcommand{\OO}{\mathbb O}

\newcommand\SB{\mathscr{B}}

\newcommand\SD{\mathscr{D}}
\newcommand\SE{\mathscr{E}}
\newcommand\SF{\mathscr{F}}
\newcommand\SG{\mathscr{G}}

\newcommand\SL{\mathscr{L}}
\newcommand\SM{\mathscr{M}}
\newcommand\SN{\mathscr{N}} 
\newcommand\SO{\mathscr{O}}
\newcommand\SP{\mathscr{P}}
\newcommand\SQ{\mathscr{Q}}

\newcommand\ST{\mathscr{T}}
\newcommand\SH{\mathscr{H}}
\newcommand\SI{\mathscr{I}}

\newcommand\SW{\mathscr{W}}
\newcommand\SV{\mathscr{V}}
\newcommand\SU{\mathscr{U}}
\newcommand\SX{\mathscr{X}}
\newcommand\SY{\mathscr{Y}}
\newcommand\SZ{\mathscr{Z}}

\newcommand\Ql{\bar{\mathbf Q}_l}

\newcommand\BP{\mathbf P}

\newcommand\BF{\mathbf F}

\newcommand\BZ{\mathbf Z}

\newcommand\BH{\mathbf H} 
\newcommand\BI{\mathbf I}

\newcommand\Bm{\mathbf m}
\newcommand\Bn{\mathbf n}
\newcommand\Bb{\mathbf b}
\newcommand\Bc{\mathbf c}
\newcommand\Be{\mathbf e}

\newcommand\Bk{\mathbf k}

\newcommand\Bd{\mathbf d}

\newcommand\Bla{\boldsymbol\lambda}

\newcommand\Bmu{\boldsymbol\mu}

\newcommand\CH{\mathcal{H}}

\newcommand\CK{\mathcal{K}}

\newcommand\CG{\mathcal{G}}

\newcommand\CZ{ \mathcal{Z}}
\newcommand\CX{ \mathcal{X}}

\newcommand\FD{\mathfrak D}
\newcommand\FN{\mathfrak N}

\newcommand\Fb{\mathfrak b}
\newcommand\Fc{\mathfrak c}

\newcommand\Fg{\mathfrak g}

\newcommand\Fh{\mathfrak h}
\newcommand\Fs{\mathfrak s}
\newcommand\Fl{\mathfrak l}
\newcommand\Ft{\mathfrak t}

\newcommand\Fn{\mathfrak n}
\newcommand\Fp{\mathfrak p}
\newcommand\Fz{\mathfrak z}

\newcommand\iv{^{-1}}
\newcommand\wh{\widehat}
\newcommand\wt{\widetilde}
\newcommand\wg{^{\wedge}}

\newcommand\ol{\overline}

\newcommand\hra{\hookrightarrow}
\newcommand\lra{\leftrightarrow}

\newcommand\IC{\operatorname{IC}}

\newcommand\Hom{\operatorname{Hom}}
\newcommand\End{\operatorname{End}}

\newcommand\Ind{\operatorname{Ind}}

\newcommand\supp{\operatorname{supp}\,}
\newcommand\Lie{\operatorname{Lie}}

\newcommand\ch{\operatorname{ch}}

\newcommand\Ad{\operatorname{Ad}}
\newcommand\ad{\operatorname{ad}}

\newcommand\reg{_{\operatorname{reg}}}
\newcommand\rg{\operatorname{reg}}
\newcommand\sr{\operatorname{sr}}

\newcommand\nilp{\operatorname{nil}}
\newcommand\uni{_{\operatorname{uni}}}
\newcommand\nil{_{\operatorname{nil}}}
\newcommand\id{\operatorname{id}}

\newcommand\lp{\operatorname{\!\langle\!}}
\newcommand\rp{\operatorname{\!\rangle\!}}
\renewcommand\Im{\operatorname{Im}}

\newcommand\nat{^{\natural}}

\newcommand\flt{^{\flat}}

\newcommand\odd{\operatorname{odd}}

\newcommand\Diag{\operatorname{Diag}}

\newcommand\dw{\dot w}

\newcommand{\isom}{\,\raise2pt\hbox{$\underrightarrow{\sim}$}\,}
\numberwithin{equation}{section}

\newtheorem{thm}{Theorem}[section]
\newtheorem{lem}[thm]{Lemma}
\newtheorem{cor}[thm]{Corollary}
\newtheorem{prop}[thm]{Proposition}

\def \para#1{\par\medskip\textbf{#1}
              \addtocounter{thm}{1}}

\def \remark#1{\par\medskip\noindent
                \textbf{Remark #1}
                \addtocounter{thm}{1}}

\begin{document}
\setlength{\baselineskip}{4.9mm}
\setlength{\abovedisplayskip}{4.5mm}
\setlength{\belowdisplayskip}{4.5mm}

\renewcommand{\theenumi}{\roman{enumi}}
\renewcommand{\labelenumi}{(\theenumi)}
\renewcommand{\thefootnote}{\fnsymbol{footnote}}
\renewcommand{\thefootnote}{\fnsymbol{footnote}}
\allowdisplaybreaks[2]
\parindent=20pt
\medskip
\begin{center}
 {\bf Symmetric spaces associated to classical groups  \\
      with even characteristic }
\\
\vspace{1cm}
Junbin Dong, Toshiaki Shoji and Gao Yang
\\
\title{}
\end{center}

\begin{abstract}
Let $G = GL(V)$ for an $N$-dimensional vector space $V$ over 
an algebraically closed field $\Bk$, and $G^{\th}$ 
the fixed point subgroup of $G$ under an involution $\th$ on $G$.  
In the case where $G^{\th} = O(V)$, 
the generalized Springer correpsondence for the unipotent 
variety of the symmetric space $G/G^{\th}$ was described in [SY],  
assuming that $\ch \Bk \ne 2$.
The definition of $\th$ given there, and of the symmetric space arising from $\th$, 
make sense even if $\ch \Bk = 2$. In this paper, we discuss 
the Springer correspondence for those symmetric spaces with even characteristic.  
We show, if $N$ is even, that the Springer correspondence is 
reduced to that of symplectic Lie algebras in $\ch \Bk = 2$, which was determined 
by Xue.  
While if $N$ is odd, the number of $G^{\th}$-orbits in the unipotent variety 
is infinite, and 
a very similar phenomenon occurs as in the case of exotic symmetric 
space of higher level, namely of level $r = 3$. 
 
\end{abstract}

\maketitle
\markboth{DONG - SHOJI - YANG}{SYMMETRIC SPACES}
\pagestyle{myheadings}

\begin{center}
{\sc 0. Introduction}
\end{center}

\para{0.1}
Let $G$ be a connected reductive group over an algebraically closed field 
$\Bk$, and $\th : G \to G$ an involutive automorphism of $G$. 
Let $G^{\th} = \{g \in G \mid \th(g) = g\}$ be the fixed point subgroup of 
$G$ by $\th$. It is known by Vust [V] that $G^{\th}$ is a reductive group 
if $\ch \Bk \ne 2$.
Let $\Fg$ be the Lie algebra of $G$.  Then $\th$ induces a linear 
automorphism on $\Fg$, which we also denote by $\th$.  
Put $\Fg^{\th} = \{x \in \Fg \mid \th(x) = x\}$. 
It is known (e.g., [Spr, Th. 5.4.4]) that 
\par\medskip\noindent
(0.1.1) \ $\Lie G^{\th} = \Fg^{\th}$ 
if $\ch \Bk \ne 2$. 
\par\medskip\noindent
In general, we have $\Lie (G^{\th}) \subset \Fg^{\th}$, 
but the equality not necessarily holds if $\ch \Bk = 2$.    
\par
Let $G^{\io\th} = \{g \in G \mid \th(g) = g\iv\}$ be the set of 
$\io\th$-fixed points in $G$, where $\io : G \to G$ 
is the anti-automorphism $g \mapsto g\iv$. Let $G\uni$ be the set of 
unipotent elements in $G$, and put $G^{\io\th}\uni = G^{\io\th} \cap G\uni$. 
The conjugation action of $G^{\th}$ on $G$ preserves $G^{\io\th}$ and 
$G^{\io\th}\uni$.   
In the case where $\ch \Bk \ne 2$, it is known by Richardson [R] 
that $G^{\io\th}_0 = \{ g\th(g)\iv \mid g \in G \}$ coincides with 
the connected component of $G^{\io\th}$ containing the unit element, 
and  the map $g \mapsto g\th(g)\iv$ gives an isomorphism 
$G/G^{\th} \simeq G^{\io\th}_0$.  
Thus we can regard $G^{\io\th}$ as a symmetric space $G/G^{\th}$. 
\par
In the Lie algebra case, with $\ch \Bk \ne 2$, the symmetric space 
is defined as $\Fg^{-\th} = \{ x \in \Fg \mid \th(x) = -x\}$.  Let 
$\Fg\nil$ be the set of nilpotent elements in $\Fg$, and put 
$\Fg^{-\th}\nil = \Fg^{-\th} \cap \Fg\nil$.  It is known that 
$G^{\io\th}\uni \simeq \Fg^{-\th}\nil$, and the isomorphism is compatible 
with the action of $G^{\th}$.
It is also known by [R] that 
\par\medskip\noindent
(0.1.2) \ The number of $G^{\th}$-orbits in $G^{\io\th}\uni$
is finite if $\ch \Bk \ne 2$. 

\para{0.2.}
We consider $G = GL(V)$, where $V$ is an $N$-dimensional vector space 
over $\Bk$ with $\ch \Bk \ne 2$. Let $\th$ be the involutive automorphism such that 
$G^{\th}$ is the symplectic group $Sp(V)$ or the orthogonal group $O(V)$.
First assume that $G^{\th} = Sp(V)$, which we denote by 
$H$. In this case, we also consider the exotic symmetric space $G^{\io\th} \times V$, 
on which $H$ acts diagonally. It is known by Kato [K1] that 
$G^{\io\th}\uni \times V$ has finitely many $H$-orbits (actually he considered 
the Lie algebra case, $\Fg^{-\th}\nil \times V$, the so-called ``Kato's exotic 
nilcone'').  Put $N = 2n$, and let $W_n$ be the Weyl group of type $C_n$.  
He established in [K1] the Springer correspondence between the $H$-orbits
in $\Fg^{-\th}\nil \times V$ and irreducible representations of $W_n$, 
based on the Ginzburg theory ([CG]) on Hecke algebras.  
After that in [SS], the Springer correspondence 
was also proved for $G^{\io\th}\uni \times V$, based on the theory of the generalized 
Springer correspondence due to Lusztig [L1].  
(We remark that a straightforward 
generalization of the classical theory of the Springer correspondence does not hold
for the symmetric space 
$G^{\io\th}\uni$ or $\Fg^{-\th}\nil$ itself).  
\par
The special feature in the exotic nilcone is that for any 
$x \in \Fg^{-\th}\nil \times V$, the stabilizer of $x$ in $G^{\th}$ is connected, 
and so only the constant sheaves appear in the Springer correspondence. 
This type of phenomenon also appears in the nilpotent orbits of the Lie algebra 
$\Fs\Fp(V)$ in characteristic 2. Noting this, in [K2], Kato proved that 
the Springer correspondence for the exotic nilcone $\Fg^{-\th}\nil \times V$
(for any $\ch \Bk$)
is the same as the Springer correspondence for the ordinary nilcone 
$\Fs\Fp(V)\nil$ (for $\ch \Bk = 2$), by using a certain deformation argument of 
schemes over $\BZ$.  

\para{0.3.}
As a generalization of the exotic symmetric space $G^{\io\th} \times V$, we consider 
the variety $G^{\io\th} \times V^{r-1}$ for an integer $r \ge 1$, on which $H$ 
acts diagonally. We consider its unipotent part $\CX_r = G^{\io\th}\uni \times V^{r-1}$ 
with diagonal $H$-action. We call $\CX_r$ the (unipotent) exotic symmetric space of level $r$. 
The crucial difference from the exotic case (i.e., $r = 2$) 
is that the number of $H$-orbits in $\CX_r$
is infinite if $r \ge 3$. Hence the discussion in [SS] can not be applied to this case. 
Nevertheless, it was shown in [Sh] that a generalization of the Springer correspondence   
still holds in the following sense; we consider the complex reflection group 
$W_{n,r} = S_n\ltimes (\BZ/r\BZ)^n$ (hence $W_{n,1} \simeq S_n$ and 
$W_{n,2} \simeq W_n$). Then for each $\r \in W_{n,r}\wg$, one can construct
a smooth, irreducible, $H$-stable, locally closed subvariety $X_{\r}$ 
of $\CX_r$ and we have
a natural correspondence $\r \lra \IC(\ol X_{\r}, \Ql)$. Moreover, for an 
element $z \in X_{\r}$, one can define the Springer fibre $\SB_z$ as a closed subset of 
the flag variety of $H$, and the cohomology $H^{i}(\SB_z,\Ql)$ has a structure 
of $W_{n,r}$-module (Sprigner representation of $W_{n,r}$). Then for a generic 
$z \in X_{\r}$, the top cohomology 
$H^{2d_z}(\SB_z, \Ql)$ gives the 
irreducible representation $\r$ (here $d_z = \dim \SB_z$).  
In this way, any irreducible representation of $W_{n,r}$ is realized as 
the top cohomology of the Springer fibre. 

\para{0.4.}
Next we consider the case where $G^{\th} = O(V)$ with $\ch \Bk \ne 2$. 
We put $H = SO(V)$. In this case, the stabilizer of 
$x \in G^{\io\th}\uni$ in $H$ is not necessarily connected, and we need to
consider twisted local systems on $H$-orbits.  Moreover the Springer correspondence 
is not enough to cover all the $H$-orbits in $G^{\io\th}\uni$. 
In [SY], the generalized Springer correspondence for $G^{\io\th}\uni$ was established.
Actually it was shown there that Lusztig's theory of 
the generalized Springer correspondence for reductive groups fits very well to  
our setting under a suitable modification. For example, if $N = 2n + 1$ or $N = 2n$,  
the Springer correspondence is given by the correspondence 
$\r \lra \IC(\ol\SO_{\r}, \Ql)$, where $\r \in S_n\wg$ and 
$\SO_{\r}$ is a certain $H$-orbit corresponding to $\r$
(note that twisted local systems do not appear in this part).  

\para{0.5.}
In this paper, we consider $\th : G \to G$, defined in a similar way as 
$\th$ in 0.4, but under the condition that $\ch\Bk = 2$. 
The definition of $G^{\io\th}$ given in 0.4 makes sense even if $\ch\Bk = 2$
(see 1.2).
We call $G^{\io\th}$ the symmetric space 
over a field of characteristic 2. The aim of this paper is to establish the 
Springer correspondence for the $G^{\th}$-orbits in $G^{\io\th}\uni$.  
\par
First assume that $N = 2n$ is even.  In that case, we can show  (Proposition 1.7) 
 that $G^{\th} = Sp(V)$, 
the symplectic group in characteristic 2, and there exists an isomoprhism 
$G^{\io\th}\uni \simeq \Fg^{\th}\nil$, compatible with the action of $Sp(V)$,
where $\Fg^{\th} = \Fs\Fp(V)$ is the Lie algebra of $Sp(V)$.  
Hence considering the Springer correspondence for the symmetric space
$G^{\io\th}\uni$ is essentially the same as considering the same problem for 
the ordinary nilcone $\Fs\Fp(V)\nil$. 
Put $H = Sp(V)$ and $\Fh = \Fs\Fp(V)$.
As was explained in 0.2, the Springer correspondence 
for $\Fh\nil = \Fs\Fp(V)\nil$ was determined by [K2] in connection 
with the exotic symmetric space.  
After that Xue [X1, X2] established the Springer correspondence 
for the Lie algebras of classical type in characteristic 2, 
based on the Lusztig's theory of the generalized Springer correspondence.  
Note that the difficulty in the Lie algebras of even characteristic is that
the regular semisimple elements not necessarily exist.   
(Recall, in the case of reductive groups $G$ with Weyl group $W$, that  
the strategy of proving the Springer correspondence is first to construct 
a semisimple perverse sheaf $K$ on $G$, equipped with $W$-action, 
by making use of the finite Galois covering arising from the set of regular semisimple 
elements in $G$, then restrict it to the unipotent variety $G\uni$,  
to obtain the correspondence $\r \lra \IC(\ol\SO, \SE)$, 
where $\r$ is the irreducible representation of $W$, and $\SO$ is a certain 
unipotent class in $G$, $\SE$ is a local system on $\SO$.)
In our situation, $\Fh$ does not have regular semisimple elements.  
In order to overcome this defect, he replaced $H$ and $\Fh$ by a bigger group 
$\wt H$ and its Lie algebra $\wt\Fh$, where $\wt H$ is an extension of $H$ by 
a connected center, so that $\wt\Fh$ has regular semisimple elements. 
Then following the above procedure, he could prove the Springer correspondence for 
$\wt\Fh\nil = \Fh\nil$, in which case it gives a bijection  
$\r \lra \IC(\ol\SO_{\r}, \Ql)$ between irreducible representations of $W_n$, and 
nilpotent orbits in $\Fh\nil$.  
Hence our problem of the Springer correspondence for $G^{\io\th}\uni$ is easily 
solved as a corollary of Xue's result. 

\para{0.6.}
Next assume that $N = 2n + 1$ is odd. 
We write $G = GL(V')$ where $V'$ is an $N$-dimensional 
vector space.  Then we can show (Proposition 1.13) 
that $G^{\th} \simeq Sp(V)$, where $V$ is an 
$2n$-dimensional subspace of $V'$. Moreover $G^{\io\th}\uni \simeq \Fg^{\th}\nil$, 
compatible with the action of $G^{\th}$.
Put $H = Sp(V)$ and $\Fh = \Fs\Fp(V)$ as in 0.5.  
We can show that there is an embedding $\Fh \hra \Fg^{\th}$, 
and $\Fg^{\th}$ is isomorphic to $\Fh \times V \times \Bk$ as algebraic varieties, 
where the action of 
$H$ on $\Fg^{\th}$ corresponds to the diagonal action of $H$ on 
$\Fh \times V$, together with the trivial action of $H$ on $\Bk$. 
Hence the situation is very similar to the exotic symmetric 
space mentioned in 0.2.  But note that the structure of the nilcone  
is different.  In the exotic case, $\Fg^{-\th}\nil \times V$ is considered as 
the nilcone.  In our case, $\Fg^{\th}\nil = (\Fh\nil \times V) \cap \Fg\nil$, 
where $\Fg = \Fg\Fl(V')$ is the Lie algebra of $G$, which is a certain $H$-stable
subset of $\Fh\nil \times V$.  
\par
More interesting thing is that the number of $H$-orbits in $\Fg^{\th}\nil$ 
is infinite.  Hence this gives a counter-example for (0.1.2) in the case where 
$\ch\Bk = 2$. (Also this gives a counter-example for (0.1.1) since $G^{\th} = H$ 
and $\Fg^{\th} \ne \Fh$.) 
We remark that $\Fg^{\th}\nil$ should be understood as an analogue of 
the exotic symmetric space of higher level as discussed in 0.3 
rather than the exotic symmetric space itself.      
In fact, following Kato's observation ([K2]), the Springer correspondence 
for $\Fh\nil$ corresponds to that for the exotic 
symmetric space $G^{\io\th}\uni \times V$.  Hence it is natural to expect 
that the Springer correspondence 
for $\Fh\nil \times V$ should correspond to that for 
$(G^{\io\th}\uni \times V ) \times V = G^{\io\th}\uni \times V^2$, and 
the correspondence is dominated by $W_{n,3}$ (the complex reflection 
group for $r = 3$) as the special case of [Sh].
\par  
We show  in Theorem 9.7 and Proposition 9.10 that this certainly holds.  
It is proved that a similar result as in 0.3 holds for $\Fg^{\th}\nil$ with respect to 
$W_{n,3}$; namely, 
there exists a smooth, irreducible, locally closed subvariety $X_{\r}$ of $\Fg^{\th}\nil$
for each $\r \in W_{n,3}\wg$ such that $\r \lra \IC(\ol X_{\r}, \Ql)$ gives 
the Springer correspondence, namely, for $z \in X_{\r}$, the Springer 
fibre $\SB_z$ is defined as a closed subset of the flag variety of $H$, and 
$H^i(\SB_z,\Ql)$ gives rise to a $W_{n,3}$-module.  
For a generic $z \in X_{\r}$, 
$H^{2d_z}(\SB_z, \Ql)$ gives the irreducible representation $\r$ of $W_{n,3}$.  
Any irreducible representation of $W_{n,3}$ is realized in this way 
on the top cohomology of the Springer fibre. . 

\para{0.7.}
The proof of the Springer correspondence for $\Fg^{\th}\nil$ is basically done by 
modifying the arguments in [Sh] for the case of exotic symmetric spaces of higher level.
But in order to apply the discussion in [Sh] to our case, 
it is necessary to construct a representation of $W_n = W_{n,2}$ on a certain 
semisimple perverse sheaf $K$ on $\Fh$.  Note that a similar perverse sheaf $\wt K$ 
on $\wt\Fh$ equipped with $W_n$-action is already constructed in 0.5.   
Here we need to determine its restriction $\wt K|_{\Fh} = K$ on
$\Fh$. 
In order to do this, we construct a representation of $W_n$ on $K$ directly, 
without referring $\wt\Fh$, 
just using the subregular semisimple elements in $\Fh$
which are open dense in the set of semisimple elements in $\Fh$ 
(though in one step we need to apply the result for $\wt\Fh$).  This process is quite 
similar to the procedure in [SS], which is regarded as 
a reflection of Kato's deformation argument between 
the Springer correspondence for $G^{\io\th}\uni \times V$ and for $\Fh\nil$. 
\par
Note that in [Sh] (and in [SS]) the proof of the main result is 
reduced to the case where $r = 1$, namely the case $G^{\io\th}\uni$, 
by induction on $r$. But the discussion in [Sh] can not cover this case 
since the Springer correspondence 
does not exist for $G^{\io\th}\uni$,   
and we had to refer the result of 
Henderson [Hen], which he proved by using the Fourier-Deligne transform of perverse
sheaves on the Lie algebra. 
So, in some sense, the discussion in [Sh] is unsatisfactory 
from the view point of the strategy in 0.5.     
In the case of $\Fg^{\th}\nil$, this unpleasant situation does not occur
since the Springer correspondence exists for $\Fh\nil$ (though we cannot 
avoid to use $\wt\Fh$).  

\par\bigskip\noindent
{\bf Some notations:}
\par\medskip
For any finite group $\vG$, we denote by $\vG\wg$ the set of 
isomorphism  classes of irreducible representations of $\vG$ over $\Ql$.
\par
Let $G$ be an algebraic group and $\Fg$ its Lie algebra. $G$ acts on 
$\Fg$ by the adjoint action. We denote this action simply by 
$\Ad(g)x = g\cdot x$ 
for $g \in G, x \in \Fg$.   
\par\bigskip\bigskip\noindent
{\bf Contents}
\par\medskip\noindent
0.  \ Introduction \\
1.  \ Symmetric spaces in characteristic 2  \\
2.  \ Intersection cohomology related to $\Fs\Fp(V)_{\sr}$ \\
3.  \ Intersection cohomology on $\Fc\Fs\Fp(V)$  \\
4.  \ The variety of semisimple orbits  \\
5.  \ Intersection cohomology on $\Fs\Fp(V)$  \\
6.  \ Intersection cohomology on $\Fg^{\th}$  \\
7.  \ Nilpotent variety for $\Fg^{\th}$   \\
8.  \ Springer correspondence for $\Fg^{\th}$  \\
9.  \ Determination of the Springer correspondence  

\par\bigskip
\section{Symmetric spaces in characteristic 2}

\para{1.1.}
Let $V$ be an $N$-dimensional vector space over an algebraically closed field 
$\Bk$. 
We assume that $\ch \Bk  \ne 2$. 
Let $\lp \ ,\ \rp$ be the non-degenerate symmetric bilinear form on $V$. 
The orthogonal group $O(V)$ associated to the form $\lp \ ,\ \rp$ is 
defined as 
\begin{equation*}
\tag{1.1.1}
O(V) = \{ g \in GL(V) \mid \lp gv, gw\rp = \lp v,w\rp \text{ for any } v,w \in V \}.
\end{equation*}
If we define the quadratic from $Q(v)$ on $V$ by $Q(v) = \lp v, v\rp$, 
(1.1.1) is equivalent to 
\begin{equation*}
\tag{1.1.2}
O(V) = \{ g \in GL(V) \mid Q(gv) = Q(v) \text{ for any } v \in V\}.
\end{equation*} 
\par
We fix a basis $e_1, \dots, e_N$ of $V$, and identify $GL(V)$ with $G = GL_N$ by using 
this basis. If we define the matrix $J \in GL_N$ by $J = (\lp e_i, e_j\rp)$, 
(1.1.1) is also written as
\begin{equation*}
\tag{1.1.3}
O(V) =  \{ g \in G \mid {}^tgJg = J\}.
\end{equation*}
\par
We define a map 
$\th : G \to G$ by $\th(g) = J\iv ({}^tg\iv)J$.  Then $\th$ gives rise to 
an involutive automorphism on $G$, and we have $G^{\th} = O(V)$. 
\par
In the above discussion, if we replace the symmetric bilinear form 
by the non-degenerate skew-symmetric bilinear 
form $\lp \ ,\ \rp$ on $V$ with even $N$, we can define the symplectic group
$Sp(V)$ in a similar way as $O(V)$ by using (1.1.1). The matrix $J \in GL_N$ 
is defined similarly, and by using an involutive automorphism $\th : G \to G$ 
defined by $\th(g) = J\iv({}^tg\iv)J$, we obtain an analogue of (1.1.3) for 
$Sp(V)$. Hence in this case also $G^{\th} = Sp(V)$. 

\para{1.2.}
Hereafter, throughout the paper, we assume that $\ch\Bk = 2$.
Put $G = GL_N$. 
Consider an involutive automorphism  $\th : G \to G$ defined by
$\th(g) = J\iv({}^tg\iv)J$, where 
\begin{align*}
\tag{1.2.1}
J &= \begin{pmatrix}
          1 &    0   &  0  \\ 
          0  &   0   &  1_n  \\
          0  &  1_n  &  0
         \end{pmatrix}   \quad \text{ if } N = 2n+1,  \\  \\
\tag{1.2.2}
J &= \begin{pmatrix}
            0  &  1_n \\
            1_n  &  0     
      \end{pmatrix}  \qquad\quad \text{ if } N = 2n,  
 \end{align*}
with $1_n$ the identity matrix of degree $n$.
Thus we can consider the fixed point subgroup $G^{\th}$ and 
the symmetric space $G^{\io\th}$. 
If $\ch\Bk \ne 2$, then $G^{\th} \simeq O(V)$, and 
the generalized Springer correspondence 
with respect to $G^{\io\th}$ was discussed in [SY]. The aim of this 
paper is to consider a similar problem for $G^{\io\th}$ in the 
case where $\ch\Bk = 2$. 

\para{1.3.}
We consider $\th : G \to G$ as in 1.2.
Since $J = J\iv = {}^t\!J$, we have
\begin{align*}
\tag{1.3.1}
G^{\io\th} &= \{ g \in G \mid  J\iv({}^tg\iv)J = g\iv \} \\
           &= \{ g \in G \mid {}^t(Jg) = Jg \}. 
\end{align*}

Let $\Fg = \Fg\Fl_N$ be the Lie algebra of $G$, and $\th : \Fg \to \Fg$ 
be the linear automorphism induced from $\th : G \to G$. Since $\ch\Bk = 2$, 
$\th$ is given as $x \mapsto - J({}^tx)J = J({}^tx)J$ for $x \in \Fg$.  
Hence 
\begin{align*}
\tag{1.3.2}
\Fg^{\th} &= \{ x \in \Fg \mid J({}^tx)J = x \} \\
          &= \{ x \in \Fg \mid {}^t(Jx) = Jx\}.
\end{align*}

Put $\Fg^{\th}\nil = \Fg^{\th} \cap \Fg\nil$. By comparing (1.3.1) and (1.3.2),
we have

\begin{lem}  
The map $g \mapsto g-1$ gives an isomorphism $G^{\io\th}\uni \isom \Fg^{\th}\nil$, 
which is compatible with the conjugate action of $G^{\th}$.  
\end{lem}
\para{1.5.}
In order to study $G^{\th}$ and $G^{\io\th}$, we need to consider 
the orthogonal group over the field of characteristic 2. 
Since (1.1.1) and (1.1.2) 
are not equivalent in the case where $\ch \Bk = 2$, 
we have to define the orthogonal group by using 
the quadratic form $Q(v)$ on $V$. Let $Q(v)$ be a quadratic form on $V$.
We define an associated bilinear form $\lp \ , \ \rp$ by 
\begin{equation*}
\tag{1.5.1}
\lp v, w \rp = Q(v + w) - Q(v) - Q(w).
\end{equation*}  
\par
Here we give the quadratic form $Q(v)$ explicitly as follows.
First assume that $N = 2n$, and fix a basis $e_1, \dots, e_n, f_1, \dots f_n$
of $V$. For $v = \sum_ix_ie_i + \sum_iy_if_i$, define 
\begin{equation*}
Q(v) = \sum_{i=1}^nx_iy_i.
\end{equation*} 
Then the associated bilinear form  is given by 
$\lp v, w\rp = \sum_i(x_iy_i' + x_iy_i')$ 
for $v = \sum_ix_ie_i + \sum_iy_ie_i, w = \sum_ix_i'e_i + \sum_iy_i'f_i$.   
Next assume that $N = 2n + 1$, and fix a basis 
$e_0, e_1, \dots, e_n, f_1, \dots, f_n$ of $V$. 
For $v = \sum_ix_ie_i + \sum_iy_if_i$, define
\begin{equation*}
Q(v) = x_0^2 + \sum_{i=1}^nx_iy_i.
\end{equation*} 
Then the associated bilinear form is given by 
$\lp v, w\rp = \sum_{i\ge 1}(x_iy_i' + x_i'y_i)$
for $v = \sum_{i \ge 0}x_ie_i + \sum_{i \ge 1}y_if_i$, 
$w = \sum_{i \ge 0}x_i'e_i + \sum_{i \ge 1}y_i'f_i$.  
\par
We define an orthogonal group $O(V)$ as in (1.1.1). 
If $g \in O(V)$, $g$ leaves the form $\lp\ ,\ \rp$ invariant
by (1.5.1). 
It follows that
\begin{equation*}
\tag{1.5.2}
O(V) \subset \{ g \in GL(V) \mid \lp gv, gw\rp = \lp v,w \rp 
                   \text{ for any $v,w \in V$} \}. 
\end{equation*}

\para{1.6.}
Assume that $N = 2n$.  Since $\ch\Bk = 2$, the symmetric bilinear form 
on $V$ coincides with the skew-symmetric bilinear form on $V$. 
The definition of the symplectic group given in 1.1 makes sense even if 
$\ch\Bk = 2$, which we denote by $Sp(V)$. 
Thus the right hand side of (1.5.2) coincides with 
$Sp(V)$ with respect to this form, and we have 
\begin{equation*}
\tag{1.6.1}
O(V) \subset Sp(V).
\end{equation*}
By using the explicit description 
of the associated symmetric bilinear form $\lp \ ,\ \rp$ on $V$ 
given in 1.5, we see 
that $Sp(V)$ can be written as
\begin{equation*}
\tag{1.6.2}
Sp(V) = \{ g \in G \mid {}^tgJg = J \},
\end{equation*}  
where $J$ is as in (1.2.2). 
In particular, we have $G^{\th} = Sp(V)$. 
\par
Let $\Fs\Fp(V)$ be the Lie algebra of $Sp(V)$. It follows from (1.6.2), we have
\begin{align*}
\Fs\Fp(V) &\subseteq  \{ x \in \Fg\Fl(V) \mid \lp xv,w \rp + \lp v, xw \rp = 0 
                   \text{ for any } v,w \in V \} \\ 
          &=  \{ x \in \Fg\Fl(V) \mid {}^txJ + Jx = 0\} \\
          &=\Fg^{\th}.
\end{align*}
The last equality follows from (1.3.2). 
Here $\dim \Fs\Fp(V) = \dim Sp(V) = 2n^2 + n$.  The dimension of the 
space of symmetric matrices in $V$ is equal to $n(2n + 1)$. Thus 
the equality holds in the above formulas.  We have
\begin{equation*}
\tag{1.6.3}
\Fs\Fp(V) = \Fg^{\th}.
\end{equation*}
\par
Summing up the above arguments, together with Lemma 1.5, we have the following. 

\begin{prop} 
Assume that $N = 2n$. 
Then $G^{\th} = Sp(V)$,  and 
$G^{\io\th}\uni \simeq \Fg^{\th}\nil = \Fs\Fp(V)\nil$. 
The behaviour of $G^{\th}$-orbits on  $G^{\io\th}\uni$ is 
the same as that of $Sp(V)$-orbits on $\Fs\Fp(V)\nil$.  
\end{prop}

\remark{1.8.}
Let $W_n$ be the Weyl group of 
type $C_n$.  It is known by Hesselink [Hes] and Spaltenstein [Spa] 
that the number of $Sp(V)$-orbits in $\Fs\Fp(V)\nil$ is 
finite, and those orbits are parametrized by $W_n\wg$. 

\remark{1.9.} 
In the case where $\ch\Bk = 2$ and $N$ is even, it is not known 
whether $O(V)$ is realized as $G^{\th}$ for a certain involutive 
automorphism  $\th : G \to G$. 

\para{1.10.}
Assume that $N = 2n+1$.  
By changing the notation, we consider the vector space 
$V'$ with basis $e_0, e_1, \dots, e_n, f_1, \dots, f_n$, and
let $V$ be the subspace of $V'$ spanned by $e_1, \dots, e_n, f_1, \dots, f_n$.
We identify $Sp(V)$ with $Sp_{2n}$ by using this basis. 
$Sp_{2n}$ can be explicitly written as 
\begin{equation*}
\tag{1.10.1}
Sp_{2n} = \biggl\{ \begin{pmatrix}
                f  &  g  \\
                h  &  k
              \end{pmatrix} \in GL_{2n} 
                \big |\  f, g, h, k \in \SM_n,  \text{ (*) } \biggr\}, 
\end{equation*}
where $\SM_n$ is the set of square matrices of degree $n$, 
and the condition (*) is given by 
\begin{equation*}
\tag{1.10.2}
{}^thf = {}^tfh, \quad {}^tkg = {}^tgk, \quad {}^thg + {}^tfk = 1.
\end{equation*}
 
We have the following result.
\begin{prop} 
Assume that $N = 2n + 1$, and $G = GL_N$.  Then 
\begin{equation*}
\tag{1.11.1}
G^{\th} = \biggl\{ \begin{pmatrix}
                1  &   0  \\
                0  &   y
              \end{pmatrix} \big | \ y \in Sp_{2n} \biggr\}. 
\end{equation*}
In particular, $G^{\th} = SO(V') \simeq Sp(V)$. 
\end{prop}

\begin{proof}
Take $x \in GL_N$, and write it as 
\begin{equation*}
x = \begin{pmatrix}
        a  &  \Bb  &  \Bc   \\
        {}^t\Bd  &   f  &  g  \\
        {}^t\Be  &   h  &  k  
     \end{pmatrix}, 
\end{equation*}
where $\Bb = (b_1, \dots, b_n), \Bc = (c_1, \dots, c_n)$ and 
$\Bd = (d_1, \dots, d_n), \Be = (e_1, \dots, e_n)$, 
together with $f,g,h,k \in \SM_n$. 
Note that $G^{\th} = \{ x \in G \mid {}^txJx =J\}$.  The condition 
${}^txJx = J$ implies, in particular,  that 
\begin{align*}
a^2 + \Be\cdot {}^t\Bd + \Bd\cdot {}^t\Be &= 1, \\
{}^t\Bb \cdot \Bb + {}^thf + {}^tfh &= 0, \\
{}^t\Bc\cdot \Bc + {}^tkg + {}^tgk &= 0.
\end{align*}
Since $\Be\cdot{}^t\Bd + \Bd\cdot{}^t\Be = 2\sum_{i=1}^nd_ie_i = 0$, 
the first equality implies that $a = 1$. 
Since the diagonal entries of ${}^thf + {}^tfh$ are all zero, 
the $ii$-entry of the matrix ${}^t\Bb\cdot\Bb + {}^thf + {}^tfh$ is equal to
$b_i^2$. Hence $\Bb = ${ \bf 0} by the second equality. Similar argument by using 
the third shows that $\Bc = 0$.   
Now the condition ${}^txJx = J$ is equivalent to the condition that 
\begin{equation*}
\tag{1.11.2}
\begin{cases}
\Be f + \Bd h = 0, \\
\Be g + \Bd k = 0, \\
{}^thg + {}^tfk = 1, \\
{}^thf  = {}^tfh, \\
{}^tkg = {}^tgk.
\end{cases}
\end{equation*}
By comparing (1.11.2) with (1.10.2), we can write as 
\begin{equation*}
x = \begin{pmatrix}
        1        &   0  &  0   \\
        {}^t\Bd  &   f  &  g  \\
        {}^t\Be  &   h  &  k  
     \end{pmatrix} \quad\text{ with }
       y = \begin{pmatrix}
            f  &  g  \\
            h  &  k
       \end{pmatrix} \in Sp_{2n}. 
\end{equation*}
Since $y$ is non-degenerate, the relation 
$(\Be,\Bd)y = 0$ 
implies that $\Bd = \Be = 0$.
This proves (1.11.1). 
Now by 1.5, one can check that $G^{\th} \subset O(V')$. 
We have $\dim G^{\th} = \dim Sp_{2n} = \dim SO_{2n+1} = 2n^2 + n$.
Since $Sp(V)$ is connected, we conclude that $G^{\th} = SO(V')$. 
The proposition is proved. 
\end{proof}

\para{1.12.}
We determine $\Fg^{\th}$ in the case where $N = 2n+1$. 
By (1.3.2), we have 
$\dim \Fg^{\th} = (n + 1)(2n + 1)$. 
Since $\dim G^{\th} = \dim Sp_{2n} = 2n^2 + n$, we see that 
\par\medskip
\noindent
(1.12.1) \ $\Lie G^{\th} \subsetneq \Fg^{\th}$, namely (0.1.1) does not 
hold for $G^{\th}$.
\par\medskip
More precisely, by the direct computation, we have the following. 
\begin{align*}
\tag{1.12.2}
\Fg^{\th} &= \left\{ x = \begin{pmatrix}
                        a   &   \Bb   &   \Bc \\
                        {}^t\Bc   &   f  &   g   \\
                        {}^t\Bb   &   h  &   {}^tf 
                     \end{pmatrix} \bigg | \ a \in \Bk, f,g,h \in \SM_n, 
                             {}^th = h, {}^t g = g \right\}, \\ 
\Lie G^{\th} &= \{ x \in \Fg^{\th} \mid a = \Bb = \Bc = 0 \} \simeq \Fs\Fp(V). 
\end{align*}   

Summing up the above arguments, together with Lemma 1.4, we have the following.
\begin{prop}  
Assume that $N = 2n+1$.  Then $G^{\th} \simeq Sp(V)$, and 
$G^{\io\th} \simeq \Fg^{\th}\nil$. 
Under the embedding 
$\Fs\Fp(V)\nil \hra \Fg^{\th}\nil$, 
$\Fs\Fp(V)\nil$ is a $G^{\th}$-stable subset of $\Fg^{\th}\nil$, and 
the action of $G^{\th}$ on $\Fs\Fp(V)\nil$ 
coincides with the conjugate action of $Sp(V)$ on it. 
\end{prop}

\para{1.14.}
We write $H = Sp(V)$ and $\Fh = \Fs\Fp(V)$.  We identify $H$ with $G^{\th}$, and 
$\Fh$ with a subspace of $\Fg^{\th}$.  Then $\Fg^{\th}$ can be written as 
$\Fg^{\th} = \Fh \oplus \Fg_{V'}$, where $\Fg_{V'}$ is a subspace of $\Fg^{\th}$ 
consisting of $x \in \Fg^{\th}$ such that $f = g = h = 0$ (notation in (1.13.2)). 
We express $x \in \Fg_{V'}$ as $x = (a, \Bb, \Bc)$.  
Put $\Fz = \{ (a, 0, 0) \mid a \in \Bk \}$, and 
let $\Fg_V$ be the subspace of $\Fg_{V'}$ consisting of $x = (0, \Bb, \Bc)$. 
Then $\Fg_{V'} = \Fg_{V} \oplus \Fz$. $\Fg_V$ is $H$-stable, and $H$ acts trivially 
on $\Fz$. 
We identify $\Fg_{V}$ with $V$ under the correspondence 
$(0,\Bb,\Bc) \in \Fg_{V} \lra \sum_{i=1}^nc_ie_i + \sum_{i=1}^nb_if_i \in V$
for $\Bb = (b_1, \dots, b_n), \Bc = (c_1, \dots, c_n)$. 
Then we can identify $\Fh \oplus \Fg_{V}$ with $\Fh \times V$, where the natural 
action of $H$ on $\Fg^{\th}$ preserves $\Fh \oplus \Fg_{V}$, which corresponds to 
the diagonal action of $H$ on $\Fh \times V$.  

\remark{1.15.}
The above discussion shows that considering the action of $H$ on $\Fg^{\th}$ 
is the same as considering the diagonal action of $H$ on $\Fh \times V$. 
Hence the situation is quite similar to the case of exotic symmetric spaces 
studied in [K1], [SS].  Recall that the exotic symmetric space (in the Lie 
algebra case) is defined as $\Fg^{-\th} \times V$ for $G^{\th} = Sp(V)$
with $\ch\Bk \ne 2$, together with the diagonal action of $G^{\th}$. 
But note that the structure of the nilpotent variety is different.  In the
exotic case, as the nilpotent variety, we have considered 
$\Fg^{-\th}\nil \times V$ (Kato's exotic 
nilpotent cone [K1]).  In the present case, we consider  
$\Fg^{\th}\nil = (\Fh\oplus \Fg_{V}) \cap \Fg\nil$, which 
 corresponds to a certain subset of $\Fh\nil \times V$.   

\para{1.16.} 
We follow the notation in 1.14.   In the remainder of this section,
we shall show that $\Fg^{\th}\nil$ has infinitely many $H$-orbits. 
Let $B'$ be the subgroup of $G = GL_N$ consisting of elements 
\begin{equation*}
\tag{1.16.1}
        \begin{pmatrix}
               a        &   0  &   \Bc  \\
               {}^t\Bd        &   f    &    g   \\
               0        &   0    &    k,
             \end{pmatrix} 
\end{equation*}
where $a \in \Bk, \Bc, \Bd \in \Bk^n$ (row vectors), $f, g, k \in \SM_n$, and 
$f$ is upper triangular, $k$ is lower triangular. 
Then $B'$ is a $\th$-stable Borel subgroup of $G$.  
Put $\Fb' = \Lie B'$, and let $\Fn'$ be the nilpotent radical of $\Fb'$. 
Then $\Fg^{\th} \cap \Fn' \subset \Fg^{\th}\nil$, and we have  
\begin{equation*}
\tag{1.16.2}
\Fg^{\th} \cap \Fn' = \biggl\{ \begin{pmatrix}
       0  &  0  &  \Bc  \\
 {}^t\Bc  &   f  &  g  \\
      0   &   0  &  {}^tf
     \end{pmatrix} \mid f \text{ : strict upper triangular, } {}^tg = g \biggr\}.
\end{equation*}
The following result gives a counter-example for (0.1.2) in the case 
where $\ch\Bk = 2$. 

\begin{prop}  
Let $\OO_0$ be the regular nilpotent 
orbit in $\Fg\Fl_N$.  Then $\OO_0 \cap \Fg^{\th}$ has infinitely 
many $H$-orbits. In particular, $\Fg^{\th}\nil$ has infinitely many $H$-orbits. 
\end{prop}
\begin{proof}
Assume that $f \in \SM_n$ corresponds to the 
transformation on the subspace $\lp e_1, \dots, e_n \rp$ of $V$ defined by  
\begin{equation*}
f : e_n \mapsto e_{n-1} \mapsto \cdots \mapsto e_1 \mapsto 0,
\end{equation*}
and put $g = \Diag(0, \dots, 0,1) \in \SM_n$. 
Then $y = \begin{pmatrix}
               f  &   g  \\
               0  &   {}^tf
      \end{pmatrix}$ 
gives an element in $\Fh = \Fs\Fp_{2n}$, 
which is a regular nilpotent element in $\Fs\Fp_{2n}$. 
Let $\Bc = (0, 0, \dots, 0, \xi)$ with $\xi \in \Bk$, and put 
$z =  (0, 0, \Bc) \in \Fg_V$. 
Then $x = y + z \in \Fg^{\th}\nil$ by (1.16.2), which we denote by $x(\xi)$. 
Since the operation of $x(\xi)$ on $V'$ is given by 
\begin{equation*}
\begin{cases}
f_1 \mapsto f_2 \mapsto \cdots \mapsto f_{n-1} \mapsto f_n,  \\
f_n \mapsto e_n + \xi e_0, \\
e_0 \mapsto \xi e_n, \\
e_n \mapsto e_{n-1} \mapsto \cdots \mapsto e_1 \mapsto 0,
\end{cases}  
\end{equation*}
$x(\xi) \in \OO_0$ for any $\xi \in \Bk^*$.  
In order to prove the proposition, it is enough to see that 
\par\medskip\noindent
(1.17.1) \ $x(\xi)$ and $x(\xi')$ are not conjugate under $H$ 
if $\xi \ne \xi'$.
\par\medskip
We show (1.17.1).
Assume there exists $h \in H$ such that $h(x(\xi)) = x(\xi')$. 
Write $x(\xi) = y + z, x(\xi') = y + z'$. Since $H$ leaves the decomposition 
$\Fg^{\th} = \Fh \oplus \Fg_{V'}$ invariant, 
we must have $h(y) = y$ and $h(z) = z'$.  Hence $h \in Z_H(y)$. 
Since $y \in \Fg^{\th} \cap \Fn'$ is regular nilplotent, $Z_H(y)$ satisfies the property
\begin{equation*}
\tag{1.17.2}
Z_H(y) \subset \left\{ \begin{pmatrix}
                    f_1   &    g_1   \\
                    0   &    {}^tf_1\iv
                \end{pmatrix}
                  \in Sp_{2n} \mid f_1 : \text{ upper unitriangular }\right\}.
\end{equation*}
On the other hand, the action of $H$ on $\Fg_V$ is given 
as in 1.14.  By (1.17.2), we have 
\begin{equation*}
h\cdot {}^t(\underbrace{0, \dots, 0, \xi}_{\text{$n$-times}}, 
\underbrace{0, 0, \dots, 0}_{\text{$n$-times}}) = 
{}^t(\underbrace{c_1, \dots, c_n}_{\text{$n$-times}}, 
          \underbrace{0, 0, \dots, 0}_{\text{$n$-times}})
\end{equation*}
for some $c_1, \dots, c_n$ with $c_n = \xi$. 
Since $h(z) = z'$, we have $\xi = \xi'$. Thus (1.17.1) holds.
The proposition is proved.
\end{proof}

\remark{1.18.}
In later discussion, we use the Jordan decomposition of Lie algebras.
It is known (see [Spr, 4.4.20]) that if $\Fg$ is the Lie algebra of 
an algebraic group $G$, the Jordan decomposition works.  So, 
in the setting of 1.15, we can apply the Jordan decomposition 
for $\Fh = \Lie H$, but not for $\Fg^{\th}$.

\par\bigskip
\section{ Intersection cohomology related to $\Fs\Fp(V)_{\sr}$ }

\para{2.1.}
Let $H = Sp(V), \Fh = \Fs\Fp(V)$, and we follow the notation in 1.6.
Let $W_n$ be the Weyl group of $H$.  As pointed out in Remark 1.8, nilpotent 
orbits in $\Fh$ are parametrized by $W_n\wg$.  The Springer correspondence 
between the set of nilpotent orbits and $W_n\wg$ was first considered by Kato [K2]. 
Then Xue [X1], [X2] established the general theory of 
the Springer correspondence for classical 
Lie algebras in characteristic 2.
\par
The main difficulty in considering $H$ and $\Fh$ relies on the fact 
that the regular semisimple elements do not exist for $\Fh$. 
In order to overcome this defect, Xue replaced $H$ and $\Fh$ 
by bigger ones $\wt H$ and 
$\wt\Fh$, extension by connected center, so that $\wt\Fh$ has regular semisimple 
elements, and established the Springer correspondence by using the bijection
$\wt\Fh\nil \simeq \Fh\nil$. 
(Actually Xue considers the adjoint group $\wt H_{\ad}$ and its Lie algebra 
$\wt\Fh_{\ad}$, but the theory of the Springer correspondence for them is 
essentially the same as that for $\wt H$ and $\wt\Fh$.)
\par
In this paper, for later applications to the case where $N = 2n+1$, 
we discuss the Springer correspondence
for $\Fh$ more directly, without using the regular semisimple elements
(though we need to use $\wt\Fh$). 
In the discussion below, we borrowed the idea to use $\FD$ from [SY].
Note that those discussions have strong resemblance with the case of 
exotic symmetric spaces associated to symplectic groups with $\ch\Bk \ne 2$
([SS, 3]), as explained in the Introduction. 

\para{2.2.}
Let $T \subset B$ be subgroups of $G = GL_N$ with $N = 2n$ given by 
\begin{align*}
B &= \biggl\{ \begin{pmatrix}
               a  &  b  \\
               0  &  {}^ta\iv
       \end{pmatrix} \mid a, b \in \SM_n, 
               a \text{ : upper triangular, } {}^tb = a\iv b\,{}^t\!a \biggr\}, \\ 
T &= \{ \Diag(t_1, \dots, t_n, t_1\iv, \dots, t_n\iv) \mid t_i \in \Bk^*\}.
\end{align*} 
By (1.10.1), $B$ is a Borel subgroup of $H$ and $T$ is a maximal torus 
of $H$ contained in $B$. 
Let $\Ft$ be the Lie algebra of $T$.  Since $\ch\Bk = 2$, we have
\begin{equation*}
\tag{2.2.1}
\Ft = \{ \Diag(t_1, \dots, t_n, t_1, \dots, t_n) \mid t_i \in \Bk \}.
\end{equation*} 
We define a subset $\Ft_{\sr}$ of $\Ft$ by 
\begin{equation*}
\tag{2.2.2}
\Ft_{\sr} = \{ s = \Diag(t_1, \dots, t_n, t_1, \dots, t_n) \mid t_i \ne t_j
                            \text{ for } i \ne j \}. 
\end{equation*}
$\Ft_{\sr}$ is open dense in $\Ft$, and for any $s \in \Ft_{\sr}$, 
$Z_H(s) \simeq SL_2 \times \cdots \times SL_2$ ($n$-times).
Put $\Fh_{\sr} = \bigcup_{g \in H}g(\Ft_{\sr})$. 
Then $\Fh_{\sr}$ is the set of semisimple elements in $\Fh$ such that 
all the eigenspaces have dimension 2. 
\par
Recall that $s \in \Fh$ is called regular semisimple if $Z_H^0(s)$ is a maximal 
torus of $H$. For any $s \in \Ft$, $\dim Z_H(s) \ge 3n$. 
Hence $\Ft$ does not contain regular semisimple elements. 
This implies that $\Fh$ does not contain regular semisimple elements 
(see Lemma 2.3).  
An element $s \in \Fh_{\sr}$ is said to be subregular semisimple. 
\par
Let $U$ be the unipotent radical of $B$. 
Let $\Fb$ be the Lie algebra of $B$, and $\Fn = \Lie U$ the nilpotent radical of $\Fb$. 
We have $\Fb = \Ft \oplus \Fn$. $\Fn$ is written as 
\begin{equation*}
\tag{2.2.3}
\Fn = \biggl\{ \begin{pmatrix}
                a  &  b  \\
                0  &  {}^ta 
         \end{pmatrix} \mid a \text{ : strict upper triangular, } {}^tb = b \biggr\}
\end{equation*}
\par
Let $\Phi^+ \subset \Hom (T, \Bk^*) \simeq \BZ^n$ be the set of positive roots  
of type $C_n$, 
\begin{equation*}
\Phi^+ = \{ \ve_i - \ve_j, \ve_i + \ve_j \quad (1\le i < j \le n), 
            2\ve_i \quad (1 \le i \le n) \}
\end{equation*}
where $\ve_1, \dots,\ve_n$ is a basis of $\Hom(T, \Bk^*)$ given by  
$\ve_i : \Diag(t_1, \dots,t_n, t_1\iv, \dots, t_n\iv) \mapsto t_i$. 
The set of positive long (resp. short) roots $\Phi^+_{l}$, 
(resp, $\Phi^+_s$)  is given as
\begin{align*}
\Phi^+_l &= \{ 2\ve_i \mid 1 \le i \le n \}, \\
\Phi^+_s &= \{ \ve_i - \ve_j, \ve_i + \ve_j \mid 1 \le i < j \le n\}.
\end{align*}
\par\noindent
The root space decomposition of $\Fn$ with respect to $T$ is given as 
\begin{equation*}
\Fn = \bigoplus_{\a \in \Phi^+}\Fg_{\a}, 
\end{equation*}
where $\Fg_{\a} = \{ x \in \Fn \mid s\cdot x = \a(s)x \text{ for any } s \in T\}$. 
Let $d\a \in \Hom (\Ft, \Bk)$ be the differential of $\a \in \Hom (T,\Bk^*)$. 
Then $\Ft$ acts on $\Fg_{\a}$ by $[s,x] = d\a(s)x$ for $s \in \Ft$. 
Since $\ch\Bk = 2$, the weight space decomposition of $\Fn$ with respect to $\Ft$ is 
given as 
\begin{equation*}
\Fn = \FD \oplus \bigoplus_{\a \in \Phi^+_{s}}\Fg_{\a} = \FD \oplus \Fn_{s},
\end{equation*}
where $\Fn_s = \bigoplus_{\a \in \Phi^+_s}\Fg_{\a}$, and 
$\FD = \bigoplus_{\a \in \Phi^+_{l}}\Fg_{\a}$ is the weight space of weight 0.  
Explicitly, $\FD$ is given as follows;
\begin{equation*}
\tag{2.2.4}
\FD = \biggl\{ \begin{pmatrix}
                 0  &  b  \\
                 0  &  0
               \end{pmatrix} \mid b \text{ : diagonal }\biggr\}.
\end{equation*} 
In particular, we have
\begin{equation*}
\tag{2.2.5}
[\Ft, \FD] = 0.
\end{equation*}
According to (2.2.4), we express an element in $\FD \simeq \Bk^n$ as
$\Bb = (b_1, \dots, b_n)\in \FD$ for $b = \Diag(b_1, \dots, b_n)$. 
For $k = 0, \dots, n$, put 
$\FD_k = \{ \Bb  \in \FD \mid b_i = 0 \text{ for } i > k \}$.
Then $\FD_k$ is a $T$-stable subspace of $\FD$. 
We also put $\FD^0_k = \{ \Bb \in \FD_k \mid b_i \ne 0 \text{ for } 1 \le i \le k \}$, 
which is an open dense subset of $\FD_k$.  
\par
We need a lemma.

\begin{lem}  
\begin{enumerate}
\item 
Assume that $x \in \Fh$ is semisimple.   Then there exists 
$g \in H$ such that $gx \in \Ft$.
\item
Assume that $x \in \Fb$ is semisimple.  Then 
there exists $g \in B$ such that $gx \in \Ft$.
\end{enumerate}
\end{lem}
\begin{proof}
Assume that $x \in \Fh$ is semisimple.  Consider the eigenspace decomposition 
$V = W_1 \oplus \cdots \oplus W_a$ of $x$. Then $W_1, \dots, W_a$ are 
mutually orthogonal with respect to the form $\lp\ ,\ \rp$. Then 
the restriction of the form on $W_i$ gives a non-degenerate skew-symmetric
bilinear form.  In particular, $\dim W_i$ is even. We can find a basis 
$e_1', \dots, e_n', f_1', \dots, f_n'$ of $V$ such that 
$\{ e_j', f_j'\mid j \in I_i\}$ gives a symplectic basis of $W_i$, 
where $[1,n] = \coprod_{1 \le i \le a}I_i$ is a partition of $[1,n]$. 
We define a map $g : V \to V$ by $g(e_j) = e_j', g(f_j) = f_j'$ 
for each $j$.  Then $g \in H$, and $x' = g\iv x \in \Ft$. This proves (i).
\par  
Next assume that $x \in \Fb$ is semisimple. 
Let $s \in \Ft$ be the projection of $x \in \Fb$. 
We consider the eigenspace decomposition of $s$ on $V$, 
$V = V_1\oplus\cdots\oplus V_a$, where $V_i$ is the eigenspace 
of $s$ with respect to the eigenvalue $\mu_i$.  This defines 
a partition $[1,n] = \coprod_iI_i$ such that 
$\{ e_j, f_j \mid j \in I_i\}$ gives a symplectic basis of $V_i$. 
Since $x$ is semisimple, and has the same eigenvalues 
$\mu_1, \dots, \mu_a$, we have a decomposition 
$V = W_1 \oplus\cdots \oplus W_a$ into eigenspaces of $x$, where 
$W_i$ is the eigenspace with respect to the eigenvalue $\mu_i$.   
Here $\dim W_i = \dim V_i$.
As before, we can find a symplectic basis 
$e_1', \dots, e_n', f_1', \dots, f_n'$ of $V$, 
and $g \in H$ associated to this basis such that $x' = g\iv x \in \Ft$. 
We show that there exists a choice of $e_1', \dots, e_n', f_1', \dots, f_n'$ such that 
$g \in B$, i.e., the choice such that the subspace spanned by $e_1, \dots, e_k$ 
coincides with that by  
$e_1', \dots, e_k'$ for $k = 1, \dots, n$. 
Since $e_1$ is an eigenvector for $x$, we can put $e_1' = e_1$.
We may assume that $e_1' \in W_1$. 
Let $\ol V = \lp e_1\rp^{\perp}/\lp e_1 \rp$.  Then $\ol V$ has a natural 
symplectic basis $\ol e_j, \ol f_j$ $(2 \le j \le n)$, 
and $x$ induces $\ol x \in \Fs\Fp(\ol V)$.    
By induction on $n$, one can find the required basis $\ol e'_j, \ol f'_j$ $(2 \le j \le n)$
of $\ol V$. This produces vectors $e'_1, \dots, e'_n, f'_2, \dots, f'_n$,
and finally we can choose $f'_1 \in W_1$ by the condition that $\lp e'_1, f'_1\rp = 1$
and $f'_1$ is orthogonal for all other vectors $e'_j, f_j'$. Thus we obtain 
the basis $e_1', \dots, e_n', f_1', \dots, f_n'$ as required, and $g \in B$. 
(ii) holds. The lemma is proved.    
\end{proof}

\para{2.4.}
We consider the varieties
\begin{align*}
\wt X &= \{ (x, gB) \in \Fh \times H/B \mid g\iv x \in \Fb\},  \\
    X &= \bigcup_{g \in H}g(\Fb),
\end{align*}
and define a map $\pi : \wt X \to X$ by $(x, gB) \mapsto x$. 
$\pi$ is a proper map onto $X$, and so $X$ is irreducible, 
closed in $\Fh$. (Later in Lemma 2.9, it will be shown that $X = \Fh$.) 
\par
For $0 \le k \le n$, 
put $\FN_{k,\sr} = \bigcup_{g \in B}g(\Ft_{\sr} + \FD_k)$. 
We define varieties
\begin{align*}
\wt Y_k &= \{ (x, gB) \in \Fh \times H/B
                        \mid g\iv x \in \FN_{k,\sr}  \}, \\
    Y_k &= \bigcup_{g \in H}g(\FN_{k,\sr}) = \bigcup_{g \in H}g(\Ft_{\sr} + \FD_k), 
\end{align*}
and define a map $\psi^{(k)} : \wt Y_k \to Y_k$ by 
$(x,gB) \to x$. 
In the case where $k = n$, we write $\wt Y_k, Y_k$ 
and $\psi^{(k)}$ simply by $\wt Y, Y$ and $\psi$. 
We have a lemma.
\begin{lem}  
$\wt Y = \pi\iv(Y)$, and $\psi$ coincides with the restriction of 
$\pi$ on $\pi\iv(Y)$. In particular, $\psi: \wt Y \to Y$ is a proper surjective
map.
\end{lem}
\begin{proof}
It is enough to show that $Y \cap \Fb = \FN_{n,\sr}$.
Assume that $y \in Y \cap \Fb$. 
Since $y \in Y$, there exists $g \in H, s \in \Ft_{sr}, z \in \FD$ 
such that $y = g(s) + g(z) \in \Fb$, where $g(s)$: semisimple, $g(z)$: nilpotent. 
Moreover since $[s, z] = 0$ by (2.2.5), we have 
$[g(s), g(z)] = 0$.  
By Lemma 2.3, replacing $y$ by its $B$-conjugate, we may assume that 
$g(s) \in \Ft_{sr}, g(z) \in \Fn$. 
Then there exists
$\dw \in N_H(T)$ with $w \in S_n \subset W_n$ such that $g\dw (s) = s$.
Since $\dw$ leaves $\FD$ invariant, by replacing $g$ by $g\dw$ and $z$ by $\dw\iv z$, 
we may further assume that 
$g(s) = s, g(z) \in \Fn$ with $s \in \Ft_{\sr}, z \in \FD$. 
Hence $g \in Z_H(s) \simeq SL_2 \times \cdots \times SL_2$ ($n$-factors). 
If we write $g = (g_1, \dots, g_n)$ with $g_i \in SL_2$, and 
$z = (z_1, \dots, z_n) \in \FD$, the action of $g$ on 
$z \in \FD$ corresponds to the conjugate action of $g_i$ on the matrix 
   $\begin{pmatrix}
         0  &  z_i  \\
         0  &  0
    \end{pmatrix}$ for $i = 1, \dots, n$. 
Now the condition $g(z) \in \Fn$ implies that, if $z_i \ne 0$ then 
$g_i$ is upper triangular.  It follows that $g(z) \in \FD$, 
and so $y \in \Ft_{sr}  + \FD$, up to $B$-conjugate. 
We have $Y \cap \Fb \subset \FN_{n,\sr}$. The other inclusion is obvious. 
The lemma is proved.    
\end{proof} 

\para{2.6.}
For $s \in \Ft_{sr}$, we have 
$Z_H(s) = Z_H(\Ft_{\sr}) \simeq SL_2 \times \cdots \times SL_2$ ($n$-factors), 
and $B \cap Z_H(\Ft_{\sr}) \simeq B_2 \times \cdots \times B_2$, where 
$B_2$ is the Borel subgroup of $SL_2$ consisting of upper triangular matrices.
The action of $Z_H(\Ft_{\sr})$ on $\FD$ is described as in 
the proof of Lemma 2.5.  In particular, $B \cap Z_H(\Ft_{\sr})$ leaves $\FD_k$
invariant for each $k$. Then  
$\wt Y_k$ can be expressed as
\begin{align*}
\tag{2.6.1}
\wt Y_k \simeq H \times^{B}\FN_{k,\sr} 
               \simeq H\times^{B \cap Z_{H}(\Ft_{\sr})}(\Ft_{\sr} + \FD_k). 
\end{align*}
Hence $\wt Y_k$ is smooth and irreducible. 
$\wt Y \simeq H \times^{B \cap Z_H(\Ft_{\sr})}(\Ft_{\sr} + \FD)$ 
is a locally trivial fibration over $H/(B \cap Z_H(\Ft_{\sr}))$ with fibre 
isomorphic to $\Ft_{\sr} + \FD$, and $\wt Y_k$  
is a subbundle of $\wt Y$ corresponding to a closed subset $\Ft_{\sr} + \FD_k$ 
of $\Ft_{\sr} + \FD$.  Hence $\wt Y_k$ is a closed subset of $\wt Y$ for each $k$.
The map $\psi^{(k)}$ is the restriction of $\psi$ on $\wt Y_k$.
Since $\psi :\wt Y \to Y$ is proper, $Y_k = \psi(\wt Y_k)$ is an irreducible 
closed subset in $Y$.  
$\psi^{(k)}: \wt Y_k \to Y_k$ is a proper surjective map. 
\par
The following relation can be verified by a similar argument 
as in the proof of Lemma 2.5. 
\begin{equation*}
\tag{2.6.2}
Y_k \cap (\Ft_{sr} + \FD) = \bigcup_{w \in S_n}\dw(\Ft_{\sr} + \FD_k), \quad (0 \le k \le n).
\end{equation*}
Put $Y_k^0 = Y_k - Y_{k-1}$.  By (2.6.2), we have
\begin{equation*}
\tag{2.6.3}
Y_k^0 = \bigcup_{g \in H}g(\Ft_{\sr} + \FD_k^0).
\end{equation*}

\para{2.7.}
For any subset $I \subset [1,n]$, put 
\begin{equation*}
\tag{2.7.1}
\FD_I = \{ \Bb \in \FD \mid b_i \ne 0 \text{ for } i \in I, 
              b_i = 0 \text{ for } i \notin I \}.
\end{equation*} 
Note that if $I = [1,k]$, $\FD_I$ coincides with $\FD^0_k$. 
Since the action of $B \cap Z_H(\Ft_{\sr})$ on $\FD$ is given by the action of its 
$T$-part, $\Ft_{\sr} + \FD_I$ is $B \cap Z_H(\Ft_{\sr})$-stable. 
We define a locally closed subvariety $\wt Y_I$ of $\wt Y$ by 
\begin{equation*}
\tag{2.7.2}
\wt Y_I \simeq H\times^{B \cap Z_H(\Ft_{\sr})}(\Ft_{\sr} + \FD_I), 
\end{equation*}
and a map $\psi_I : \wt Y_I \to Y$ by $g*x \mapsto gx$, where 
$g*x$ is the image of $(g,x) \in H \times (\Ft_{\sr} + \FD_I)$ on 
its quotient.
Then $\Im \psi_I = \bigcup_{g \in H}(\Ft_{\sr} + \FD_I)$ 
coincides with $Y_k^0$ for $k = |I|$ by (2.6.3), 
which depends only on $k$. 
\par
For $I \subset [1,n]$, we define a parabolic subgroup $Z_H(\Ft_{\sr})_I$ of 
$Z_H(\Ft_{\sr})$ by the condition that the $i$-th factor is $B_2$ if $i \in I$, and 
is $SL_2$ otherwise. Since $Z_H(\Ft_{\sr})_I$ stabilizes $\FD_I$, one can define 
\begin{equation*}
\tag{2.7.3}
\wh Y_I = H \times^{Z_H(\Ft_{\sr})_I}(\Ft_{\sr} + \FD_I).
\end{equation*}
Then $\psi_I$ factors through $\wh Y_I$, 
\begin{equation*}
\tag{2.7.4}
\begin{CD}
\psi_I : \wt Y_I @>\xi_I>> \wh Y_I @>\eta_I>> Y_k^0, 
\end{CD}
\end{equation*}
for $|I| = k$, where $\xi_I$ is the natural surjection, and 
$\eta_I$ is given by $g*x \mapsto gx$ (similar notation as $\psi_I$). 
Then $\xi_I$ is a locally trivial fibration 
with fibre isomorphic to 
\begin{equation*}
Z_H(\Ft_{\sr})_I/(B \cap Z_H(\Ft_{\sr})) \simeq (SL_2/B_2)^{I'} \simeq \BP_1^{I'},
\end{equation*}
where $I'$ is the compliment of $I$ in $[1,n]$.  
\par
Let $S_I$ be the symmetric group of letters in $I \subset [1,n]$, hence 
$S_I \simeq S_k$ for $|I| = k$. 
Then $\SW_I = N_H(Z_H(\Ft_{\sr})_I)/Z_H(\Ft_{\sr})_I \simeq S_I \times S_{I'}$. 
$\SW_I$ acts on $\wt Y_I$ and $\wh Y_I$ since $\Ft_{\sr} + \FD_I$ is stable by 
$N_H(Z_H(\Ft_{\sr})_I)$.
Now the map $\eta_I : \wh Y_I \to Y^0_k$ turns out to be a finite Galois covering 
with Gaolis group $\SW_I$, 
\begin{equation*}
\tag{2.7.5}
\eta_I : \wh Y_I \to \wh Y_I/\SW_I \simeq Y^0_k.
\end{equation*} 

\para{2.8.}
For $0 \le k \le n$, we define $\wt Y^+_k$ as $\psi\iv(Y^0_k)$.  
Then $\wt Y^+_k = \coprod_I \wt Y_I$, where $I$ runs over all the subsets 
$I \subset [1,n]$ such that $|I| = k$. (The disjointness follows from (2.6.2)).  
$\wt Y_I$ is smooth, irreducible by (2.7.2), and $\wt Y_I$ form the connected components
of $\wt Y^+_k$.  
Since $Y = \coprod_{0 \le k \le n}Y^0_k$, we have
\begin{equation*}
\wt Y = \coprod_{0 \le k \le n}\wt Y^+_k.
\end{equation*}
In the case where $I = [1,k]$, we denote $\wt Y_I$
by $\wt Y_k^0$.  Then  
$\wt Y^0_k$ is an open dense subset of $\wt Y_k$. 
By (2.6.1), $S_n \simeq N_H(Z_H(\Ft_{\sr}))/Z_H(\Ft_{\sr})$ 
acts on $\wt Y$, which leaves $\wt Y^+_k$ stable for any $k$. 
We have
\begin{equation*}
\tag{2.8.1}
\wt Y^+_k = \coprod_{\substack{ I \subset [1,n] \\ |I| = k}}
                        \wt Y_I = \coprod_{w \in S_n/(S_k \times S_{n-k})}w(\wt Y^0_k). 
\end{equation*} 
We have the following lemma.

\begin{lem} 
Let the notations be as before.
\begin{enumerate}
\item 
$X = \bigcup_{g \in H}g(\Fb) = \Fh$.
\item 
$Y_k$ is an irreducible closed subset in $Y$.  Hence $Y_k^0$ is open 
dense in $Y_k$. 
\item
$\dim \wt Y_k = \dim H - n + k$.
\item
$\dim Y_k = \dim \wt Y_k - (n-k) = (\dim H - 2n) + 2k$.
\item
$Y = \coprod_{0 \le k \le n}Y_k^0$ gives a stratification of $Y$ by smooth strata $Y_k^0$,
and the map $\psi : \wt Y \to Y$ is semismall with respect to this stratification.
\end{enumerate}
\end{lem}

\begin{proof}
(ii) is already given in 2.6.
By (2.6.1), 
\begin{align*}
\dim \wt Y_k &= \dim H - \dim (B \cap Z_H(\Ft_{\sr})) + \dim (\Ft_{\sr} + \FD_k) \\
             &= \dim H - 2n + (n + k)  \\
             &= \dim H - n + k,
\end{align*}
since $\dim (B \cap Z_H(\Ft_{\sr})) = \dim (B_2 \times \cdots \times B_2) = 2n$.
Thus (iii) holds. 
Since $\wh Y_I$ is smooth, irreducible, and $\eta_I$ is a finite Galois covering, 
$Y_k^0 = \eta_I(\wh Y_I)$ is smooth, irreducible. 
By using the decomposition $\psi_I = \eta_I\circ \xi_I$ for $I = [1,k]$, 
we see that $\dim \wt Y_k = \dim Y_k + (n-k)$. Hence (iv) holds.
It follows from (iv), $\dim Y = \dim Y_n = \dim H$.
Since $\dim \psi\iv(x) = n - k$ for any $x \in Y_k^0$ by (2.7.1) 
and by $\psi_I = \eta_I\circ \xi_I$, 
we have $\dim \psi\iv(x) = (\dim Y - \dim Y_k^0)/2$ by (iv).  Thus (v) holds. 
Since $Y$ is open dense in $X$, 
$\dim X = \dim H$.  Since $X$ is irreducible closed in $\Fh$, we obtain (i).
\end{proof}

\para{2.10.}
Let $\psi_k : \wt Y^+_k \to Y_k^0$ be the restriction of $\psi$ on $\psi\iv(Y_k^0)$.
Since $\psi$ is proper, $\psi_k$ is also proper. 
Let $\Ql$ be the constant sheaf on $\wt Y^+_k$.  Since $\wt Y_I$ is a connected 
component for any $I$ by (2.8.1), we have
\begin{equation*}
\tag{2.10.1}
(\psi_k)_!\Ql \simeq \bigoplus_{\substack{ I \subset [1,n] \\ |I| = k }}
                 (\psi_I)_!\Ql. 
\end{equation*}
\par
On the other hand, since $\eta_I : \wh Y_I \to Y^0_k$ is a finite Galois 
covering with group $\SW_I$, we have
\begin{equation*}
\tag{2.10.2}
(\eta_I)_!\Ql \simeq \bigoplus_{\r \in (\SW_I)\wg}\r \otimes \SL_{\r},  
\end{equation*}
where $\SL_{\r} = \Hom (\r, (\eta_I)_!\Ql)$ is a simple local system on $Y^0_k$ 
corresponding to $\r \in (\SW_I)\wg$. 
\par
Since $\psi_k$ is proper, and $\wt Y_I$ is a closed subset of $\wt Y^+_k$, 
$\psi_I$ is proper.  As $\psi_I = \eta_I\circ \xi_I$, $\xi_I$ is also proper. 
By a similar discussion as in [Sh, (1.6.1)], we see that $R^i(\xi_I)_!\Ql$ 
is a constant sheaf for any $i \ge 0$.  
Since $\xi_I$ is a $\BP_1^{I'}$-bundle, we have 
\begin{equation*}
\tag{2.10.3}
(\xi_I)_!\Ql \simeq (\xi_I)_!(\xi_I)^*\Ql \simeq H^{\bullet}(\BP_1^{I'})\otimes \Ql,
\end{equation*}
where $H^{\bullet}(\BP_1^{I'})$ denotes $\bigoplus_{i \ge 0}H^{2i}(\BP_1^{I'},\Ql)$,
which we regard as a complex of vector spaces $(K_i)$ with $K_{\odd} = 0$. 
It follows that 
\begin{equation*}
\tag{2.10.4}
(\psi_I)_!\Ql \simeq (\eta_I)_!(\xi_I)_!\Ql \simeq  
               H^{\bullet}(\BP_1^{I'})\otimes (\eta_I)_!\Ql. 
\end{equation*}
\par
Note that $\BP_1$ is the flag variety of $SL_2$, and $\BZ/2\BZ$ is the Weyl group
of $SL_2$. 
We define an action of $\BZ/2\BZ$ on $H^{\bullet}(\BP_1)$ as the Springer representation 
of $\BZ/2\BZ$, i.e., $\BZ/2\BZ$ acts non-trivially on $H^2(\BP_1) = \Ql$, and 
acts trivially on $H^0(\BP_1) = \Ql$.   
We define an action of $(\BZ/2\BZ)^{[1,n]}$ on 
$H^{\bullet}(\BP_1^{I'}) \simeq H^{\bullet}(\BP_1)^{\otimes |I'|}$ 
by the Springer action of the factor $\BZ/2\BZ$ 
corresponding to $I'$, and the trivial action of the factor $\BZ/2\BZ$ 
corresponding to $I$.  
Note that  $W_n = S_n \ltimes (\BZ/2\BZ)^n$ and  
$\SW_I \simeq S_I \times S_{I'}$. 
Let $W_I = S_I \ltimes (\BZ/2\BZ)^I$ be the Weyl group of type $C_{|I|}$,
and define $W_{I'}$ similarly. 
For $\r \in \SW_I\wg = (S_I \times S_{I'})\wg$, we consider the action of 
$W_I \times W_{I'}$ on 
$H^{\bullet}(\BP_1^{I'})\otimes \r$, 
where $(\BZ/2\BZ)^{[1,n]}$ acts trivially on $\r$. 
In particular, for $I = [1,k]$, we obtain $(W_k \times W_{n-k})$-module
$H^{\bullet}(\BP_1^{n-k})\otimes \r$.  Note that 
$S_n/(S_k \times S_{n-k}) \simeq W_n/(W_k \times W_{n-k})$.  
By (2.8.1), (2.10.1) and (2.10.4), we have
\begin{equation*}
\tag{2.10.5}
(\psi_k)_!\Ql \simeq \bigoplus_{\r \in (S_k \times S_{n-k})\wg}
                      \Ind_{W_k \times W_{n-k}}^{W_n}
                 (H^{\bullet}(\BP_1^{n - k})\otimes \r)\otimes \SL_{\r},
\end{equation*}
where $\SL_{\r} = \Hom (\r, (\eta_I)_!\Ql)$ is the simple local system 
on $Y_k^0$ with $I = [1,k]$. 

\para{2.11.}
For each $k$, let $\ol\psi_k$ be the restriction of $\psi$ on 
$\psi\iv(Y_k)$.  Then $\ol\psi_k : \psi\iv(Y_k) \to Y_k$ is a proper map. 
Assume that $k \ge 1$.  We have $Y_k - Y_{k-1} = Y_k^0$.  
Since $\ol \psi_k$ is proper, 
$(\ol\psi_k)_!\Ql$ is a semisimple complex on $Y_k$.  
We note the following.
\par\medskip\noindent
(2.11.1) \ Assume that $(\ol\psi_{k-1})_!\Ql$ is equipped with 
an action of $W_n$.  
Then the $W_n$-action can be extended to a $W_n$-action on  
$(\ol\psi_k)_!$.  
\par\medskip
In fact, let $j : Y_k^0 \hra Y_k$ be the open immersion.  
By (2.10.5), $(\psi_k)_!\Ql$ has an action of $W_n$. This induces 
an action of $W_n$ on $(j\circ \psi_k)_!\Ql$, and on its perverse cohomology 
${}^pH^i((j\circ\psi_k)_!\Ql)$.  On the other hand, by the assumption, 
${}^pH^i((\ol\psi_{k-1})_!\Ql)$ is equipped with $W_n$-action. 
We consider the long exact sequence of the 
perverse cohomology obtained from the distinguished triangle 
$(j_!(\psi_k)_!\Ql, (\ol\psi_k)_!\Ql, (\ol\psi_{k-1})_!\Ql)$. 
By (2.10.5), $(\psi_k)_!\Ql$ is a semisimple complex which is a sum of 
various $\SL_{\r}[2i]$.  Hence ${}^pH^i((j\circ\psi_k)_!\Ql) = 0$ for odd $i$.  
By induction, we have ${}^pH^i((\psi_k)_!\Ql)= 0$ for odd $i$.  
Since $(\ol\psi_k)_!\Ql$ is a semisimple 
complex, the $W_n$-action of ${}^pH^{i}((\ol\psi_{k-1})_!\Ql)$ and on 
${}^pH^{i}((j\circ\psi_k)_!\Ql)$ determines the $W_n$-action on 
${}^pH^{i}((\ol\psi_k)_!\Ql)$, uniquely. 
(2.11.1) is proved.     

\para{2.12.}
We have a 
natural bijection 
\begin{equation*}
\tag{2.12.1}
\coprod_{0 \le k \le n}(S_k \times S_{n-k})\wg \simeq W_n\wg,  
  \qquad \r \longleftrightarrow  \wh\r
\end{equation*}
satisfying the following properties; 
take $\r = \r'\boxtimes \r'' \in (S_k \times S_{n - k})\wg$, 
where $\r' \in S_k\wg, \r'' \in S_{n-k}\wg$. 
We extend $\r'$ to $\wt\r' \in W_k\wg$ so that $(\BZ/2\BZ)^k$ acts 
trivially on it. On the other hand, we extend $\r''$ to $\wt\r'' \in W_{n-k}\wg$
so that each factor $\BZ/2\BZ$ of $(\BZ/2\BZ)^{n-k}$ 
acts non-trivially on $\wt\r''$.   
Put 
\begin{equation*}
\tag{2.12.2}
\wh\r = \Ind_{W_k \times W_{n-k}}^{W_n} (\wt\r'\boxtimes \wt\r''). 
\end{equation*}
Then $\wh\r \in W_n\wg$, and the correspondence 
$\r \mapsto \wh\r$ gives the required 
bijection.

We show the following proposition.  Put $d_k = \dim Y_k$. 
\begin{prop} 
$\psi_!\Ql[d_n]$ is a semisimple perverse sheaf on $Y$, equipped with $W_n$-action, 
and is decomposed as 
\begin{equation*}
\tag{2.13.1}
\psi_!\Ql[d_n] \simeq \bigoplus_{0 \le k \le n}
                      \bigoplus_{\r \in (S_k \times S_{n-k})\wg}\wh \r \otimes 
                                     \IC(Y_k, \SL_{\r})[d_k], 
\end{equation*}
where $\SL_{\r}$ is a simple local system on $Y_k^0$ defined in (2.10.2). 
\end{prop}

\begin{proof}
The formula (2.10.5) can be rewritten as 
\begin{equation*}
\tag{2.13.2}
(\psi_k)_!\Ql \simeq \biggl(\bigoplus_{\r \in (S_k \times S_{n-k})\wg}
                 \wh\r \otimes \SL_{\r}\biggr)[-2(n - k)] + \SN_k,
\end{equation*}
where $\SN_k$ is a sum of various $\SL_{\r'}[-2i]$ for 
$\r' \in (S_k \times S_{n-k})\wg$ with $0 \le i < n - k$. 
\par
For $0 \le m \le n$, let $\ol\psi_m$ be as in 2.10. We consider the following formula.

\begin{align*}
\tag{2.13.3}
(\ol\psi_m)_!\Ql
     \simeq \bigoplus_{0 \le k \le m}\bigoplus_{\r \in (S_k \times S_{n - k})\wg}
               \wh\r \otimes \IC(Y_k, \SL_{\r})[-2(n - k)] 
          + \ol\SN_m,           
\end{align*}
where $\ol\SN_m$ is a $\BZ$-linear combination of 
various $\IC(Y_k, \SL_{\r'})[-2i]$ for $0 \le k \le m$
and $\r' \in (S_k \times S_{n-k})\wg$ with $i < n - k$. 
We note that (2.13.3) will imply the proposition.  In fact, 
$\ol\psi_n = \psi$ for $k = n$, and 
$d_n - d_k = 2n - 2k$ by Lemma 2.9. 
But since $d_n - 2i > d_k$ if $i < n-k$, 
$\IC(Y_k, \SL_{\r'})[d_n - 2i]$ is not a perverse sheaf. 
Since $\psi$ is semismall, $\psi_!\Ql$ is a semisimple perverse sheaf. 
Thus $\ol\SN_n = 0$ and (2.13.1) follows. 
\par
We show (2.13.3) by induction on $m$. If $m = 0$, then $(\ol\psi_m)_!\Ql$ 
coincides with $(\psi_m)_!\Ql$. Hence (2.13.3) holds by (2.13.2) applied for $k = 0$.  
We assume that (2.13.3) holds for any $k < m$. 
Recall that $Y_m^0 = Y_m - Y_{m -1}$. 
Since $(\ol\psi_m)_!\Ql$ is a semisimple complex, it is a direct sum of the form $A[s]$ 
for a simple perverse sheaf $A$. Suppose that the support $\supp A$ of $A$ is not contained in 
$Y_{m -1}$.  Then $(\supp A) \cap Y_m^0 \neq \emptyset$, and $A|_{Y_m^0}$ is a simple perverse 
sheaf on $Y_m^0$. The restriction of $(\ol\psi_m)_!\Ql$ on $Y_m^0$ is isomorphic to 
$(\psi_m)_!\Ql$, and it is described as in (2.13.2). It follows that 
$A|_{Y_m^0} = \SL_{\r}$ (up to shift) for some $\r \in (S_m \times S_{n-m})\wg$. 
This implies that $A = \IC(Y_m, \SL_{\r})[d_m]$, and the direct sum of $A[s]$ 
appearing in $(\ol\psi_m)_!\Ql$ 
such that $(\supp A) \cap Y_m^0 \ne \emptyset$ is given by
\begin{equation*}
K_1 = \bigoplus_{\r \in (S_m \times S_{n-m})\wg}\wh\r \otimes 
                   \IC(Y_m, \SL_{\r})[-2(n - m)] + \SN'_m, 
\end{equation*}
where $\SN'_m$ is a $\BZ$-linear combination of $\IC(Y_m, \SL_{\r'})[-2i]$ 
for $\r' \in (S_m \times S_{n-m})\wg$ with $0 \le i < n - m$.   
\par
If $\supp A$ is contained in $Y_{m-1}$, $A[s]$ appears as a summand of 
$(\ol\psi_{m -1})_!\Ql$.  
By induction, $(\ol\psi_{m -1})_!\Ql$ 
is described as in (2.13.3) by replacing $m$ by $m -1$. 
Thus if exclude the contribution from the restriction of $K_1$ on $Y_{m -1}$, such 
$A[s]$ is determined from $(\ol\psi_{m -1})_!\Ql$. 
Note that, by induction,  we can 
construct an action of $W_n$ on $(\ol\psi_m)_!\Ql$ by (2.11.1).  
We consider the restriction of $K_1$ on $Y_{m -1}$.
Since each simple component of $\SN'_m$ is 
contained in $\ol\SN_m$, we can ignore this part. 
Let $K_1'$ be the direct sum part of $K_1$.  
Then the restriction of $K_1'$ on $Y_{m - 1}$ affords the representation of 
$W_n$ corresponding to a sum of various $\wh\r$ for 
$\r \in (S_m \times S_{n-m})\wg$. 
But by (2.13.3) applied for $m - 1$, the direct sum part of 
$(\ol\psi_{m -1})_!\Ql$
affords the representation of $W_n$. The irreducible representations appearing 
there is of the form $\wh\r'$, which are different from $\wh\r$ for $K_1'|_{Y_{m -1}}$.
Since each component of $\ol\SN_{m -1}$ is contained in $\ol\SN_m$, we see that 
the restriction of $K_1$ on $Y_{m -1}$ has no overlapping with 
$(\ol\psi_{m -1})_!\Ql$ 
modulo $\ol\SN_m$. Thus (2.13.3) holds for $m$.  The proposition is proved.   
\end{proof}

\par\bigskip
\section{ Intersection cohomology on $\Fc\Fs\Fp(V)$}

\para{3.1.} 
We follow the notation in Section 2. The conformal symplectic group 
$CSp_N$ is defined  by 
\begin{equation*}
CSp_N = \{ g \in GL_N \mid {}^tgJg = \la_gJ 
            \text{ for some }\la_g \in \Bk^* \},
\end{equation*}
which we denote by $\wt H$. By fixing the basis of $V$ as before, 
we also write it as $\wt H = CSp(V)$. $\wt H$ is a connected group 
with connected center $\wt Z$, where 
\begin{equation*}
\wt Z = \{ \la 1_N \mid \la \in \Bk^*\}, 
\end{equation*} 
and contains $H$ as a closed subgroup. 
$\wt H/\wt Z$ is the adjoint symplectic group $\wt H_{\ad}$. 
Let $\wt\Fh = \Fc\Fs\Fp(V)$ be the Lie algebra of $\wt H$. 
$\wt\Fh$ contains the center $\wt\Fz \simeq \Bk$, and 
we put $\wt\Fh_{\ad} = \wt\Fh/\wt\Fz$, which is the Lie 
algebra of $\wt H_{\ad}$, and is called the adjoint Lie algebra. 
In [X1, Lemma 6.2], Xue proved that $\wt\Fh_{\ad}$ has regular semisimple
elements, and established the Springer correspondence for $\Fh\nil$ by 
making use of the intersection cohomology on $\wt\Fh_{\ad}$. 
Considering $\wt\Fh_{\ad}$ is essentially the same as considering $\wt\Fh$.  
In this section, we shall connect Xue's result with ours discussed in Section 2. 

\para{3.2.}
Put
\begin{align*}
\tag{3.2.1}
\wt T &= \{\Diag(t_1, \dots, t_n, \la t_1\iv, \dots, \la t_n\iv)
                \mid t_i \in \Bk^*, \la \in \Bk^* \}, \\
\wt \Ft &= \{\Diag(t_1, \dots, t_n, t_1 + \la, \dots, t_n + \la)
               \mid t_i \in \Bk, \la \in \Bk \}.  
\end{align*}
Then $\wt T$ is a maximal torus of $\wt H$ containing $T$, and 
$\Lie \wt T = \wt \Ft \supset \Ft$.
We denote an element in $\wt\Ft$ as $\xi = (s,\la)$ for 
$s = (t_1, \dots, t_n) \in \Bk^n$ and $\la \in \Bk^*$.   
Let $\wt B = \wt TU$ be the Borel subgroup of $\wt H$ containing $B$.  
Put $\wt\Fb = \Lie \wt B$.  We have $\wt\Fb = \wt\Ft \oplus \Fn$. 
Put 
\begin{align*}
\tag{3.2.2}
\wt\Ft\reg = \{ (s,\la) \in \wt\Ft \mid t_i \ne t_j \text{ for } i \ne j, \la \in \Bk^*\}.
\end{align*}
Then for $\xi \in \wt\Ft\reg$, we have $Z_H(\xi) = \wt T$.  Thus $\xi$ is a regular 
semisimple element in $\wt\Ft$. 
$\wt\Ft\reg$ is open dense in $\wt\Ft$.  Put 
$\wt\Fh\reg = \bigcup_{g \in \wt H}g(\wt\Ft\reg)$. By using the conjugacy of maximal tori 
in $\wt H$, we see that $\wt\Fh\reg$ coincides with the set of 
regular semisimple elements in 
$\wt\Fh$.  $\wt\Fh\reg$ is open dense in $\wt\Fh$ since it is the intersection with 
regular semisimple elements in $\Fg\Fl_N$. Put 
$\wt\Fb\reg = \wt\Fh\reg \cap \wt\Fb$. 
We consider the varieties
\begin{align*}
\wt Y\flt     &= \{ (x, g\wt B ) \in \wt\Fh \times \wt H/\wt B
                      \mid g\iv x \in \wt\Fb\reg \}, \\
    Y\flt     &= \bigcup_{g \in \wt H}g(\wt\Fb\reg) = \wt\Fh\reg,
\end{align*}
and define a map $\psi\flt : \wt Y\flt \to Y\flt$ by $(x, g\wt B) \mapsto x$.
Then 
\begin{equation*}
\wt Y\flt \simeq \wt H\times^{\wt B}\wt\Fb\reg \simeq \wt H \times^{\wt T}\wt\Ft\reg,
\end{equation*}
and $\psi\flt$ is a finite Galois covering with Galois group $W_n$. 
\par
Let $\Ql$ be the constant sheaf on $\wt Y\flt$.  Then $(\psi\flt)_!\Ql$ is a semisimple 
local system on $Y\flt$, equipped with $W_n$-action, and is decomposed as 
\begin{equation*}
\tag{3.2.3}
(\psi\flt)_!\Ql \simeq \bigoplus_{\r \in W_n\wg}\r \otimes \SL\flt_{\r},
\end{equation*} 
where $\SL\flt_{\r} = \Hom (\r, (\psi\flt)_!\Ql)$ is a simple local system on $Y\flt$. 

\para{3.3.}
We consider the varieties
\begin{align*}
\wt X\flt &= \{ (x, g\wt B) \in \wt\Fh \times \wt H/\wt B \mid g\iv x \in \wt\Fb \}, \\
    X\flt &= \bigcup_{g \in \wt H}g(\wt\Fb), 
\end{align*}  
and define a map $\pi\flt : \wt X\flt \to X\flt$ by 
$(x, g\wt B) \mapsto x$. 
$\pi\flt$ is a proper map onto $X\flt$. Hence 
$X\flt$ is irreducible and closed in $\wt\Fh$. Since $\wt\Fh\reg \subset X\flt$, 
we have $X\flt = \wt\Fh$. 
\par
We consider the complex 
$K = (\pi\flt)_!\Ql$ on $X\flt = \wt\Fh$. 
We can define a similar map 
$\pi\flt_{\ad} ; \wt X\flt_{\ad} \to X\flt_{\ad} = \wt\Fh_{\ad}$, by replacing 
$\wt X\flt, X\flt, \pi\flt$ by $\wt X\flt_{\ad}, X\flt_{\ad}, \pi\flt_{\ad}$, 
respectively. We consider the complex $K_{\ad} = (\pi\flt_{\ad})_!\Ql$ on $\wt\Fh_{\ad}$. 
Let $\f : \wt\Fh \to \wt\Fh_{\ad}$ be the natural projection. 
By the base change theorem, we have $\f^*K_{\ad} \simeq K$. 
It is known by [X1, Prop. 6.6] that $K_{\ad}$ coincides with the intersection 
cohomology $\IC(\wt\Fh_{\ad}, \SL_{\ad})$ for a certain semisimple local system $\SL_{\ad}$ 
on the set of regular semisimple elements in $\wt\Fh_{\ad}$. 
Since $\f$ is smooth with connected fibre, $K$ is also expressed by 
an intersection cohomology on $\wt\Fh$. 
Since $K|_{Y\flt} \simeq (\psi\flt)_!\Ql$, we have  
$K \simeq \IC(\wt\Fh, (\psi\flt)_!\Ql)$. Thus by (3.2.3), the following result holds.
\begin{prop}  
$(\pi\flt)_!\Ql[\dim \wt\Fh]$ is a semisimple perverse sheaf on $\wt\Fh$, 
equipped with $W_n$-action, and is decomposed as
\begin{equation*}
\tag{3.4.1}
(\pi\flt)_!\Ql[\dim \wt\Fh] \simeq \bigoplus_{\r \in W_n\wg}
                 \r \otimes \IC(\wt\Fh, \SL\flt_{\r})[\dim \wt\Fh],
\end{equation*}
where $\SL\flt_{\r}$ is a simple local system on $\wt\Fh\reg$ given in (3.2.3). 
\end{prop}

\para{3.5.}
The set of nilpotent elements $\wt\Fh\nil$ in $\wt\Fh$ coincides with 
$\Fh\nil$. 
The subvariety $(\pi\flt)\iv(\Fh\nil)$ of $\wt X\flt$ can be identified with 
\begin{equation*}
\wt X\nil = \{ (x, gB) \in \Fh \times G/B \mid g\iv x \in \Fn\},
\end{equation*}
and the restriction of $\pi\flt$ on $\wt X\nil$ coincides with the map
$\pi_1 : \wt X\nil \to \Fh\nil$. Note that $\pi_1$ is surjective 
since $\bigcup_{g \in H}g(\Fn) = \Fh\nil$ by Lemma 2.9 (i). 
Since $(\pi\flt)_!\Ql$ has a natural action of $W_n$ by Proposition 3.4, 
$(\pi_1)_!\Ql$ has also an action of $W_n$.  The following result gives the 
Springer correspondence for $\Fs\Fp(V)$, which is essentially due to Xue [X1].

\begin{thm}  
\begin{enumerate}
\item 
$(\pi_1)_!\Ql[\dim \Fh\nil]$ is a semisimple perverse sheaf on $\Fh\nil$, equipped with 
$W_n$-action, and is decomposed as 
\begin{equation*}
(\pi_1)_!\Ql[\dim \Fh\nil] \simeq \bigoplus_{\r \in W_n\wg}
             \r \otimes \IC(\ol\SO_{\r}, \Ql)[\dim \SO_{\r}],
\end{equation*}
where $\SO_{\r}$ is an $H$-orbit in $\Fh\nil$, and the map 
$\r \mapsto \SO_{\r}$ gives a bijective correspondence between $W_n\wg$ 
and the set of $H$-orbits in $\Fh\nil$. 
\item 
For each $\r \in W_n\wg$, we have 
\begin{equation*}
\IC(\wt\Fh, \SL\flt_{\r})|_{\Fh\nil} \simeq \IC(\ol\SO_{\r}, \Ql), 
                    \quad\text{ $($up to shift$)$.}
\end{equation*}
\end{enumerate}
\end{thm}
\par
In fact, Xue proved in [X1, Prop. 6.4] the corresponding formula for 
$(\wt\Fh_{\ad})\nil$, the nilpotent variety of $\wt\Fh_{\ad}$.
Since $\f\iv((\wt\Fh_{\ad})\nil) \simeq \wt\Fz \times \wt\Fh\nil$,  
his result can be translated to the formula in $\wt\Fh\nil$, which induces
our formula for $\Fh\nil$.
Note that $\f$ gives a bijective map $\Fh\nil \to (\wt\Fh_{\ad})\nil$, compatible 
with the action of $\wt H$ and $\wt H_{\ad}$.  
Also note that it is known by [Spa] that for any $x \in \Fh\nil$, $Z_H(x)$ 
is connected.  Hence only the constant sheaf $\Ql$ appears as a local system 
on $\SO_{\r}$ in the Springer correspondence. The explicit correspondence was 
described in [X2] by making use of (a generalization of) Lusztig's symbols. 

\para{3.7.}
Let $\SB = \wt H/\wt B$ be the flag variety of $\wt H$, 
and $W  = N_{\wt H}(\wt T)/\wt T \simeq W_n$ be the Weyl group of $\wt H$. 
For each $x \in \wt\Fh$, put 
$\SB_x = \{ g\wt B \in \SB \mid g\iv x \in \wt\Fb \}$. 
We consider the structure of $\SB_x$.  Let $x = s + z$ be the Jordan decomposition of 
$x$ in $\wt \Fh$, where $s$: semisimple and $z$: nilpotent such that $[s, z] = 0$. 
We assume that $s \in \wt\Ft$. Put $C = Z^0_{\wt H}(s)$ and $\Fc = \Lie C$.    
Then $z \in \Fc\nil$. $\wt B_C = \wt B \cap C$ is a Borel subgroup of $C$ containing 
$\wt T$, and we consider the flag variety $\SB^C = C/\wt B_C$, and its subvariety 
$\SB^C_z$.  For each $w \in W $, $\wt B_{C,w} = w\wt Bw\iv \cap C$ 
is a Borel subgroup of $C$ containing 
$\wt T$, and one can consider the flag variety $\SB^{C,w} = C/\wt B_{C,w}$ of $C$.   
Let $W_s$ be the stabilizer of $s$ in $W$.  Then $W_s$ is the Weyl group 
of $C$.  
Put 
\begin{align*}
\wh\SM  &= \{ g \in \wt H \mid g\iv s\in \wt\Fb \}, \\  
  \SM &= \{ g \in \wt H \mid g\iv s \in \wt\Ft \}.
\end{align*} 
Then $C \times \wt B$ acts on $\wh\SM$, 
by $(h, b) : y \mapsto hyb\iv$, 
and similarly $C \times \wt T$ acts on $\SM$. 
Put $\vG = C\backslash \wt H/\wt T$, $\wh\vG  = C \backslash \wt H/\wt B$.
The natural map $\vG \to \wh\vG$ gives a bijection $\vG \simeq \wh\vG$, 
and we can identify $\vG$ with a set of representatives in $W$ of the cosets 
$W_s\backslash W$. 
This implies that $\coprod_{w \in \vG}\SB^{C,w} \simeq \SB_s$ by 
$h(\wt B_{C,w}) \mapsto hwB$, and we have
\begin{equation*}
\tag{3.7.1}
\SB_x \simeq \coprod_{w \in \vG}\SB^{C,w}_z 
         \simeq \coprod_{w \in W_s\backslash W}\SB^C_z.
\end{equation*} 
\par
Recall that $K = (\pi\flt)_!\Ql$ is a complex with $W$-action by Proposition 3.4.  
Hence for any $x \in \wt\Fh$, 
$\SH^i_xK \simeq H^i(\SB_x, \Ql)$ has a structure of $W$-module 
(Springer representation of $W$). For $C$, we can also 
consider $\pi_C : C \times^{\wt B_C}\wt\Fb_C \mapsto \Fc$, similarly to $\pi\flt$, 
where $\wt\Fb_C = \Lie \wt B_C$.  Since 
$\Fc$ has regular semisimple elements, the previous discussion can be applied, and 
$(\pi_C)_!\Ql$ turns out to be a complex 
with $W_s$-action. 
It follows that $H^i(\SB^C_z, \Ql)$ is a $W_s$-module (Springer representation of $W_s$). 
Then (3.7.1) can be interpreted by Springer representations as follows. 
\begin{thm}  
$H^i(\SB_x, \Ql) \simeq \Ind_{W_s}^WH^i(\SB^C_z, \Ql)$ as $W$-modules. 
\end{thm}

\begin{proof}
This formula corresponds to the special case of Lusztig's character formula 
for generalized Green functions [L2, Thm. 8.5]. Concerning the proof,  
essentially the same argument can be applied to our setting (but ignoring the 
$\BF_q$-structure). 
We give an outline of the proof below for the sake of completeness. 
\par
We fix $s \in \wt\Ft, z \in \Fn$ such that $[s,z] = 0$. 
For $x \in \wt\Fh$, let $x_s$ be its semisimple part. 
As in [L2, Lemma 8.6], one can find an open dense subset $\SU$ of 
$\Fc = \Lie Z^0_{\wt H}(s)$
satisfying the following properties;
\par\medskip\noindent
(3.8.1) \ $\SU$ contains 0, and 
\par\medskip
(i) $g(\SU) = \SU$ for any $g \in C$, 
\par
(ii) $x \in \SU$ if and only if $x_s \in \SU$, 
\par
(iii) If $x \in \SU, g \in C$, $g\iv(s + x) \in \wt\Fb$, then 
$g\iv x_s \in \wt\Fb$ and $g\iv s \in \wt\Fb$. 
\par
(iv) If $x \in \SU, g \in C$, $g\iv(s + x) \in \wt\Ft$, then 
$g\iv x_s \in \wt\Ft$ and $g\iv s \in \wt\Ft$.    
\par
\medskip
Note that $\SU$ contains $\Fc\nil$ by (ii).
We define a subvariety $\wt X\flt_{\SU}$ of $\wt X\flt$ by 
\begin{equation*}
\wt X\flt_{\SU} = \{ s + x, g\wt B) \in \wt X\flt \mid x \in \SU \}. 
\end{equation*}
Let $\wh\g \in \wh\vG$ be a double coset in $\wt H$, and consider 
the variety
\begin{align*}
\wt X\flt_{\SU,\wh\g} &= \{ (s + x, g\wt B) \in \wt X\flt_{\SU} \mid g \in \wh\g \} \\
                       &= \{ (s + x, g\wt B) \in (s + \SU) \times \wt H/\wt B
                              \mid g \in \wh\g, g\iv(s + x) \in \wt\Fb\}.  
\end{align*}  
Then one can show that 
\begin{equation*}
\tag{3.8.2}
\wt X\flt_{\SU} = \coprod_{\wh\g \in \wh\vG}\wt X\flt_{\SU,\wh\g},
\end{equation*}
where $\wt X\flt_{\SU,\wh\g}$ is an open and closed subset of 
$\wt X\flt_{\SU}$ for $\wh\g \in \wh\vG$. 
\par
Let $\g \in \vG$ be the double coset in $\wt H$ corresponding to 
$\wh\g \in \wh\vG$.  Take $g_{\g} \in \g$ and 
put $B_{\g} = g_{\g}\wt Bg_{\g}\iv \cap C$. By definition, $s \in g_{\g}(\wt\Ft)$, and 
$B_{\g}$ is a Borel subgroup of $C$ containing a maximal torus 
$T_{\g} = \wt T$.
By replacing $\wt H, \wt B, \wt T$ by $C = Z_{\wt H}^0(s), B_{\g}, T_{\g}$, 
we can define $\psi_{\g} : \wt Y_{\g} \to Y_{\g} = \Fc\reg$, 
$\pi_{\g} : \wt X_{\g} \to X_{\g} = \Fc$
corresponding to $\psi\flt : \wt Y\flt \to Y\flt = \wt\Ft\reg$, 
$\pi\flt : \wt X\flt \to X\flt = \wt\Ft$.   
Put
\begin{equation*}
\wt X_{\SU,\g} = \pi_{\g}\iv(\SU) \subset \wt X_{\g}.
\end{equation*}
By using th property (iv) in (3.8.1), one can show that 
\par\medskip\noindent
(3.8.3) \ The map $(x, zB_{\g}) \mapsto (s + x, zg_{\g}\wt B)$ gives an isomorphism 
$\wt X_{\SU,\g} \isom \wt X\flt_{\SU, \wh\g}$. 
\par\medskip
Replacing $\wt X\flt$ by $\wt Y\flt$, a similar discussion works.
Put 
\begin{align*}
\wt Y\flt_{\SU, \g} &= \{ (x, g\wt T) \in \wt Y\flt_{\SU} \mid g \in \g \}, \\
\wt Y_{\SU,\g}      &= (\psi_{\g})\iv(\SU) \subset \wt Y_{\g},
\end{align*}
where 
$\wt Y\flt_{\SU} = (\psi\flt)\iv(\wt\Ft\reg \cap (s + \SU))$. 
Let $\g_0$ be the orbit in $\vG$ corresponding to $W_s \subset W$.
Now $W$ acts on $\wt Y\flt_{\SU}$ by $w : (x, g\wt T) \mapsto (x, gw\iv \wt T)$. 
Then $W$ permutes the subsets $\wt Y\flt_{\SU,\g}$, which corresponds to the 
right action of $W$ on $\vG$.  In particular, the stabilizer of $\wt Y\flt_{\SU,\g_0}$
in $W$ coincides with $W_s$.  
As an analogue of (3.8.2), we have
\begin{equation*}
\tag{3.8.4}
\wt Y\flt_{\SU} = \coprod_{\g \in \vG}\wt Y\flt_{\SU, \g} 
                    \simeq \coprod_{w \in W/W_s}w(\wt Y\flt_{\SU,\g_0}),
\end{equation*}
where $\wt Y\flt_{\SU,\g}$ is a non-empty, open and closed subset 
for each $\g \in \vG$. 
\par
Combining it with (3.8.3), the following holds.
\par\medskip\noindent
(3.8.5) \ The map $\pi\flt : \wt X\flt_{\SU,\wh\g} \to \wt\Fh$
is a proper map onto $s + \SU$. $\wt X\flt_{\SU,\wh\g}$ is irreducible.
$\psi\flt(\wt Y\flt_{\SU,\g}) = (s + \SU) \cap \wt\Fh\reg$, 
and  
$\wt Y\flt_{\SU,\g} = (\pi\flt)\iv\bigl((s + \SU) \cap \wt\Fh\reg\bigr) 
                           \cap \wt X\flt_{\SU,\g}$.
\par\medskip
Let  $\SV = s + (\Fc\reg \cap \SU)$, and define 
\begin{align*}
\wt Y\flt_{\SV} &= \{ (x, g\wt T) \in \wt Y\flt \mid x \in \SV \}, \\
(\wt Y_{\g})_{-s + \SV} &= \{ (x, gT_{\g}) \in \wt Y_{\g} \mid s + x \in \SV \}.  
\end{align*}
Then we have the following commutative diagram
\begin{equation*}
\tag{3.8.6}
\begin{CD}
\wt Y\flt_{\SV} @<\a<<  \coprod_{\g \in \vG}(\wt Y_{\g})_{-s + \SV} \\
     @VVV                        @VVV   \\
      \SV        @<\b<<      -s + \SV,
\end{CD}
\end{equation*}
where $\a: (x, zT_{\g}) \mapsto (s + x, zg_{\g}\wt T)$, 
$\b: x \mapsto s + x$, and the vertical maps are projections to the first factor.
It is shown that $\a$ turns out to be an isomorphism. $\b$ is also an isomorphism.  
\par
$\wt Y\flt_{\SV}$ is invariant under the action of $W$ on $\wt Y\flt_{\SU}$. 
In view of (3.8.4), it is decomposed as
\begin{equation*}
\tag{3.8.7}
\wt Y\flt_{\SV} = \coprod_{\g \in \vG}\wt Y\flt_{\SV, \g} 
                      \simeq \coprod_{w \in W/W_s}w(\wt Y\flt_{\SV,\g_0}), 
\end{equation*}
where $\wt Y\flt_{\SV,\g} = \{ (x, g\wt T) \in \wt Y\flt \mid x \in \SV, g \in \vG\}$. 
$\a$ gives an isomorphism $(\wt Y_{\g})_{-s + \SV} \isom \wt Y\flt_{\SV, \g}$ for 
each $\g \in \vG$.  $W_s$ acts naturally on $(\wt Y_{\g_0})_{-s + \SV}$, 
and the map $(\wt Y_{\g_0})_{-s + \SV} \to \wt Y\flt_{\SV, \g_0}$ is $W_s$-equivariant.
\par
Now (3.8.6) and (3.8.7) imply that 
\begin{equation*}
\tag{3.8.8}
\b^*\bigl((\psi\flt)_!\Ql\bigr)|_{\SV} \simeq 
                 \bigoplus_{\g \in \vG}(\psi_{\g})_!\Ql|_{-s + \SV}
         \simeq \Ind_{W_s}^W \left((\psi_{\g_0})_!\Ql|_{-s + \SV}\right)
\end{equation*}
as local systems equipped with $W$-action.  Here we consider $(\psi_{\g_0})_!\Ql$ as 
a local system with $W_s$-action.
Put $K = (\pi\flt)_!\Ql$, and $K_{\g} = (\pi_{\g})_!\Ql$. Since 
$K|_{Y\flt} = (\psi\flt)_!\Ql$, $K_{\g}|_{Y_{\g}} \simeq (\psi_{\g})_!\Ql$,  
(3.8.8) can be rewritten as 
\begin{equation*}
\b^*(K|_{\SV}) \simeq \bigoplus_{\g \in \vG}K_{\g}|_{-s + \SV}.
\end{equation*}
Note that $\SV$ is an open subset of $s + \SU$.  Then the above isomorphism can be
extended to an isomorphism 
\begin{equation*}
\tag{3.8.9}
\b^*(K|_{s + \SU}) \simeq \bigoplus_{\g \in \vG}(K_{\g}|_{\SU}). 
\end{equation*}  
$K|_{s + \SU}$ has a natural $W$-action induced from the $W$-action on $(\psi\flt)_!\Ql$, 
which coincides with the $W$-action restricted from the $W$-action on $K$. Similarly, 
the $W_s$-action on $K_{\g_0}|_{\SU}$ induced from that on $(\psi_{\g_0})_!\Ql$ coincides with 
the $W_s$-action restricted from that on $K_{\g_0}$.  Thus (3.8.9) can be written, as
complexes with $W$-action, 
\begin{equation*}
\b^*(K|_{s + \SU}) \simeq \Ind_{W_s}^W (K_{\g_0}|_{\SU}).
\end{equation*}
Now taking the stalk at $s + z \in s + \SU$ of the $i$-th cohomology sheaf on both sides, 
we have an isomorphism of $W$-modules. 
\begin{equation*}
\SH^i_{s + z}K \simeq \Ind_{W_s}^W (\SH^i_z K_{\g_0}).
\end{equation*} 
The theorem follows from this.
\end{proof}

\para{3.9.}
Recall that $Y$ is an open dense subset of $\Fh$.  We regard 
$Y$ as a locally closed subset of $\wt\Fh$.  We consider the 
restriction $(\pi\flt)_!\Ql|_Y$ of $(\pi\flt)_!\Ql$ on $Y$, which 
inherits the $W$-module structure from $(\pi\flt)_!\Ql$.  
By applying Theorem 3.8, we shall prove the following result.

\begin{prop}  
$(\pi\flt)_!\Ql|_Y$ is isomorphic to $\psi_!\Ql$ as complexes 
equipped with $W$-action. 
In particular, for $\r \in (S_k \times S_{n-k})\wg$, we have
\begin{equation*}
\tag{3.10.1}
\IC(\wt\Fh, \SL\flt_{\wh\r})|_Y \simeq \IC(Y_k, \SL_{\r})[d_k - d_n].
\end{equation*}
\end{prop}

\begin{proof}
The isomorphism $(\pi\flt)_!\Ql|_Y \simeq \psi_!\Ql$ follows from 
the base change theorem.  We show that this isomorphism is compatible 
with $W$-action.  
For $0 \le k \le n$, we can consider $Y_k^0$ as a locally closed 
subvariety of $\wt\Fh$. 
By the base change theorem, we have $(\pi\flt)_!\Ql|_{Y_k^0} \simeq (\psi_k)_!\Ql$. 
$(\pi\flt)_!\Ql|_{Y_k^0}$ has a $W $-action inherited from that on $(\pi\flt)_!\Ql$. 
We compare this $W$-action with that of $(\psi_k)_!\Ql$.  For this, we investigate  
the $W$-module structure of the stalk at $x \in Y^0_k$ of both cohomology sheaves.   
We apply Theorem 3.8 in the  case where $x = s + z \in \Ft_{sr} + \FD^0_k$.
In this case, $C = Z_{\wt H}^0(s) \simeq \wt Z (SL_2 \times \cdots \times SL_2)$
($n$-factors) under the notation in 2.6, 
and $\Fc \simeq \wt\Fz + (\Fs\Fl_2 \oplus \cdots \oplus \Fs\Fl_2)$.
$z \in \Fc\nil$ is written as $z = \sum_{i=1}^nz_i$, where $z_i \in (\Fs\Fl_2)\nil$  
is such that $z_i \ne 0$ for $1 \le i \le k$ and $z_i = 0$ for $i > k$. 
Thus $\SB^C \simeq \BP_1^n$ and $\SB^C_z \simeq \BP_1^{n-k}$. 
We have
\begin{equation*}
\tag{3.10.2}
H^{\bullet}(\BP_1^{n - k}) = (\Ql[-2] \oplus \Ql)^{\otimes (n-k)}
  \simeq \bigoplus_{J \subset [k+1,n]}\Ql[-2|J|].
\end{equation*}
Since $W_s$ is the Weyl group of $C$, we have 
$W_s \simeq (\BZ/2\BZ)^n$, and the Springer 
representation of $W_s$ on $H^{\bullet}(\SB^C_z)$ is given by 
$\vf_J$ on each factor $\Ql[-2|J|]$, where $\vf_J$ is a one-dimensional 
representation of $(\BZ/2\BZ)^n$ such that the $i$-th factor $\BZ/2\BZ$ acts 
non-trivially for $i \in J$, and acts trivially for $i \in [1,n] - J$.  
It follows, for $|J| = k'$, that 
\begin{equation*}
\tag{3.10.3}
\Ind_{W_s}^W\vf_J = \bigoplus_{\r \in (S_{n - k'}\times S_{k'})\wg}
                         \wh\r \otimes \Ql^{\dim \r}.
\end{equation*}
In particular, by applying $J = [k+1, n]$, we have
\begin{equation*}
\tag{3.10.4}
((\pi\flt)_!\Ql)_x \simeq H^{\bullet}(\SB_x, \Ql) \simeq 
             \bigoplus_{\r \in (S_k \times S_{n-k})\wg}
                 \wh\r \otimes \Ql^{\dim \r}[-2(n-k)] \oplus \SN_x, 
\end{equation*}
where $\SN_x$ is a sum of complexes $\Ql[-2i]$ with $i < n-k$.  
\par
On the other hand, by taking the stalk at $x \in \Fh$ in both sides of (2.13.2),
and by taking into account the action of $W$ given in (2.10.5), 
we have 
\begin{equation*}
\tag{3.10.5}
((\psi_k)_!\Ql)_x \simeq \bigoplus_{\r \in (S_k \times S_{n-k})\wg}
                            \wh\r \otimes \Ql^{\dim \r}[-2(n-k)] + \SN_x',
\end{equation*}
where $\SN_x'$ is a sum of complexes of the form $\Ql[-2i]$ with $i < n - k$. 
By comparing (3.10.4) and (3.10.5), we obtain the following.
\par\medskip\noindent
(3.10.6) \ The $W$-module structure of $(\psi_k)_!\Ql$ coincides with 
the $W$-module structure of $(\pi\flt)_!\Ql|_{Y^0_k}$, up to a sum 
of $\SL[-2i]$ with $i <  n - k$ for various local systems $\SL$ on $Y^0_k$.   
\par\medskip
Now the proof of Proposition 2.14 shows that the $W$-module structure  of $\psi_!\Ql$ 
is completely determined from the $W$-module structure of $(\psi_k)_!\Ql$ ignoring the
part $\SN_k$.  A similar discussion holds also for $(\pi\flt)_!\Ql$, the $W$-module 
structure of $(\pi\flt)_!\Ql|_Y$ is completely determined from that of 
$(\pi\flt)_!\Ql|_{Y_k^0}$ ignoring the part $\SL[-2i]$ with $i < n-k$. This proves 
the first assertion of the proposition. (3.10.1) then follows by comparing 
(2.14.1) and (3.4.1).  The proposition is proved.   
\end{proof}
\par\bigskip
\section{ The variety of semisimple orbits }

\para{4.1.}
Recall the map $\pi : \wt X \to X = \Fh$ as in 2.4.
In this section, we consider $\SB$ as $H/B$ (not as $\wt H/\wt B$), 
and for each $x \in \Fh$ 
put $\SB_x = \{ gB \in H/B \mid g\iv x \in \Fb \}$.
Then $\SB_x \simeq \pi\iv(x)$. 
Let $\SO_x$ be the $H$-orbit of $x$ in $\Fh$.  
Put $\nu_H = \dim U$.  We have the following.
\begin{equation*}
\tag{4.1.1}
\dim \SB_x \le \nu_H - \frac{1}{2}\dim \SO_x 
         = \frac{1}{2}(\dim Z_H(x) - \dim T).
\end{equation*}   
\par
In fact, this formula was proved by Xue [X1, Lemma 6.1] in the case 
where $x \in \Fh$ is nilpotent. 
Although he assumes that the group is of adjont type, the proof 
works for $H$ and $\Fh$.  
In the general case, write $x = s + z$, where $s$: semisimple, 
$z$: nilpotent such that $[s,z] = 0$.  Put $C  = Z^0_H(s)$, and 
consider the variety $\SB^C_z$ defined similarly to $\SB_x$, where 
$\SB^C$ is the flag variety for $C$, and 
$z \in \Lie C$ is nilpotent. By a similar discussion as in (3.7.1), 
we have $\dim \SB_x = \dim \SB^C_z$. 
Then (4.1.1) follows from the corresponding formula for the nilpotent case.

\para{4.2.}
In this section we put $W = N_H(T)/T \simeq W_n$. 
For $x \in \Fh$, let $x = x_s + x_n$ be the Jordan decomposition of $x$, 
where $x_s$ is semisimple, $x_n$ is nilpotent such that $[x_s, x_n] = 0$. 
By Lemma 2.3, the set of semisimple orbits in $\Fh$ is 
in bijection with $\Xi = \Ft/S_n$, 
under the natural action of $S_n \subset W_n = W $ on $\Ft$. 
Since $\Xi$ is identified with the set of closed orbits in $\Fh$, 
the Steinberg map $\w : \Fh \to \Xi$ is defined by associating the 
$H$-orbit of $x_s$ for $x$ (see [X1, 6.4]). 
\par   
For a semisimple element $s \in \Fh$, let 
$V = V_1\oplus V_2 \oplus \cdots \oplus V_a$ be the 
eigenspace decomposition of $s$, with $\dim V_i = n_i$ : even. 
Then $Z_H(s) \simeq Sp_{n_1} \times \cdots \times Sp_{n_a}$. 
Put $\Bn = (n_1, \dots, n_a)$. Then $\Bn$ determines the structure of 
$Z_H(s)$ up to isomorphism, which we denote by $C(\Bn)$.  Let $\SO$ 
be a nilpotent orbit in $\Lie C(\Bn)$. We define 
\begin{equation*}
X_{\Bn, \SO} = \{ x \in \Fh \mid x = x_s + x_n,  
                       Z_H(x_s) \simeq C(\Bn), x_n \in \SO \}.
\end{equation*}  
Then $X_{\Bn,\SO} = X_{\Bn',\SO'}$ if they are not disjoint, and 
$\Fh = \bigcup_{\Bn, \SO}X_{\Bn,\SO}$ gives a partition of $\Fh$. 
By considering the Steinberg map, for $x \in X_{\Bn,\SO}$, we have 
\begin{equation*}
\dim X_{\Bn,\SO} = \dim \SO_x + \dim \{ s \in \Ft \mid Z_H(s) \simeq C(\Bn) \}.
\end{equation*}
In particular, for $x \in X_{\Bn,\SO}$, 
\begin{equation*}
\tag{4.2.1}
\dim X_{\Bn,\SO} \le \dim \SO_x + \dim \Ft. 
\end{equation*}
We have a lemma.
\begin{lem}  
The map $\pi : \wt X \to X$ is semismall. 
\end{lem}
\begin{proof}
We know that $\wt X$ is smooth and $\pi$ is proper. 
For $x \in X_{\Bn,\SO}$, by (4.1.1) and (4.2.1), 
\begin{align*}
\dim \SB_x &\le \frac{1}{2}\bigl(\dim X - (\dim \SO_x + \dim \Ft)\bigr) \\ 
           &\le \frac{1}{2}(\dim X - \dim X_{\Bn,\SO}).
\end{align*}
The lemma follows. 
\end{proof}

\para{4.4.}
We consider the Steinberg variety
\begin{equation*}
Z = \{ (x, gB, g'B) \in \Fh \times \SB \times \SB 
          \mid g\iv x\in \Fb, {g'}\iv x \in \Fb \}. 
\end{equation*}
We denote by $\vf : Z \to \Fh$ the projection $(x, gB, g'B) \mapsto x$. 
We have a commutative diagram

\begin{equation*}
\begin{CD}
Z @>\a>> \Ft  \\
@V\vf VV    @VV\w_1 V  \\
\Fh  @> \w>>  \Xi,
\end{CD}
\end{equation*}
where $\a : (x, gB, g'B) \mapsto p_1(g\iv x)$ 
($p_1$ is the projection $\Fb \to \Ft$), 
$\w_1$ 
is the restriction of $\w$ on $\Ft$. 
Note that $\w_1$ is a finite morphism. 
Put $\s = \w_1\circ \a$, and $d' = \dim \Fh - \dim \Ft$.
We define a constructible sheaf $\ST$ on $\Xi$ by 
\begin{equation*}
\tag{4.4.1}
\ST = \SH^{2d'}(\s_!\Ql) = R^{2d'}\s_!\Ql.
\end{equation*}
\par
Recall the notion of perfect sheaves in [L1, (5.4.4)].
A constructible sheaf $\SE$ on an irreducible variety is called 
a perfect sheaf if it satisfies the following two condition;
(a) there exists an open dense smooth subset $V_0$ 
of $V$ such that $\SE|_{V_0}$ is locally constant, and that 
$\SE = \IC(V, \SE|_{V_0})$, (b) the support of any non-zero 
constructible subsheaf of $\SE$ has support $V$. 
\par
Perfect sheaves enjoy the following 
properties; if $\pi: V' \to V$ is a finite morphism with 
$V'$ smooth, and $\SE'$ is a locally constant sheaf on $V'$, then 
$\SE = \pi_*\SE'$ is a perfect sheaf.  Moreover, if 
$0 \to \SE_1 \to \SE_2 \to \SE_3 \to 0$ is an exact sequence of 
constructible sheaves on $V$ such that $\SE_1, \SE_3$ are perfect, 
then $\SE_2$ is perfect.  
\par
We have the following lemma.
\begin{lem}  
The sheaf $\ST$ is a perfect sheaf on $\Xi$.
\end{lem}
\begin{proof}
The lemma can be proved by a similar argument as in the proof 
of Theorem 5.5 in [L1].  For the sake of completeness, and 
for fixing the notations, we will give
the proof below.
Let $p : Z \to \SB \times \SB$ be the projection 
$(x, gB, g'B) \mapsto (gB, g'B)$. 
$\SB \times \SB$ is decomposed into $H$-orbits, 
$\SB \times \SB = \coprod_{w \in W }\SO_w$, where 
$\SO_w$ is the $H$-orbit containing $(B, wB)$. 
Put $Z_w = p\iv(\SO_w)$ for each $w \in W$.
$Z_w \to \SO_w$ is a locally trivial fibration 
with fibre isomorphic to $\Fb \cap w\Fb$. 
Let $\a_w$ be the 
restriction of $\a$ on $Z_w$.
Then $\a_w$ is a locally trivial fibration with fibre isomorphic to
\begin{equation*}
\tag{4.5.1}
H \times^{(B \cap wBw\iv)}(\Fn \cap w\Fn),
\end{equation*}
where $\Fn = \Lie U$. 
In particular, $\dim \a_w\iv(s) = d'$ for any $s \in \Ft$. 
Moreover, each fibre is an irreducible variety. 
Let $\s_w$ be the restriction of $\s$ on $Z_w$, and put 
$\ST_w = \SH^{2d'}((\s_w)_!\Ql)$. 
It follows from the above remark that 
\begin{equation*}
R^{2d'}(\a_w)_!\Ql \simeq \Ql.
\end{equation*}
Since $\w_1$ is a finite morphism, we have 
\begin{equation*}
\tag{4.5.2}
R^{2d'}(\s_w)_!\Ql \simeq R^0(\w_1)_!R^{2d'}(\a_w)_!\Ql \simeq (\w_1)_!\Ql.
\end{equation*}
It follows that $\ST_w$ is a perfect sheaf since $\w_1$ is finite. 
By (4.5.1), $\a_w\iv(s)$ is a vector bundle over $\SO_w$, and $\SO_w$ is a vector 
bundle over $\SB$. It follows that $H^i_c(\a_w\iv(s),\Ql) = 0$ for odd $i$. 
This implies that $R^i(\s_w)_!\Ql = 0$ for odd $i$. 
Now we have a filtration $Z = \coprod_{w \in W}Z_w$ by locally closed subvarieties 
$Z_w$. For an integer $m \ge 0$, let $Z_m$ be the union of $Z_w$ such that $\dim \SO_w = m$, 
and put $Z_{\le m} = \coprod_{m' \le m}Z_{m'}$.  Then $Z_{\le m}$ is closed in $Z$, and 
$Z_m$ is open in $Z_{\le m}$. Let $\s_m$ be the restriction of $\s$ on $Z_m$, and 
define $\s_{\le m}$ similarly. 
Since $R^i(\s_w)_!\Ql = 0$ for odd $i$, we have 
$R^i(\s_m)_!\Ql = 0$ for odd $i$.  Thus we have an exact sequence 
\begin{equation*}
\begin{CD}
0 @>>> R^{2d'}(\s_m)_!\Ql @>>> R^{2d'}(\s_{\le m})_! @>>> R^{2d'}(\s_{\le m-1})_!\Ql 
                            @>>>  0.
\end{CD}
\end{equation*}
Since $R^{2d'}(\s_m )_!\Ql = \bigoplus_{\dim \SO_w = m}\ST_w$ 
is a perfect sheaf on $\Xi$, by induction on $m$, we see that $R^{2d'}\s_!\Ql$ 
is a perfect sheaf. The lemma is proved.   
\end{proof}

\begin{prop} 
$\ST \simeq \bigoplus_{w \in W}\ST_w$ as sheaves on $\Xi$.  
\end{prop}
\begin{proof}
Recall $\Ft_{\sr}$ in (2.2.2), 
and put $\Xi_{\sr} = \w_1(\Ft_{\sr})$. 
Then $\Xi_{\sr}$ is an open dense subset of $\Xi$. 
Since $\ST$ and $\bigoplus_{w \in W}\ST_w$ are perfect sheaves on $\Xi$, 
it is enough to show that their restrictions on $\Xi_{\sr}$ are isomorphic.
Put $Z_0 = \s\iv(\Xi_{\sr})$.  Then 
$Z_0 \simeq \wt Y \times_Y \wt Y$, where 
$\psi : \wt Y \to Y$ is as in 2.4.
Let $\s_0$ be the restriction of $\s$ on $Z_0$, which is the composite 
of the natural map $Z_0 \to Y$ with the map $Y \to \Xi_{\sr}$. 
(Note that for any $s \in \Ft_{\sr}$, 
$\w\iv(\w_1(s)) = \bigcup_{g \in H}g(s + \FD) = Y$). 
The restriction of $\ST$ on $\Xi_{\sr}$ is isomorphic to $R^{2d'}(\s_0)_!\Ql$. 
Recall that $\wt Y^+_k = \coprod_I\wt Y_I$ for $0 \le k \le n$ in (2.8.1).
For any $I \subset [1,n]$ such that $|I| = k$, we consider the map 
$\psi_I : \wt Y_I \to Y_k^0$ as in 2.7. 
Let $\wh Y_I$ be as in (2.7.3).  Then $\psi_I$ is decomposed as 
$\psi_I = \eta_I\circ \xi_I$, where $\eta_I$ is a finite Galois covering with 
Galois group $\SW_I$, and $\xi_I$ is a locally trivial fibration with fibre isomorphic 
to $\BP_1^{I'}$.  
For each subset $I, J$ of $[1,n]$ such that $|I| = |J| = k$, put 
$\wt Z_{IJ} = \wt Y_I \times_{Y_k^0}\wt Y_J$ under the inclusion $Y^0_k \hra Y$.    
We have a partition $Z_0 = \coprod_{I,J}\wt Z_{IJ}$ by locally closed subsets $\wt Z_{IJ}$.
We define $\wh Z_{IJ} = \wh Y_I \times_{Y_k^0}\wh Y_j$. 
The natural map $\vf_{IJ} : \wt Z_{IJ} \to Y_k^0$ is decomposed as 
$\vf_{IJ} = \eta_{IJ}\circ \xi_{IJ}$, where $\eta_{IJ} : \wh Z_{IJ} \to Y_k^0$ 
is a finite Galois covering with Galois group $\SW_I \times \SW_J$, and 
$\xi_{IJ} : \wt Z_{IJ} \to \wh Z_{IJ}$ is a locally trivial fibration with fibre 
isomorphic to $\BP_1^{I'} \times \BP_1^{J'}$. 
\par
Let $\a_0 : Z_0 \to \Ft$ be the restriction of $\a$ on $Z_0$, and 
$\a_0^w$ the restriction of $\a_0$ on $Z_0 \cap Z_w$. 
For $s \in \Ft_{\sr}$, we have a partition of $\wt Z_{IJ}$,
\begin{equation*}
\wt Z_{IJ} \cap \a_0\iv(s) = \coprod_{w \in W}\bigl(\a_0^w)\iv(s) \cap \wt Z_{IJ}\bigr).
\end{equation*}
By using the property of $\vf_{IJ} = \xi_{IJ}\circ \eta_{IJ}$ mentioned above, 
we see that $(\a^w_0)\iv(s) \cap \wt Z_{IJ}$ is an open and closed subset of 
$\wt Z_{IJ}$ for any $w \in W$. 
Moreover, the odd cohomology of $(\a^w_0)\iv(s) \cap \wt Z_{IJ}$ vanishes. 
It follows that 
\begin{equation*}
\tag{4.6.1}
(\a_{IJ})_!\Ql \simeq \bigoplus_{w \in W }(\a^w_{IJ})_!\Ql, 
\end{equation*} 
where $\a_{IJ}$ is the restriction of $\a_0$ on $\wt Z_{IJ}$, and  
$\a^w_{IJ}$ is the restriction of 
$\a_0^w$ on $Z_w \cap \wt Z_{IJ}$.
Moreover, we have $R^{i}(\a^w_{IJ})_!\Ql = 0$ for odd $i$.  
By considering the long exact sequence arising from the filtration 
$Z_0 = \coprod_{I,J}\wt Z_{IJ}$, (4.6.1) implies that
\begin{equation*}
\tag{4.6.2}
R^{2d'}(\a_0)_!\Ql \simeq \bigoplus_{w \in W}R^{2d'}(\a^w_0)_!\Ql.
\end{equation*}
By applying $R^0(\w_1)_!$ on both sides of (4.6.2), we have
\begin{equation*}
R^{2d'}(\s_0)_!\Ql \simeq \bigoplus_{w \in W } R^{2d'}(\s^w_0)_!\Ql, 
\end{equation*}
where $\s^w_0$ is the restriction of $\s_0$ on $Z_0 \cap Z_w$. 
This shows that the restrictions of $\SF$ and of $\bigoplus_w\SF_w$  
on $\Xi_{\sr}$ are isomorphic.  The proposition is proved.   
\end{proof}

\para{4.7.}
By the K\"unneth formula, we have $\vf_!\Ql \simeq \pi_!\Ql \otimes \pi_!\Ql$. 
Since $(\pi\flt)_!\Ql$ is a complex with $W$-action by Proposition 3.4, 
$\pi_!\Ql = (\pi\flt)_!\Ql|_{\Fh}$ has the action of $W$ inherited 
from $(\pi\flt)_!\Ql$. Hence $\vf_!\Ql$ has 
a natural action of $W \times W $. 
It follows that $\ST = \SH^{2d'}(\s_!\Ql) \simeq \SH^{2d'}(\w_!\vf_!\Ql)$ 
is a sheaf equipped with $W \times W$-action. 
We note that under the decomposition of $\ST$ in Proposition 4.6, the action of 
$W \times W$ has the following property; for each $w_1, w_2 \in W$, 
\begin{equation*}
\tag{4.7.1}
(w_1, w_2)\cdot \ST_w  = \ST_{w_1ww_2\iv}.
\end{equation*}

In fact, since $\ST$ is a perfect sheaf by Lemma 4.5, it is enough to check 
the relation (4.7.1) for the restriction of $\ST$ on $\Xi_{\sr}$. 
The action of $\SW_I \times \SW_J$ on $(\vf_{IJ})_!\Ql$ is extended to the action of 
$\wt\SW _I\times \wt\SW_J$ (here $\wt \SW_I = \SW_I\ltimes (\BZ/2\BZ)^n$) 
so that the $i$-th factor $\BZ/2\BZ$ of $\wt\SW_I$ acts trivially if 
$i \in I$ and acts non-trivially if $i \in I'$, and similarly for $\wt\SW_J$. 
This action induces an action of $W \times W$ on $(\vf_0)_!\Ql$, which 
is nothing but the action of $W \times W$ inherited from the action of 
$W$ on $\pi_!\Ql$ by Proposition 3.10. 
Then a similar relation as (4.7.1) for $(\vf_{IJ})_!\Ql$ under the decomposition 
\begin{equation*}
(\vf_{IJ})_!\Ql \simeq \bigoplus_{w \in W}(\vf^w_{IJ})_!\Ql,
\end{equation*}
where $\vf^w_{IJ}$ is the restriction of $\vf_{IJ}$ on $\wt Z_{IJ} \cap Z_w$,
can be verified directly by using the decomposition $\vf_{IJ} = \eta_{IJ}\circ \xi_{IJ}$.
Thus (4.7.1) is proved. 
\par
We consider the cohomology group $H^{2n}_c(\Xi, \ST)$. 
The following fact holds.

\begin{prop}  
$H^{2n}_c(\Xi,\ST)$ has a structure of $W \times W$-module, which is isomorphic to 
the two-sided regular representation of $W$.  
\end{prop}

\begin{proof}
Since $\SF$ is a sheaf with $W \times W$-action, $H^i_c(\Xi, \SF)$
has a structure of 
$W \times W$-module. By Proposition 4.6, we have a decomposition 
\begin{equation*}
H^{2n}_c(\Xi,\ST) \simeq \bigoplus_{w \in W}H^{2n}_c(\Xi, \ST_w). 
\end{equation*}
By (4.5.2) 
\begin{equation*}
H^{2n}_c(\Xi,\ST_w) \simeq H^{2n}_c(\Xi, (\w_1)_!\Ql) \simeq 
                      H^{2n}_c(\Ft, \Ql) = \Ql
\end{equation*}   
since $\dim \Ft = n$. The proposition then follows from (4.7.1). 
\end{proof}

The following lemma, originally due to Lusztig [L1, Lemma 6.7], was proved 
in [Sh, Lemma 7.6].   Note that in [Sh] it is stated for the unipotent variety, 
but it works for any variety.  Actually this result is proved in [L2, (7.4.2)], in 
a full generality.  

\begin{lem} 
Let $A, A'$ be simple perverse sheaves on $X = \Fh$.  Then we have
\begin{equation*}
\dim \BH^0_c(X, A \otimes A') = 
       \begin{cases}
          1 &\quad\text{ if } A' \simeq D(A), \\
          0 &\quad\text{ otherwise, }
       \end{cases}
\end{equation*}
where $D(A)$ is the Verdier dual of $A$. 
\end{lem}

\para{4.10.}
Recall that $\pi : \wt X \to X$ is semismall by Lemma 4.3, 
and so $K = \pi_!\Ql[d]$ is a semisimple 
perverse sheaf on $X = \Fh$, where $d = \dim X = \dim \Fh$.
We can write it as 
\begin{equation*}
\tag{4.10.1}
K = \bigoplus_A V_A\otimes A,
\end{equation*}
 where $A$ is a simple perverse sheaf 
and $V_A = \Hom (K, A)$ is the multiplicity space for $A$. 
We have the following.

\begin{prop}  
Put $m_A = \dim V_A$ for each $A$. Then we have
\begin{equation*}
\sum_Am^2_A = |W |.
\end{equation*}
\end{prop}   

\begin{proof}
We have  
\begin{align*}
H^{2n}_c(\Xi, \ST) = H^{2n}_c(\Xi, R^{2d'}\w_!(\pi_!\Ql\otimes\pi_!\Ql)). 
\end{align*}
Consider the spectral sequence
\begin{equation*}
H^i_c(\Xi, R^j\w_!(\pi_!\Ql\otimes\pi_!\Ql)) \Longrightarrow
                \BH^{i+j}_c(X, \pi_!\Ql\otimes\pi_!\Ql).
\end{equation*}
Since $\dim X = d' + n = d$, $\dim \Xi = n$, we have
\begin{align*}
H^{2n}_c(\Xi, \ST) &\simeq \BH^{2d}_c(X, \pi_!\Ql\otimes\pi_!\Ql) \\
                   &\simeq \BH^0_c(X, K\otimes K).
\end{align*}
Hence by (4.10.1), we have
\begin{equation*}
\dim H^{2n}_c(\Xi,\ST) = \sum_{A,A'}m_Am_{A'}\dim \BH^0_c(X, A\otimes A').
\end{equation*} 
By Lemma 4.9, $\BH^0(X, A\otimes A') \ne 0$ only when $D(A) \simeq A'$, in which case, 
we have $\dim \BH^0(X, A\otimes A') = 1$. 
But since $K$ is self-dual, $m_A = m_{D(A)}$ for each $A$. It follows that 
$\dim H^{2n}_c(\Xi,\ST) = \sum_Am_A^2$. 
On the other hand, by Proposition 4.8, $\dim H^{2n}_c(\Xi,\ST) = |W |$. 
The proposition is proved.  
\end{proof}

\par\bigskip
\section{ Intersection cohomology on $\Fs\Fp(V)$ } 

\para{5.1.}
In Section 3, we have considered the intersection cohomologies on 
$\wt\Fh$.  In this section, we consider their restriction on $\Fh$. 
We follow the notation in 2.2.  In particular, 
$\Fn_s = \bigoplus_{\a \in \Phi^+_s}\Fg_{\a}$ and 
$\FD = \bigoplus_{\a \in \Phi^+_{l}}\Fg_{\a}$ are subspaces of $\Fh$
such that $\Fn = \Fn_s \oplus \FD$. 
Put $\FN_0 = \bigcup_{g \in B}g(\Ft)$.
Let $\ol\FN_0$ be the closure of $\FN_0$ in $\Fh$. 
We show the following lemma. 

\begin{lem}  
\begin{enumerate}
\item
$\ol\FN_0 = \Ft \oplus \Fn_s$.  In particular, $\Ft\oplus \Fn_s$ is $B$-stable.
\item
Let $\Fh_{ss}$ be the set of semisimple elements in $\Fh$.  Then 
\begin{equation*}
\tag{5.2.1}
\ol\Fh_{ss} = \bigcup_{g \in H}g(\Ft\oplus \Fn_s).
\end{equation*}
Moreover, $\dim \ol\Fh_{ss} = \dim \Fh - 2n$. 
\end{enumerate}
\end{lem}

\begin{proof}
First we show that 
\begin{equation*}
\tag{5.2.2}
\bigcup_{g \in B}g(\Ft) \subset \Ft \oplus \Fn_s.
\end{equation*}
Any $g \in B$ can be written by (1.10.1) as
\begin{equation*}
g = \begin{pmatrix}
            b  &   c   \\
            0  &   {}^tb\iv
    \end{pmatrix},
\end{equation*}
where $b,c$ are square matrices of degree $n$, with $b$ non-singular
upper triangular, 
and $c$ satisfies the condition that ${}^tc = b\iv c\,{}^tb$.  
Then $g\iv$ can be written as
\begin{equation*}
g\iv = \begin{pmatrix}
         b\iv  &  -b\iv c\, {}^tb  \\
          0    &    {}^t b
       \end{pmatrix}
      = \begin{pmatrix}
           b\iv &   {}^tc  \\
           0    &   {}^tb
         \end{pmatrix}.
\end{equation*}
Thus, for a diagonal matrix $s$ of degree $n$, we have
\begin{equation*}
   x =  g\begin{pmatrix}
              s   &   0  \\
              0   &   s
      \end{pmatrix}g\iv = 
             \begin{pmatrix}
                bs b\iv  &  bs\,{}^tc + cs\,{}^tb \\
                 0      &  {}^tb\iv s\, {}^tb
             \end{pmatrix}. 
\end{equation*}
Since $bs\,{}^tc + cs\,{}^tb$ is of the form $A + {}^tA$ for a squar matrix $A$, 
its diagonal entries are all zero, hence $x \in \Ft \oplus \Fn_s$. (5.2.2) holds.  
\par
Recall that, for $k = 0$,  $\wt Y_0 \simeq H \times^B\FN_{0,\sr}$, where 
$\FN_{0,\sr} = \bigcup_{g \in B}g(\Ft_{\sr})$.  Hence 
$\dim \wt Y_0 = \dim H - \dim B + \dim \FN_{0,\sr}$.  
Since $\dim \wt Y_0 = \dim H - n$ by Lemma 2.9 (iii), 
we see that 
$\dim \FN_{0,\sr} = \dim B - n$. 
We have 
\begin{equation*}
\FN_{0,\sr} \subset \FN_0 \subset \Ft \oplus \Fn_s
\end{equation*}
by (5.2.2).  Since $\dim \Ft\oplus \Fn_s = \dim B - n$, we have
$\dim \FN_0 = \dim \Ft \oplus \Fn_s$.  Since $\FN_0$ is irreducible, 
(i) holds. 
\par
Put $\wt X_0 = H \times^B(\Ft\oplus\Fn_s)$ and 
$X_0 = \bigcup_{g \in H}(\Ft\oplus\Fn_s)$. 
The map $\pi^{(0)} : \wt X_0 \to \Fh$ is proper and $\Im\pi^{(0)} = X_0$. Hence 
$X_0$ is a closed subset of $\Fh$. 
Recall that $Y_0 = \bigcup_{g \in H}g(\Ft_{\sr})$, then 
$Y_0 \subset \bigcup_{g \in G}g(\Ft) \subset X_0$.  
Since $Y_0$ is open dense in $\Fh_{ss}$, $\ol Y_0 = \ol \Fh_{ss}$. Hence 
$\ol\Fh_{ss} \subset X_0$. 
On the other hand, $\bigcup_{g \in B}g(\Ft) = \FN_0$ is contained in 
$\Fh_{ss}$.  Hence $\Ft\oplus \Fn_s \subset \ol\Fh_{ss}$ by (i), 
and so $X_0 \subset \ol\Fh_{ss}$.  
It follows that $X_0 = \ol\Fh_{ss}$.  (5.2.1) is proved.
The last assertion follows from Lemma 2.9 (iv) since $\dim \ol\Fh_{ss} = \dim Y_0$. 
\end{proof}

\remark{5.3.}
In the case of reductive Lie algebras $\Fg$ with $p \ne 2$, clearly the closure of 
$\bigcup_{g \in B}g(\Ft)$ coincides with $\Fb$, and $\ol\Fg_{ss} = \Fg$ holds. 
Lemma 5.2 gives a special phenomenon occurring in the case where $p = 2$ and 
regular semisimple elements do not exist.  
\para{5.4.}
We shall generalize Lemma 5.2 in connection with $\FN_{k,\sr}$. 
For $k = 0,1, \dots, n$, put
\begin{equation*}
\tag{5.4.1}
\FN_k = \bigcup_{g \in B}g(\Ft + \FD_k).
\end{equation*}
Note that for $k = 0$, $\FN_k$ coincides with the previous notation.
Let $\ol\FN_k$ be the closure of $\FN_k$ in $\Fh$.
We define varieties 
\begin{align*}
\tag{5.4.2}
\wt X_k &= \{(x, gB) \in \Fh \times H/B \mid g\iv x \in \ol\FN_k \}, \\
    X_k &= \bigcup_{g \in H}g(\ol\FN_k), 
\end{align*}
and define a map $\pi^{(k)}: \wt X_k \to \Fh$ by 
$(x, gB) \mapsto x$.  Then $\pi^{(k)}$ is proper, and 
$\Im \pi^{(k)} = X_k$.  Hence $X_k$ is closed in $\Fh$.
We show a lemma.

\begin{lem} 
\begin{enumerate}
\item
$\ol\FN_k = \Ft\oplus\Fn_s \oplus \FD_k$ for each $k$.  
In particular, $\Ft\oplus\Fn_s\oplus \FD_k$ is $B$-stable. 
\item
$X_0 = \ol\Fh_{ss}$ and $X_n = \Fh$. In particular, 
the map $\pi^{(n)} : \wt X_n \to X_n$ coincides with the map 
$\pi : \wt X \to X$ given in 2.4. 
\end{enumerate}
\end{lem}

\begin{proof}
First we show 
\begin{equation*}
\tag{5.5.1}
\bigcup_{g \in B}g(\FD_k) \subset 
           \bigoplus_{\substack{\a \in \Phi^+ \\ \a = \ve_i + \ve_j }}\Fg_{\a}
             \oplus \FD_k.
\end{equation*}
\par  
In fact, 
$x \in \FD_k$ is written as $x = \begin{pmatrix}
                                     0  &  d  \\
                                     0  &  0
                                 \end{pmatrix}$, 
where $d = \Diag (d_1, \dots, d_n)$ is a diagonal matrix with $d_i = 0$ 
for $i > k$. 
Thus under the notation in the proof of Lemma 5.2, for $g \in B$, we have
\begin{equation*}
gxg\iv = \begin{pmatrix}
             0   &   bd\, {}^tb  \\
             0   &   0
         \end{pmatrix} \in \bigoplus_{\a = \ve_i + \ve_j}\Fg_{\a} \oplus \FD.
\end{equation*}    
Put $y = (y_{ij}) = bd\, {}^tb$, and $b = (b_{ij})$.  Then 
$y_{ii} = \sum_{j \ge i}b_{ij}^2d_j$. 
It follows that $y_{ii} = 0$ for $i > k$, and (5.5.1) holds.
\par
Since we know $\bigcup_{g \in B}g(\Ft) \subset \Ft\oplus\Fn_s$ by Lemma 5.2, 
we have
\begin{equation*}
\tag{5.5.2}
\FN_k = \bigcup_{g \in B}g(\Ft + \FD_k) \subset \Ft \oplus \Fn_s \oplus \FD_k.
\end{equation*} 
\par
On the other hand, since $\FN_{k,\sr} \subset \FN_k$ and 
$\wt X_k \simeq H \times^B\ol\FN_k$, $\wt Y_k$ is regarded as a subvariety 
of $\wt X_k$. 
It follows, by (5.5.2), that
\begin{align*}
\tag{5.5.3}
\dim \wt Y_k &\le \dim \wt X_k  \\
             &= \dim H - \dim B + \dim \FN_k \\
             &\le \dim H - \dim B + \dim (\Ft\oplus \Fn_s \oplus \FD_k) \\
             &= \dim H  - n + k.
\end{align*}
We know $\dim \wt Y_k = \dim H - n + k$ by Lemma 2.9. 
Hence the inequalities in(5.5.3) are actually  equalities, 
and we have $\dim \FN_k = \dim \Ft\oplus\Fn_s\oplus\FD_k$. 
Since $\FN_k$ is irreducible, we conclude that 
$\ol\FN_k = \Ft\oplus\Fn_s\oplus\FD_k$. This proves (i). 
\par
For (ii), we know $X_0 = \ol\Fh_{ss}$ by Lemma 5.2.  
Since $\ol \FN_n = \Fb$ by (i), we have $X_n = X$.
Thus (ii) holds. The lemma is proved.
\end{proof}

As a corollary, we have the following.

\begin{prop}  
For $k = 0, 1, \dots, n$, we have
\begin{enumerate}
\item 
$\wt X_k$ is a smooth, irreducible variety. 
\item
$X_k$ is a closed subset of $\Fh$, with 
$X_k = \bigcup_{g \in H}g(\Ft\oplus\Fn_s\oplus\FD_k)$.  
\item
$\wt Y_k$ is open dense in $\wt X_k$, and $Y_k$ is open dense in $X_k$.
\item
$\dim \wt X_k = \dim H - n + k$, 
$\dim X_k = \dim H - 2n + 2k$. 
\end{enumerate}
\end{prop}

\begin{proof}
By Lemma 5.5, $\wt X_k \simeq H \times^B(\Ft\oplus\Fn_s\oplus\FD_k)$. 
Hence $\wt X_k$ is smooth, irreducible.  This proves (i).  (ii) also 
follows from Lemma 5.5.
We have $\wt Y_k \simeq H \times^B\FN_{k,\sr}$, and $\FN_{k,\sr}$ is open dense 
in $\Ft\oplus\Fn_s\oplus\FD_k$.  Hence $\wt Y_k$ is open dense in $\wt X_k$. 
$Y_k$ coincides with the subset of 
$X_k$ consisting of $x \in X_k$ such that its semisimple part 
$x_s$ is contained in $Y_0 = \bigcup_{g \in H}g(\Ft_{\sr})$.  Since $Y_0$ is open dense 
in $\Fh_{ss}$, $Y_k$ is open dense in $X_k$. This proves (iii).
Then (iv) follows from  Lemma 2.9.  
\end{proof}

We shall prove the following theorem.  

\begin{thm}  
$\pi_!\Ql[d_n]$ is a semisimple perverse sheaf on $X = \Fh$, equipped with 
the action of $W_n $, and is decomposed as
\begin{equation*}
\pi_!\Ql[d_n] \simeq \bigoplus_{0 \le k \le n}
                 \bigoplus_{\r \in (S_k \times S_{n-k})\wg}
                    \wh\r \otimes \IC(X_k, \SL_{\r})[d_k]. 
\end{equation*}  
\end{thm} 

\para{5.8.}
As in the case of $Y_k^0$, put $X_k^0 = X_k - X_{k-1}$. 
By (2.6.3), $Y_k^0 \subset X_k^0$.  Hence $Y_k^0$ is open dense in $X_k^0$. 
For each $k$, we define a locally closed subvariety $\wt X^+_k$ of $\wt X$ by 
$\wt X^+_k = \pi\iv(X_k^0)$. 
Let $\pi_k : \wt X^+_k \to X_k^0$ be the restriction of $\pi$ on $\wt X^+_k$. 
\par
Take $s \in \Ft$, and consider the decomposition 
$V = V_1 \oplus \cdots \oplus V_a$ into eigenspaces of $s$, 
with $\dim V_i = 2n_i$.  
Then $Z_H(s) \simeq Sp(V_1) \times \cdots \times Sp(V_a)$.
Put $H_i = Sp(V_i)$ and $\Fh_i = \Fs\Fp(V_i)$. 
We consider the corresponding decomposition 
$\FD = \FD^1 \oplus\cdots \oplus\FD^a$ with $\FD^i \subset \Fh_i$.
The subvariety $X^{H_i,0}_{k_i}$ of $\Fh_i$ is defined similarly to 
$X^0_k$ by replacing 
$H, \Fh, \FD$, etc. by $H_i, \Fh_i, \FD^i$, etc. 
\par
For each $(x, gB) \in \wt X^+_k$, we associate a subset $I \subset [1,n]$
such that $|I| = k$ as follows.  By definition, $g\iv x \in X_k^0 \cap \Fb$.   
In the case where $g\iv x \in \Fn$, put $I = [1,k]$.  
In general, 
by replacing $g\iv x \in \Fb$ by its $B$-conjugate if necessary, we may assume that 
$g\iv x = s + z$, where $s \in \Ft, z \in \Fn$ with $[s,z] = 0$, 
hence $z \in \Lie Z_H(s)$. 
Then $z$ can be written as $z = \sum_iz_i$ with $z_i \in \Fh_i$. 
There exists $k_i$ such that $z_i \in X^{H_i,0}_{k_i}$ for $i = 1, \dots, a$
and that $\sum_{i=1}^ak_i = k$. 
We put $I_i' = [1, k_i] \subset [1, n_i]$, which corresponds to the nilpotent case 
for $H_i$.   
Now $V = V_1\oplus\cdots \oplus V_a$ gives a partition $[1,n] = \coprod_iJ_i$
such that $|J_i| = n_i$. Under the correspondence $J_i \lra [1, n_i]$, 
$I_i'$ gives a subset $I_i \subset J_i$ (the first $k_i$ letters in $J_i$). 
Put $I = \coprod_{i=1}^a I_i$.  Thus $I$ is a subset of $[1,n]$ such that 
$|I| = k$. 
Note that $I$ depends only on the $B$-conjugates of $g\iv x$. Thus 
$I$ is determined by $(x, gB) \in \wt X^+_k$.  We denote this assignment 
by $(x, gB) \mapsto I$.    
For each $I$ with $|I| = k$, we define a subset $\wt X_I$ of $\wt X^+_k$ by 
\begin{equation*}
\tag{5.8.1}
\wt X_I = \{ (x, gB) \in \wt X^+_k \mid (x, gB) \mapsto I \}.
\end{equation*}
We show the following lemma.

\begin{lem}  
$\wt X^+_k$ is decomposed as
\begin{equation*}
\tag{5.9.1}
\wt X^+_k = \coprod_{\substack{I \subset [1,n] \\ |I| = k}} \wt X_I,
\end{equation*}
where $\wt X_I$ is an irreducible component of $\wt X^+_k$ for each $I$. 
\end{lem}

\begin{proof}
It is clear from the definition that $\wt X_I$ are mutually disjoint, and 
gives a partition (5.9.1) of $\wt X^+_k$.
We show that $\wt X_I$ is irreducible.  In fact, 
$X_k \cap \Fh\nil = \bigcup_{g \in H}g(\Fn_s\oplus\FD_k)$ is irreducible. 
The set of $s \in \Ft$ such that the eigenspace decomposition of $V$ gives a
 fixed partition $[1,n] = \coprod_iJ_i$ 
is irreducible.  Hence the set of $x = s + z \in X^0_k \cap \Fb$ with $s$
above is irreducible.  Since $\wt X_I$ is the set of $H$-conjugates of  
$(x, B)$ with $x$ as above, $\wt X_I$ is irreducible.
Now 
$\wt Y_I$ is a subset of $\wt X_I$, which consists of $(x, gB) \mapsto I$ 
such that $g\iv x = s + z$ with $s \in \Ft_{\sr}$.
Thus $\wt Y_I$ is an open dense subset of $\wt X_I$.
Since $\wt Y^+_I = \coprod_i \wt Y_I$, with $\wt Y_I$ irreducible, 
$\wt X^+_k = \bigcup_i\ol{\wt Y_I}$ gives a decomposition into irreducible 
components, where $\ol{\wt Y_I}$ is the closure of $\wt Y_I$ in $\wt X^+_k$.  
In order to prove the lemma, it is enough to see that $\wt X_I$ is closed. 
But the set $Z_I = \bigcup_{I'}\wt X_{I'}$ is closed in $\wt X$, where 
$I'$ runs over all subsets of $[1,n]$ such that $I' \subseteq I$.
Hence $\wt X_I = Z_I \cap \wt X^+_k$ is closed in $\wt X^+_k$.  The lemma is proved. 
\end{proof}

\para{5.10.}
For each $x \in X^0_k$, we associate an $x$-stable isotropic subspace $W_x \subset V$
with $\dim W_x = k$ as follows.  Assume that $x \in X^0_k \cap \Fb$. Up to 
$B$-conjugate, we can write $x = s + z$ with $s \in \Ft, z \in \Fn, [s, z] = 0$.
In the case where $x \in \Fn$, namely $s = 0$, 
let $W_x$ be the subspace of $V$ spanned by $e_1, \dots, e_k$. 
Then $W_x$ is an $B$-stable isotropic subspace of $V$, and is $x$-stable.  
For $x = s + z \in \Fb$ in general, as in the discussion in 5.8, 
$z$ can be written as $z = \sum_{i=1}^az_i$ 
with $z_i \in X^{H_i,0}_{k_i}$ for some $k_i$ such that $0 \le k_i \le n_i$ and that 
$\sum k_i = k$. We define an isotropic subspace $W_i \subset V_i$ for each $i$, 
by applying the above discussion to the nilpotent element $z_i \in \Fh_i$, and 
put $W_x = W_1\oplus \cdots \oplus W_a$. 
Then $W_x$ is an $(B \cap Z_H(s))$-stable isotropic subspace of $V$ with $\dim W_x = k$. 
In general, for $x \in X^0_k$, one can find $g \in H$ such that $g\iv x \in \Fb$, 
and that $g\iv x$ can be written as $g\iv x = x' = s + z$ as above.  If we fix such $s$, 
the choice of $g$ is unique up to $(B \cap Z_H(s))$-conjugate.
We define $W_{x'}$ as above, and put $W_x = g(W_{x'})$.  This $W_x$ satisfies 
the required property.     
\par
$W_x^{\perp}/W_x$ has a natural symplectic structure, and put 
$H_x = Sp(W_x^{\perp}/W_x)$.  
The action of $x$ on $W_x$ induces $x|_{W_x^{\perp}/W_x} \in \Lie H_x$. 
One can check that $x|_{W_x^{\perp}/W_x}$ is contained in $X_0^{H_x}$, where 
$X_0^{H_x}$ is defined similarly to $X_0$ by replacing $H$ by $H_x$. 

\para{5.11.}
For $i = 1, \dots, n$, let $M_i$ be the isotropic subspace of $V$ 
spanned by $e_1, \dots, e_i$. Put $\ol M_i = M_i^{\perp}/M_i$.
$\ol M_i$ has a natural symplectic structure. 
We fix $k$, and consider 
$G_1 = GL(M_k)$, $H_2 = Sp(\ol M_k)$. Also put
$\Fg_1 = \Lie G_1, \Fh_2 = \Lie H_2$. 
Let $X'^0_{k'}$ be the subvariety of $X' = \Fh_2$ defined 
similarly to $X^0_k$ by replacing $H$ by $H_2$.
We consider $X'^0_0 = X'_0$ for $k' = 0$. 
\par
We define a variety $\SG_k$ by
\begin{align*}
\tag{5.11.1}
\SG_k = \{ (x, &\f_1, \f_2) \mid x \in X^0_k, \\
           &\f_1 : W_x \isom M_k, \f_2 : W_x^{\perp}/W_x \isom \ol M_k
                   (\text{ symplectic isom.}) \}.  
\end{align*}

We consider the diagram 
\begin{equation*}
\tag{5.11.2}
\begin{CD}
\Fg_1 \times X'_0 @<\s<< \SG_k @>q>> X^0_k, 
\end{CD}
\end{equation*}
where 
\begin{align*}
q &: (x, \f_1, \f_2) \mapsto x, \\
\s &: (x, \f_1, \f_2) \mapsto (\f_1(x|_{W_x})\f_1\iv, 
                         \f_2(x|_{W_x^{\perp}/W_x})\f_2\iv).
\end{align*}

\par
$H \times (G_1 \times H_2)$ acts on $\SG_k$ by 
\begin{equation*}
(g, (h_1, h_2)) : (x, \f_1, \f_2) \mapsto (gx, h_1\f_1g\iv, h_2\f_2g\iv)
\end{equation*}
for $g \in H, h_1 \in G_1, h_2 \in H_2$. 
Moreover, $\s$ is $H \times (G_1 \times H_2)$-equivariant with respect to 
the adjoint action of $G_1 \times H_2$ and the trivial action of 
$H$ on $\Fg_1 \times X'_0$. 
By a standard argument, one can check 
\par\medskip\noindent
(5.11.3) \ The map $q$ is a principal bundle with fibre isomorphic to 
$G_1 \times H_2$.  The map $\s$ is a locally 
trivial fibration with smooth, connected fibre of dimension $\dim H$. 

\remark{5.12.}
The variety $\SG_k$ introduced here 
is a different type of the variety $\SG_k$ discussed in [SS, 4.5].
The discussion below has some similarly with the discussion in 
[Sh, Section 2]. 

\para{5.13.}
Let $B_1$ be the Borel subgroup of $G_1$ which is the stabilizer of 
the flag $(M_i)_{0 \le i \le k}$ in $G_1$, and $B_2$ the Borel subgroup of 
$H_2$ which is the stabilizer of the flag $(M_{k + i}/M_k)_{k \le i \le n}$ 
in $H_2$. Put 
\begin{equation*}
\wt\Fg_1 = \{ (x, gB_1) \in \Fg_1 \times G_1/B_1 \mid g\iv x \in \Lie B_1 \}, 
\end{equation*}  
and define $\pi^1: \wt\Fg_1 \to \Fg_1$ by $(x,gB_1) \mapsto x$. 
We define $\pi^2: \wt X' \to X' = \Fh_2$ similarly 
to $\pi : \wt X \to X$, by replacing $H$ by $H_2$.  
We put $\wt X'_0 = \wt X'^+_{0} = (\pi^2)\iv(X'_{0})$, and let 
$\pi^2_{0}$ be the restriction of $\pi^2$ on $\wt X'_{0}$. 
We define a variety
\begin{align*}
\tag{5.13.1}
\wt Z^+_k = \{ (x, gB, &\f_1, \f_2) \mid (x, gB) \in \wt X^+_k, \\
                &\f_1 : W_x \isom M_k, \f_2 : W_x^{\perp}/W_x \isom \ol M_k \}
\end{align*}
and define a map $\wt q : \wt Z^+_k \to \wt X^+_k$ by the natural projection.
We define a map $\wt\s : \wt Z^+_k \to \wt\Fg_1 \times \wt X'^+_{0}$ 
as follows; take $(x,gB, \f_1, \f_2) \in \wt Z^+_k$.  
Let $s$ be the semisimple part of $x$.  
Then $B_s = Z_H(s) \cap gBg\iv$ is a Borel subgroup of $Z_H(s)$ such that 
$x \in \Lie B_s$. 
By 5.10, $W_x$ is a $B_s$-stable subspace of 
$V$, and is decomposed as $W_x = W_1 \oplus \cdots \oplus W_a$ 
according to the decomposition $V = V_1 \oplus \cdots \oplus V_a$ into 
eigenspaces of $s$. 
Then $\prod_i(GL(W_i) \cap B_s)$ is a Borel subgroup of $\prod_iGL(W_i)$, and 
there exists a unique Borel subgroup $B^1_x$ of $GL(W_x)$ containing it.
We see that $x|_{W_x} \in \Lie B^1_x$. We denote by $g_1B_1 \in G_1/B_1$ the element 
corresponding to the Borel subgroup $\f_1(B^1_x)\f_1\iv$ of $G_1$.     
On the other hand, $W_i$ is $B_s$-stable, and we have a  homomorphism 
$B_s \to Sp(W_i^{\perp}/W_i)$. If we denote by $B_{s,i}$ 
the image of this map, then  
$\prod_i B_{s,i}$ is a Borel subgroup of $\prod_iSp(W_i^{\perp}/W_i)$, and 
there exists a unique Borel subgroup $B'_x$ of $Sp(W_x^{\perp}/W_x)$ containing it.
We see that $x|_{W_x^{\perp}/W_x} \in B'_x$. 
We denote by $g_2B_2 \in H_2/B_2$ the element corresponding to the Borel subgroup 
$\f_2(B'_x)\f_2\iv$ of $H_2$.  
We now define $\wt\s : \wt Z^+_k \to \wt \Fg_1 \times \wt X'_0$ by 
\begin{equation*}
\wt\s : 
(x, gB, \f_1, \f_2) \mapsto ((\f_1(x|_{W_x})\f_1\iv, g_1B_1),
                             (\f_2(x_{W_x^{\perp}/W_x})\f_2\iv, g_2B_2).
\end{equation*}  

\par
We define a map $\wt\pi_k : \wt Z^+_k \to \SG_k$ by 
$(x, gB, \f_1, \f_2) \mapsto (x, \f_1, \f_2)$. 
Then we have the following commutative diagram extending (5.11.2).
\begin{equation*}
\tag{5.13.2}
\begin{CD}
\wt \Fg_1 \times \wt X'_{0} @<\wt\s << \wt Z^+_k @>\wt q>> \wt X^+_k \\
    @V\pi^1 \times \pi^2_{0}VV   @VV\wt\pi_kV          @VV\pi_kV    \\
  \Fg_1 \times X'_{0}  @<\s<< \SG_k  @>q>>  X^0_k. 
\end{CD}
\end{equation*}

\para{5.14.}
Let $\Fg_{1,\rg}$ be the set of regular semisimple elements in $\Fg_1$.
Let $\psi^1$ be the restriction of $\pi^1: \wt\Fg_1 \to \Fg_1$ to
$(\pi^1)\iv (\Fg_{1,\rg})$. Then $\psi^1$ is a finite Galois covering 
with Galois group $S_k$, and $(\psi^1)_!\Ql$ is decomposed as
\begin{equation*}
(\psi^1)_!\Ql \simeq \bigoplus_{\r_1 \in S_k\wg}\r_1 \otimes \SL^1_{\r_1},
\end{equation*}
where $\SL^1_{\r_1}$ is a simple local system on $\Fg_{1,\rg}$ 
corresponding to $\r_1$.  $\Fg_{1,\rg}$ is an open dense subset of $\Fg_1$, 
and it is well-known that 
$(\pi^1)_!\Ql[\dim \Fg_1]$ is a semisimple perverse sheaf, equipped with 
$S_k$-action, and is decomposed as 
\begin{equation*}
\tag{5.14.1}
(\pi^1)_!\Ql[\dim \Fg_1] \simeq \bigoplus_{\r_1 \in S_k\wg}
          \r_1\otimes \IC(\Fg_1, \SL^1_{\r_1})[\dim \Fg_1].
\end{equation*} 
We put $A_{\r_1} = \IC(\Fg_1, \SL^1_{\r_1})[\dim \Fg_1]$. 
\par
On the other hand, the varieties $Y'^0_{i}, \wt Y'^+_{i}$ and the map 
$\psi^2_{i} : \wt Y'^+_{i} \to Y'^0_{i}$ are defined similarly to 
$Y^0_i, \wt Y_i$ and $\psi_i : \wt Y^+_i \to Y_i^0$,  
by replacing $H$ by $H_2$.  In particular, in the case where $i = 0$, 
we have, by (2.10.5), 
\begin{equation*}
\tag{5.14.2}
(\psi^2_{0})_!\Ql \simeq \bigoplus_{\r_2 \in S_{n-k}\wg}
                    H^{\bullet}(\BP_1^{n-k})\otimes \r_2 \otimes \SL^2_{\r_2},
\end{equation*}
where $\SL^2_{\r_2}$ is a simple local system on $Y'_0 = Y'^0_{0}$ 
corresponding to $\r_2$. 
Since $Y'_{0}$ is an open dense smooth subset of $X'_{0}$, we 
can consider the intersection cohomology 
$A_{\r_2} = \IC(X'_{0}, \SL^2_{\r_2})[\dim X'_{0}]$ on $X'_{0}$.   
\par
Now $A_{\r_1}\boxtimes A_{\r_2}$ is a $(G_1 \times H_2)$-equivariant simple 
perverse sheaf on $\Fg_1 \times X'_{0}$. 
By (5.11.3), there exists a unique simple perverse sheaf $A_{\r}$ on $X^0_k$ such that 
\begin{equation*}
\tag{5.14.3}
q^*A_{\r}[\b_2] \simeq \s^*(A_{\r_1} \boxtimes A_{\r_2})[\b_1], 
\end{equation*}
where $\b_1 = \dim H$ and $\b_2 = \dim (G_1 \times H_2)$. 
(Here we put $\r = \r_1\boxtimes \r_2 \in (S_k \times S_{n-k})\wg$.)
\par
We have the following lemma.

\begin{lem}  
Let $\SL_{\r}$ be a simple local system on $Y_k^0$ given in (2.10.5).
Then we have
\begin{equation*}
A_{\r} \simeq \IC(X^0_k, \SL_{\r})[d_k].
\end{equation*} 
\end{lem}

\begin{proof}
Since $A_{\r}$ is a simple perverse sheaf on $X^0_k$, 
in order to prove the lemma, it is enough to see that 
\begin{equation*}
\tag{5.15.1}
\SH^{-d_k}A_{\r}|_{Y^0_k} \simeq \SL_{\r}.
\end{equation*}
\par
We consider, for $I = [1,k], I' = \emptyset$,  the following commutative diagram.
\begin{equation*}
\tag{5.15.2}
\begin{CD}
\wt\Fg_{1,\rg} \times \wt Y'_{I'}  @<\wt \s_0<<  \wt Z^0_I @>\wt q_0>> \wt Y_I \\
   @V\xi^1 \times \xi^2_{I'}VV     @VV\wt\xi_IV           @VV\xi_IV \\
(G_1/T_1 \times \Ft_{1,\rg}) \times \wh Y'_{I'}   
              @<\wh\s_0<<  \wh Z_I  @>\wh q_0>> \wh Y_I  \\
    @V\eta^1 \times \eta^2_{I'}VV       @VV\wt\eta_IV   @VV\eta_IV  \\ 
\Fg_{1,\rg} \times Y'_{0} @<\s_0<<   \SG_{k,\sr}   @>q_0>>   Y^0_k,
\end{CD}
\end{equation*}
where $\SG_{k,\sr} = q\iv(Y^0_k), \wt Z^0_I = \wt q\iv(\wt Y_I)$, and 
$\wh Z_I$ is the quotient of $\wt Z^0_I$ by the natural action of the group
$Z_G(\Ft_{\sr})_I/(Z_G(\Ft_{\sr}) \cap B)$. The maps $\wt q_0, q_0, \wt\s_0, \s_0$ 
are defined as the restriction of the corresponding maps $\wt q, q, \wt s, \s$. 
\par
Now the map $\eta^1 \times \eta^2_{I'}$ is a finite Galois covering 
with Galois group $S_k \times S_{n-k}$.  Since the bottom squares in the diagram 
(5.15.2) are cartesian, this Galois covering is compatible with the Galois covering 
$\eta_I$.  Hence for any $\r_1 \in S_k\wg, \r_2 \in S_{n-k}\wg$, 
we have
\begin{equation*}
\s^*(\SL^1_{\r_1} \boxtimes \SL^2_{\r_2}) \simeq q^*\SL_{\r}
\end{equation*}
for $\r = \r_1\boxtimes\r_2 \in (S_k \times S_{n-k})\wg$. 
(5.15.1) follows from this.  The lemma is proved.
\end{proof}

\par
We can now state the following result.
\begin{prop}  
$(\pi_k)_!\Ql$ is decomposed as
\begin{equation*}
(\pi_k)_!\Ql \simeq H^{\bullet}(\BP_1^{n-k}) \otimes 
                 \bigoplus_{\r \in (S_k \times S_{n-k})\wg}
              \wh \r \otimes \IC(X^0_k, \SL_{\r}),
\end{equation*}
where $\wh \r$ is regarded as a vector space, ignoring the $W _n$-action. 
\end{prop}

\para{5.17.}
We prove Proposition 5.16 and Theorem 5.7 simultaneously, by induction 
on $n$.  We assume that the theorem and the proposition holds for $n' < n$. 
First we show 
\begin{lem}  
Proposition 5.16 holds for $k \ne 0$. 
\end{lem}
\begin{proof}
For any $I \subset [1,n]$ with $|I| = k$, put $\wt Z_I = \wt q\iv(\wt X_I)$. 
We have the following commutative diagram
\begin{equation*}
\tag{5.18.1}
\begin{CD}
\wt\Fg_1 \times \wt X'_0 @<<<  \wt Z_I  @>>>  \wt X_I  \\
   @V\pi^1 \times \pi^2_0VV      @VVV             @VV\pi_IV   \\
   \Fg_1 \times X'_0  @<\s<<   \SG_k  @>q>>     X^0_k, 
\end{CD}
\end{equation*}
where $\pi^1, \pi^2_0$ are as in (5.13.2). 
$\pi_I$ is the restriction of $\pi_k : \wt X^+_k \to X_k^0$. 
Since $\wt X_I$ is closed in $\wt X^+_k$, $\pi_I$ is proper.  
We note that both squares are cartesian squares. 
\par
We show the following.
\par\medskip\noindent
(5.18.2) \ Any simple summand (up to shift) of the semisimple complex 
$(\pi_I)_!\Ql$ is contained in the set $\{ A_{\r} \mid \r \in (S_k \times S_{n-k})\wg \}$.
\par\medskip
Put $K_1 = (\pi^1)_!\Ql$ and $K_2 = (\pi^2_0)_!\Ql$.  
We have
\begin{align*}
K_1 &\simeq \bigoplus_{\r_1 \in S_k\wg}\r_1\otimes \IC(\Fg_1, \SL^1_{\r_1}), \\
K_2 &\simeq H^{\bullet}(\BP_1^{n-k})\otimes \bigoplus_{\r_2 \in S_{n-k}\wg}
                           \wh\r_2 \otimes \IC(X'_0, \SL^2_{\r_2}).
\end{align*}
In fact, the first formula follows from (5.14.1), the second formula 
follows from Proposition 5.16, by applying the induction hypothesis 
to the case $n' = n - k < n$ and $k' = 0$. 
Since both squares in (5.18.1) are cartesian, we have 
$\s^*(K_1 \boxtimes K_2) \simeq q^*(\pi_I)_!\Ql$, up to shift.  
Then (5.18.2) follows from (5.14.3).
\par 
Now by Lemma 5.9, we have $(\pi_k)_!\Ql \simeq \bigoplus_{I}(\pi_I)_!\Ql$. 
Hence (5.18.2) implies, by Lemma 5.15, that
\par\medskip\noindent
(5.18.3) \ Any simple summand (up to shift) of the semisimple complex $(\pi_k)_!\Ql$
is contained in the set $\{ \IC(X^0_k, \SL_{\r}) \mid \r \in (S_k \times S_{n-k})\wg\}$. 
\par\medskip
(5.18.3) implies, in particular, that any simple summand of $K = (\pi_k)_!\Ql$ 
has its support $X_k^0$. Since the restriction of $K$ on $Y^0_k$ coincides with 
$K_0 = (\psi_k)_!\Ql$, the decomposition of $K$ into simple summands is determined by 
the decomposition of $K_0$. Hence the lemma follows from  (2.10.5). 
\end{proof}

\para{5.19.}
We consider the semisimple complex 
$(\pi_0)_!\Ql$ for $\pi_0 : \wt X_0 \to X_0 = \ol\Fh_{ss}$.  
In this case, the induction hypothesis can not be applied. 
But $(\pi_0)_!\Ql|_{Y_0} \simeq (\psi_0)_!\Ql$, and by (2.10.5) we have
\begin{equation*}
(\psi_0)_!\Ql \simeq H^{\bullet}(\BP_1^n)\otimes 
                \bigoplus_{\r \in S_n\wg}\wh\r \otimes \SL_{\r},
\end{equation*} 
by ignoring the $W _n$-module structure. It follows that $(\pi_0)_!\Ql$ can be written as 
\begin{equation*}
\tag{5.19.1}
(\pi_0)_!\Ql \simeq H^{\bullet}(\BP_1^n)\otimes\bigoplus_{\r \in S_n\wg}
                      \wh\r \otimes \IC(X_0, \SL_{\r}) + \SN_0.
\end{equation*}
Here $\SN_0$ is a sum of various complexes of the form $A[i]$, 
where $A$ is a simple perverse sheaf such that $\dim\supp A < \dim X_0$. 
\par
For each $0 \le m \le n$, let $\ol\pi_m$ be the restriction of $\pi$ 
on $\pi\iv(X_m)$. The following formula can be proved by 
a similar argument as in the proof of (2.13.3), by using Lemma 5.18 and 
(5.19.1) instead of (2.10.5).  
\begin{align*}
\tag{5.19.2}
(&\ol\pi_m)_!\Ql[d_m]  \\ 
    &\simeq \bigoplus_{0 \le k \le n}
                          \bigoplus_{\r \in (S_k \times S_{n-k})\wg}
          \wh\r \otimes \IC(X_k, \SL_{\r})[d_m - 2(n-k)] + \SM_m + \SN_0,   
\end{align*}
where $\SM_m$ is a sum of various $\IC(X_k, \SL_{\r})[d_m -2i]$ for $0 \le k \le m$ 
and $\r \in (S_k \times S_{n-k})\wg$ with $i < n - k$. 
\par\medskip
Note that, if $k > 0$, then all the simple perverse sheaves $A$ appearing 
in the decomposition of $(\pi_k)_!\Ql$ (up to shift) have support 
$X_k$ by Lemma 5.18.
This is also true for a simple perverse sheaf $A$ appearing in the first term of 
$(\pi_0)_!\Ql$ in (5.19.1). By Lemma 2.9, we have $\dim X_k \ge \dim X_0$ for any $k$.
Hence the above perverse sheaves $A$ have the property that $\dim\supp A \ge \dim X_0$. 
Since any perverse sheaf $A'$ appearing in $\SN_0$ has the property that 
$\dim\supp A' < \dim X_0$, there is no interaction between $\SN_0$ and other parts
in the computation of $(\ol\pi_m)_!\Ql$.  Thus $\SN_0$ appears in (5.19.2) 
without change (up to shift).  
\par
We consider the case where $m = n$. In this case, 
$(\ol\pi_n)_!\Ql[d_n] = \pi_!\Ql[d]$ is a semisimple perverse sheaf by Lemma 4.3.
This implies that $\SM_n = 0$, and we have
\begin{equation*}
\tag{5.19.3}
\pi_!\Ql[d] \simeq 
           \bigoplus_{0 \le k \le n}
                          \bigoplus_{\r \in (S_k \times S_{n-k})\wg}
          \wh\r \otimes \IC(X_k, \SL_{\r})[d_k] +  \SN_0.   
\end{equation*}
\par
We now apply Proposition 4.11.  
Since $\sum_{\wh\r \in W_n\wg}(\dim \wh\r)^2 = |W _n |$, 
we must have $\SN_0 = 0$. This proves Proposition 5.16 in the case where $k = 0$. 
Hence Proposition 5.16 holds for any $k$ by Lemma 5.18. The theorem now  
follows from (5.19.3).  It remains to consider the case where $n = 1$. 
But in this case, $X_0 = \Ft$ coincides with the center of $\Fh = \Fs\Fl_2$, 
and the proposition is easily verified.  This completes the proof of Theorem 5.7
and Proposition 5.16. 
\par\medskip
As a corollary to Theorem 5.7, we have the following.

\begin{cor} 
Assume that $\r \in (S_k \times S_{n-k})\wg$, and let $\wh\r \in W_n\wg$
be as in (2.12.1). Then 
\begin{align*}
\tag{5.20.1}
\IC(\wt\Fh, \SL\flt_{\wh\r})|_{\Fh} &\simeq \IC(X_k, \SL_{\r}) 
                           \quad \text{\rm (up to shift)}, \\
\tag{5.20.2}
\IC(X_k, \SL_{\r})|_{\Fh\nil} &\simeq \IC(\ol\SO_{\wh\r}, \Ql) 
                           \quad \text{\rm  (up to shift)}.
\end{align*}
\end{cor}

\begin{proof}
(5.20.1) is obtained by comparing Theorem 5.7 with Proposition 3.4.
By comparing Proposition 3.4 and Theorem 3.6, we have 
$\IC(\wt\Fh, \SL\flt_{\wh\r})|_{\Fh\nil} \simeq \IC(\ol\SO_{\wh\r}, \Ql)$, 
up to shift.  Hence by (5.20.1),  
\begin{equation*}
\IC(X_k, \SL_{\r})|_{\Fh\nil} \simeq \IC(\wt\Fh, \SL\flt_{\wh\r})|_{\Fh\nil}
                              \simeq \IC(\ol\SO_{\wh\r}, \Ql),
\end{equation*}
up to shift. Thus (5.20.2) holds. 
\end{proof}

\par\bigskip
\section{Intersection cohomology on $\Fg^{\th}$}

\para{6.1.}
From this section until the end of this paper, we discuss about $G^{\io\th}$ with 
$N$ : odd. 
So, assume that $N = 2n+1$, and $G = GL_N$.  We follow the notation in 1.10.
We have $G^{\th} \simeq SO(V') \simeq Sp(V)$, and $G^{\io\th} \simeq \Fg^{\th}$
by Proposition 1.11 and 1.13. 
Put $H = Sp(V)$ and $\Fh = \Fs\Fp(V)$. As in 1.14, 
$\Fg^{\th} = \Fh \oplus \Fg_{V'} = \Fh \oplus \Fg_V \oplus \Fz$. 
$H$ acts trivially on $\Fz \simeq \Bk$, and the action of $H$ on $\Fh\oplus\Fg_V$ 
can be identified with the diagonal action of $H$ on $\Fh \times V$. 
Moreover, $G^{\io\th}\uni \simeq \Fg^{\th}\nil = (\Fh\oplus \Fg_V)\nil$.
Hence considering $G^{\io\th}$ with $G^{\th}$-action is essentially the same 
as considering the variety $\Fh \times V$ with diagonal action of $H$ 
(but see Remark 1.15). 
For $H$ and $\Fh$, we use the same notation as in the previous sections. 

\para{6.2.}
Put 
\begin{equation*}
\SQ_{n,3} = \{ \Bm = (m_1, m_2, m_3) \in \BZ^3_{\ge 0}\mid \sum m_i = n\}.
\end{equation*} 
Recall the definition of $\FN_k$ in (5.4.1) and $\FN_{k,\sr}$ in 2.4.. 
For $i = 1, \dots, n$, let $M_i$ be 
the subspace of $V$ spanned by $e_1, \dots, e_i$. 
For $\Bm \in \SQ_{n,3}$, we define varieties
\begin{align*}
\wt\SX_{\Bm} &= \{ (x, v, gB) \in \Fh \times V \times H/B
                      \mid g \iv x \in \ol\FN_{p_2}, g\iv v \in M_{p_1} \} \\
    \SX_{\Bm} &= \bigcup_{g \in H}g(\ol\FN_{p_2} \times M_{p_1}),
\end{align*}
where we put $p_1 = m_1, p_2 = m_1 + m_2$.
We define a map $\pi^{(\Bm)} : \wt\SX_{\Bm} \to \SX_{\Bm}$ by $(x,v, gB) \mapsto (x,v)$. 
$\pi^{(\Bm)}$ is a proper surjective map.  Since 
\begin{equation*}
\wt \SX_{\Bm} \simeq H \times^B(\ol\FN_{p_2} \times M_{p_1}), 
\end{equation*}
$\wt\SX_{\Bm}$ is smooth, irreducible by Lemma 5.5.  
We also define varieties
\begin{align*}
\wt\SY_{\Bm} &= \{ (x, v, gB) \in \Fh \times V  \times H/B
                    \mid g\iv x \in \FN_{p_2, \sr}, g\iv v \in M_{p_1} \}, \\
   \SY_{\Bm} &= \bigcup_{g \in H}g(\FN_{p_2,\sr} \times M_{p_1}), 
\end{align*}
and a map $\psi^{(\Bm)} : \wt\SY_{\Bm} \to \SY_{\Bm}$ by 
$(x,v, gB) \mapsto (x,v)$. 
In the case where $\Bm = (n,0,0)$, we write 
$\wt\SX_{\Bm}, \SX_{\Bm}, \pi^{(\Bm)}$ and 
$\wt\SY_{\Bm}, \SY_{\Bm},\psi^{(\Bm)}$ simply as 
$\wt\SX, \SX, \pi$ and 
$\wt\SY, \SY, \psi$. 
Note that 
$\SX = \bigcup_{g \in H}g(\Fb \times M_n)$ is a closed subset of $\Fh \times V$, 
and $\SX_{\Bm}$ is a closed subset of $\SX$ for any $\Bm$. 
Also note that in the case where $m_1 = 0$, $\SX_{\Bm}, \SY_{\Bm}$, etc.
coincide with $X_{m_2}, Y_{m_2}$, etc. in previous sections. 
\par
As in (2.6.1), one can write $\wt\SY_{\Bm}$ as
\begin{align*}
\wt \SY_{\Bm} &\simeq H\times^B(\FN_{p_2,\sr} \times M_{p_1}) \\
         &\simeq H \times^{B \cap Z_H(\Ft_{\sr})}
             \bigl((\Ft_{\sr} + \FD_{p_2}) \times M_{p_1}\bigr).
\end{align*}
 
\par
For each $I \subset [1,n]$,
let $M_I$ be the subset of $M_n$ consisting of $v = \sum_{i \in I} c_ie_i$ 
with $c_i \ne 0$.  
For $\Bm \in \SQ_{n,3}$, we define $\SI(\Bm)$ as the set of $\BI = (I_1, I_2, I_3)$
such that $[1,n] = \coprod_{1 \le i \le 3} I_i$ with $|I_i| = m_i$  
For $\BI \in \SI(\Bm)$, put
$\FD_{\BI} = \FD_{I_2} + \ol \FD_{I_1}$, where $\ol\FD_{I_1}$ is the closure of 
$\FD_{I_1}$ in $\FD$. 
We define  a variety 
\begin{align*}
\tag{6.2.1}
\wt\SY_{\BI} = H \times^{B \cap Z_H(\Ft_{\sr})}
                \bigl((\Ft_{\sr} +  \FD_{\BI}) \times M_{I_1}\bigr).
\end{align*}
Since the actions of $B \cap Z_H(\Ft_{\sr})$ on $\FD$ and on $M_n$ are 
given by the actions of the torus part $T$, 
$(\Ft_{\sr} + \FD_I) \times M_{I_1}$ is $B \cap Z_H(\Ft_{\sr})$-stable. 
Hence $\wt \SY_{\BI}$ is well-defined. 
Let $\psi_{\BI} : \wt\SY_{\BI} \to \SY$ be the map defined by 
$g*(x, v) \mapsto (gx, gv)$, where 
$g \in H, (x,v) \in (\Ft_{\sr} + \SD_{\BI}) \times M_{I_1}$. 
Then $\Im \psi_{\BI}$ is independent of the choice of $\BI \in \SI(\Bm)$, which we denote 
by $\SY^0_{\Bm}$.  We have, for any $\BI \in \SI(\Bm)$, 
\begin{equation*}
\SY^0_{\Bm} = \bigcup_{g \in H}g\bigl((\Ft_{\sr}+ \SD_{\BI}) \times M_{I_1}\bigr).
\end{equation*}
\par
For $\BI \in \SI(\Bm)$, we define a parabolic subgroup $Z_H(\Ft_{\sr})_{\BI}$ 
by the condition that the $i$-th factor is $SL_2$ if $i \in I_3$ and is $B_2$ otherwise
(a generalization of $Z_H(\Ft_{\sr})_I$ in 2.7). 
Since $Z_H(\Ft_{\sr})_{\BI}$ stabilizes $\SD_{\BI}$ and $M_{I_1}$, one can define 
\begin{equation*}
\tag{6.2.2}
\wh\SY_{\BI} = H \times^{Z_H(\Ft_{\sr})_{\BI}}
               \bigl((\Ft_{\sr} + \SD_{\BI}) \times M_{I_1}\bigr).
\end{equation*} 

The map $\psi_{\BI}$ factors through $\wh\SY_{\BI}$, 
\begin{equation*}
\begin{CD}
\psi_{\BI} : \wt\SY_{\BI} @>\xi_{\BI}>>  \wh\SY_{\BI} @>\eta_{\BI}>>  \SY^0_{\Bm},
\end{CD}
\end{equation*} 
where $\xi_{\BI}$ is the natural projection, and $\eta_{\BI}$ is the map 
defined by  $g*(x,v) \mapsto (gx, gv)$. 
$\xi_{\BI}$ is the locally trivial fibration with fibre isomorphic to 
\begin{equation*}
Z_H(\Ft_{\sr})_{\BI}/(B \cap Z_H(\Ft_{\sr})) 
            \simeq (SL_2/B_2)^{I_3} \simeq \BP_1^{I_3}.
\end{equation*}
\par
Let $S_{\BI}\simeq S_{I_1}\times S_{I_2} \times S_{I_3}$ be the stabilizer of 
$\BI = (I_1, I_2, I_3)$ in $S_n$. 
Now $N_H(\Ft_{\sr})/Z_H(\Ft_{\sr}) \simeq S_n$, and $S_n$ acts on 
$Z_H(\Ft_{\sr}) \simeq SL_2 \times \cdots \times SL_2$ as the permutation of 
factors. Since $S_{\BI}$ stabilizes $M_{I_1}$ and $\SD_{\BI}$, $S_{\BI}$ acts on 
$\wt\SY_{\BI}$ and on $\wh\SY_{\BI}$. 
Now the map $\eta_{\BI}$ is a finite Galois covering with Galois group 
$S_{\BI}$. 
\par
For each $\Bm$, we define
$\BI(\Bm) =  ([1, p_1], [p_1 +1, p_2], [p_2 + 1, n])$, 
and put $\wt\SY_{\BI(\Bm)} = \wt\SY^0_{\Bm}$, $S_{\BI(\Bm)} = S_{\Bm}$.
Note that $\wt\SY^0_{\Bm}$ is an open dense subset of $\wt\SY_{\Bm}$, hence 
irreducible.
Put $\psi\iv(\SY^0_{\Bm}) = \wt\SY^+_{\Bm}$. 
$S_n$ acts naturally on $\wt\SY$, and the map $\psi$ is $S_n$-equivariant with respect 
to the trivial action of $S_n$ on $\SY$.  
Hence it preserves the subset $\wt\SY^+_{\Bm}$ of $\wt\SY$, and the stabilizer of 
$\wt\SY^0_{\Bm}$ in $S_n$ coincides with $S_{\Bm}$.    
One can check that 
\begin{equation*}
\tag{6.2.3}
\wt\SY^+_{\Bm} = \coprod_{\BI \in \SI(\Bm)}\wt\SY_{\BI}
               = \coprod_{w \in S_n/S_{\Bm}}w(\wt\SY^0_{\Bm}), 
\end{equation*}
where $\wt\SY_{\BI}$ is an irreducible component, hence is a connected component. 
\par
As in [Sh, 1.3], we define a partial order on $\SQ_{n,3}$ by $\Bm' \le \Bm$ if 
$p_i' \le p_i$ for $i = 1,2$, where $(p_1',p_2')$ are define for $\Bm'$ 
similarly to $(p_1,p_2)$ for $\Bm$. 
Then $\SY_{\Bm'} \subset \SY_{\Bm}$ and $\SX_{\Bm'} \subset \SX_{\Bm}$
if $\Bm' \le \Bm$.  One can check that
\begin{equation*}
\tag{6.2.4}
\SY^0_{\Bm} = \SY_{\Bm} -  \bigcup_{\Bm' < \Bm}\SY_{\Bm'}.
\end{equation*} 
Thus $\SY^0_{\Bm}$ is an open dense subset of $\SY_{\Bm}$, and we have
a partition $\SY_{\Bm} = \coprod_{\Bm'\le \Bm}\SY^0_{\Bm'}$. 
It follows that $\SY_{\Bm'} \subseteq \SY_{\Bm}$ if and only if $\Bm' \le \Bm$.
Also we have $\SY = \coprod_{\Bm \in \SQ_{n,3}}\SY^0_{\Bm}$. 
\par
The following lemma can be proved in a similar way as Lemma 1.4 in [Sh] 
(the special case where $r = 3$).

\begin{lem}  
Let $\Bm \in \SQ_{n,3}$.
\begin{enumerate}
\item 
$\SY_{\Bm}$ is open dense in $\SX_{\Bm}$, and $\wt\SY_{\Bm}$ is open dense
in $\wt\SX_{\Bm}$. 
\item
$\dim \wt\SX_{\Bm} = \dim\wt\SY_{\Bm} = 2n^2 + 2m_1 + m_2$.
\item
$\dim \SX_{\Bm} = \dim \SY_{\Bm} = 2n^2 + 2m_1 + m_2 - m_3$. 
\item
$\SY = \coprod_{\Bm \in \SQ_{n,3}}\SY^0_{\Bm}$ gives a stratification of $\SY$ by 
smooth strata $\SY^0_{\Bm}$, and the map $\psi : \wt\SY \to \SY$ is
semismall with respect to this stratification. 
\end{enumerate}
\end{lem}

\para{6.4.}
Let $\psi_{\Bm} : \wt\SY^+_{\Bm} \to \SY^0_{\Bm}$ be the restriction of 
$\psi$ on $\wt\SY^+_{\Bm}$. 
Then $\psi_{\Bm}$ is $S_n$-equivariant with respect to the natural action of 
$S_n$ on $\wt\SY^+_{\Bm}$ and the trivial action of $S_n$ on $\SY^0_{\Bm}$. 
By (6.2.3), we have
\begin{equation*}
\tag{6.4.1}
(\psi_{\Bm})_!\Ql \simeq \bigoplus_{\BI \in \SI(\Bm)}(\psi_{\BI})_!\Ql.
\end{equation*}
Since $\eta_{\BI} : \wh\SY_{\BI} \to \SY^0_{\Bm}$ is a finite Galois covering 
with Galois group $S_{\BI}$, $(\eta_{\BI})_!\Ql$ is a semisimple local system on 
$\SY^0_{\Bm}$, and is decomposed as
\begin{equation*}
\tag{6.4.2}
(\eta_{\BI})_!\Ql \simeq \bigoplus_{\r \in S_{\BI}\wg}\r \otimes \SL_{\r},
\end{equation*}
where $\SL_{\r} = \Hom (\r, (\eta_{\BI})_!\Ql)$ is a simple local system on $\SY^0_{\Bm}$. 
\par
Now by a similar argument as in 2.10, we have
\begin{equation*}
\tag{6.4.3}
(\psi_{\BI})_!\Ql \simeq (\eta_{\BI})_!(\xi_{\BI})_!\Ql 
                  \simeq H^{\bullet}(\BP_1^{I_3})\otimes (\eta_{\BI})_!\Ql.
\end{equation*}
For a positive integer $r$, let $W_{n,r} = S_n \ltimes (\BZ/r\BZ)^n$ 
be the complex reflection group.
Hereafter we assume that $r = 3$, and consider $W_{n,r} = W_{n,3}$. 
Since $S_{\Bm}$ is a subgroup of $S_n$, 
we can consider $W_{\Bm, r} = S_{\Bm}\ltimes (\BZ/r\BZ)^n$ as a subgroup of 
$W_{n,r}$. 
Let $\z$ be  a primitive $r$-th root of unity in $\Ql$, and define a linear character
$\tau_i : \BZ/r\BZ \to \Ql^*$ by $\tau_i(a) = \z^{i-1}$, where $a$ is a fixed generator 
of $\BZ/r\BZ$. 
Let $\r \in S_{\Bm}\wg$. 
Since $S_{\Bm} \simeq \prod_iS_{m_i}$,  
$\r$ is written as $\r = \r_1\boxtimes \r_2 \boxtimes \r_3$, with 
$\r_i \in S_{m_i}\wg$. 
Here $W_{\Bm,r} \simeq \prod_iW_{m_i,r}$. 
We extend the irreducible $S_{m_i}$-module $\r_i$ to 
an irreducible $W_{m_i,r}$-module $\wt\r_i$ by defining the action of 
$(\BZ/r\BZ)^{m_i}$ via $\tau_i^{\otimes m_i}$, and put 
$\wt\r = \wt\r_1\boxtimes\wt\r_2\boxtimes\wt\r_3 \in W_{\Bm,r}\wg$.  
We also define $\wt\r' = \wt\r_1\boxtimes\wt\r_2\boxtimes\wt\r_3' \in W_{\Bm,r}\wg$, 
where $\wt\r_3'$ is the trivial extension of $\r_3$ to $W_{m_3,r}$. 
Put $\wh\r = \Ind_{W_{\Bm,r}}^{W_{n,r}}\wt\r$.   
Then $\wh\r$ is an irreducible 
$W_{n,r}$-module, and any irreducible representation of $W_{n,r}$ 
is obtained in this way from $\r \in S_{\Bm}\wg$ for various $\Bm$.  
\par
In view of (6.4.1), (6.4.2) and (6.4.3), similarly to the discussion 
in 2.10, $(\psi_{\Bm})_!\Ql$
can be written as 
\begin{equation*}
\tag{6.4.4}
(\psi_{\Bm})_!\Ql \simeq \bigoplus_{\r \in S_{\Bm}\wg}
                      \Ind_{S_{\Bm}}^{S_n}(H^{\bullet}(\BP_1^{m_3}) \otimes \r)
                                \otimes \SL_{\r}.   
\end{equation*}
We define an action of $\BZ/r\BZ$ on $H^{\bullet}(\BP_1) = H^2(\BP_1) \oplus H^0(\BP_1)$
by $\tau_r \oplus \tau_1$, and define an action of $(\BZ/r\BZ)^{m_r}$ on 
$H^{\bullet}(\BP_1^{I_r}) \simeq H^{\bullet}(\BP_1)^{\otimes m_r}$ as its tensor product. 
Thus we can extend $H^{\bullet}(\BP_1^{m_r})\otimes \r$ to a complex of $W_{\Bm,r}$-modules
$H^{\bullet}(\BP_1^{m_r})\otimes \wt\r'$.  Thus by (6.4.4), we can define a $W_{n,r}$-action 
on $(\psi_{\Bm})_!\Ql$, 
\begin{equation*}
\tag{6.4.5}
(\psi_{\Bm})_!\Ql \simeq \bigoplus_{\r \in S_{\Bm}\wg}
                    \Ind_{W_{\Bm,r}}^{W_{n,r}}(H^{\bullet}(\BP_1^{m_3})\otimes \wt\r')
                            \otimes \SL_{\r}.
\end{equation*}
\par
Note that (6.4.5) can be rewritten as 
\begin{equation*}
\tag{6.4.6}
(\psi_{\Bm})_!\Ql \simeq \biggl(\bigoplus_{\r \in S_{\Bm}\wg}\wh\r \otimes \SL_{\r}\biggr)
                     [-2m_3] + \SN_{\Bm},
\end{equation*}
where $\SN_{\Bm}$ is a sum of various $\SL_{\r}[-2i]$ for $\r \in S_{\Bm}\wg$ 
with $0 \le i < m_3$. 

\para{6.5.}
For each $\Bm \in \SQ_{n,3}$, let $\ol\psi_{\Bm}$ be the restriction of $\psi$
on $\psi\iv(\SY_{\Bm})$. 
Put $d_{\Bm} = \dim \SY_{\Bm}$. 
Let $\SQ_{n,3}^0$ be the set of $\Bm = (m_1, m_2, m_3) \in \SQ_{n,3}$ such that 
$m_3 = 0$. Take $\Bm = (m_1, m_2, 0) \in \SQ_{n,3}^0$. 
For $0 \le k \le m_2$, define $\Bm(k) \in \SQ_{n,3}$ by 
$\Bm(k) = (m_1, k, m_2 - k)$. The following result can be proved in a similar 
way as Proposition 1.7 in [Sh]. It is a special case where $r = 3$ of 
[loc. cit.]. (See also the  proof of Proposition 2.13).
  
\begin{prop} 
Assume that $\Bm \in \SQ^0_{n,3}$. 
Then $(\ol\psi_{\Bm})_!\Ql[d_{\Bm}]$ is a semisimple perverse sheaf on $\SY_{\Bm}$, 
equipped with $W_{n,3}$-action, and is decomposed as
\begin{equation*}
(\ol\psi_{\Bm})_!\Ql[d_{\Bm}] \simeq \bigoplus_{0 \le k \le m_2}
                    \bigoplus_{\r \in S_{\Bm(k)}\wg}
                        \wh\r \otimes \IC(\SY_{\Bm(k)}, \SL_{\r})[d_{\Bm(k)}].
\end{equation*} 
\end{prop}
  
\para{6.7.}
For $\Bm = (m_1, m_2, 0) \in \SQ^0_{n,3}$, let $W\nat_{\Bm}$ be the subgroup 
of $W_n$ defined by $W\nat_{\Bm} \simeq S_{m_1} \times W_{m_2}$, where 
$S_{m_1}$ is the group of permutations for $[1, m_1]$ and $W_{m_2} = W_{m_2,2}$ 
is the group 
of signed permutations for $[m_1+1, n]$. 
Let $\Bm(k) = (m_1,k, k')$ with $k + k' = m_2$. 
For $\r = \r_1\boxtimes \r_2 \boxtimes \r_3 \in S_{\Bm(k)}\wg$, we define 
$\r\nat \in W\nat_{\Bm}$ by 
$\r\nat = \r_1\boxtimes \r_2'$, where $\r_2' \in W_{m_2}\wg$ is given by  
$\r_2' = \Ind_{W_{k} \times W_{k'}}^{W_{m_2}}(\wt\r_2 \boxtimes \wt\r_3)$. 
(Here $\wt\r_2$ is the trivial extension of $\r_2$ to $W_{k}$, $\wt\r_3$ is
 the extension of $\r_3$ to $W_{k'}$ 
with non-trivial action of $\BZ/2\BZ$ for each factor.) 
\par
Recall the map $\psi^{(\Bm)} : \wt\SY_{\Bm} \to \SY_{\Bm}$ given in 6.2. 
The following result is a variant of Proposition 6.6, and is proved by a similar 
argument (see also Proposition 3.5 in [Sh]). 

\begin{prop}  
Assume that $\Bm = (m_1, m_2, 0) \in \SQ^0_{n,3}$. 
$\psi^{(\Bm)}_!\Ql[d_{\Bm}]$ is a semisimple perverse sheaf on $\SY_{\Bm}$, equipped 
with $W\nat_{\Bm}$-action, and is decomposed as
\begin{equation*}
\psi^{(\Bm)}_!\Ql[d_{\Bm}] \simeq \bigoplus_{0 \le k \le m_2}
    \bigoplus_{\r \in S\wg_{\Bm(k)}}
        \r\nat\otimes \IC(\SY_{\Bm(k)}, \SL_{\r})[d_{\Bm(k)}].
\end{equation*}
\end{prop}

\para{6.9.}
For each $\Bm \in \SQ_{n,3}$, 
we define $\ol\pi_{\Bm}$ the map $\pi\iv(\SX_{\Bm}) \to \SX_{\Bm}$ as 
the restriction of $\pi$ to $\pi\iv(\SX_{\Bm})$. Thus $\ol\pi_{\Bm}$ is a proper 
surjective map to $\SX_{\Bm}$. 
The following result is a generalization of Theorem 5.7 
(note that if $\Bm = (0, m_2, 0)$, this coincides with Theorem 5.7.) 

\begin{thm}  
Assume that $\Bm \in \SQ_{n,3}^0$. Then $(\ol\pi_{\Bm})_!\Ql[d_{\Bm}]$
is a semisimple perverse sheaf on $\SX_{\Bm}$, equipped with $W_{n,3}$-action, 
and is decomposed as
\begin{equation*}
(\ol\pi_{\Bm})_!\Ql[d_{\Bm}] \simeq
\bigoplus_{0 \le k \le m_2}\bigoplus_{\r \in S_{\Bm(k)}\wg }
                            \wh\r \otimes \IC(\SX_{\Bm(k)}, \SL_{\r})[d_{\Bm(k)}].
\end{equation*}
\end{thm}

\para{6.11.}
The theorem can be proved in an almost parallel way as the proof of 
Theorem 2.2 in [Sh], the special case where $r = 3$, once formulated 
appropriately. Also the proof is quite similar to the proof of Lemma 5.18.
So in the following, we give an outline of the proof.
\par
For $(x, v) \in \SX$, we define an $x$-stable isotropic subspace $W = W(x,v)$ as follows;
If $x$ is nilpotent, put $W = \Bk[x]v$. Since $(x,v) \in g(\Fn \times M_n)$ for 
some $g \in H$, $W$ is isotropic.  For general $x = s + z$, where 
$s$ is semisimple, $z$ is nilpotent such that $[s,z] = 0$, we consider the 
decomposition of $V$ into eigenspaces of $s$, 
$V = V_1 \oplus \cdots \oplus V_a$ as in 5.8.  Then 
$Z_H(s) \simeq Sp(V_1) \times \cdots \times Sp(V_a)$, and $z \in \Fh\nil$ can be written 
as $z = \sum_{i= 1}^az_i$, where $z_i \in \Fs\Fp(V_i)\nil$.  
We write $v = \sum_{i=1}^av_i$ with $v_i \in V_i$.  Then $(z_i, v_i) \in \SX^{(i)}$, 
where $\SX^{(i)}$ is defined similarly to $\SX$ by replacing $H$ by $Sp(V_i)$. 
We define a subspace $W_i = W(z_i,v_i)$ of $V_i$ as above, and put 
$W = \bigoplus_iW_i$.  Then $W = W(x,v)$ satisfies the required properties.  
We consider $\SX_{\Bm}$ for $\Bm \in \SQ_{n,3}$. 
Note that 
\begin{equation*}
\tag{6.11.1}
\dim W(x,v) \le m_1 \text{ if } (x,v) \in \SX_{\Bm}.
\end{equation*} 
It follows from the construction of $W$, we have 
\par\medskip\noindent
(6.11.2) \ $x|_W  \in \Fg\Fl(W)$ is a regular element. 
\par\medskip
Put $W = W(x,v)$ for $(x,v) \in \SX_{\Bm}$.  Then $V' = W^{\perp}/W$ 
has a natural symplectic structure, and we define $H' = Sp(V'), \Fh' = \Lie H'$.
$x$ induces an endomorphism $x'$ on $V'$, and we have $x' \in \Fh'$. 
Let $X'_k, {X'}^0_k, \FN'_k$, etc. be the varieties defined for $H'$, 
similarly to $X_k, X^0_k, \FN_k$, etc.  defined for $H$.
It is easy to see that if $(x,v) \in \ol\FN_{m_1 + m_2} \times M_{m_1}$ 
with $W(x,v) = M_{m_1}$, then $x' \in \ol\FN'_{m_2}$. 
It follows that 
\begin{equation*}
\tag{6.11.3}
x' \in X'_{m_2}  \text{ if } (x,v) \in \SX_{\Bm} \text{ and } \dim W(x,v) = m_1.
\end{equation*}   

\para{6.12.}
Assume that $\Bm \in \SQ_{n,3}$.  
As in the case of $\SY^0_{\Bm}$, we define an open subset $\SX^0_{\Bm}$ 
of $\SX_{\Bm}$ as
\begin{equation*} 
\SX^0_{\Bm} = \SX_{\Bm} - \bigcup_{\Bm' < \Bm}\SX_{\Bm'}.
\end{equation*}
It follows from (6.11.1) and (6.11.3) that 
\begin{equation*}
\tag{6.12.1}
\SX^0_{\Bm} = \{ (x,v) \in \SX \mid 
      \dim W(x,v) = m_1 \text{ and } x' \in {X'}^0_{m_2} \}. 
\end{equation*}

We define $\wt\SX^+_{\Bm} = \pi\iv(\SX^0_{\Bm})$, and let 
$\pi_{\Bm} : \wt\SX^+_{\Bm} \to \SX^0_{\Bm}$ be the restriction of 
$\pi$ on $\wt\SX^+_{\Bm}$. 
To $(x,v, gB) \in \wt\SX^+_{\Bm}$, we associate $\BI \in \SI(\Bm)$ as follows;
assume that $(x,v) \in \Fb \times M_n$, and $x = s + z$ be the Jordan decomposition 
of $x \in \Fb$. By replacing $(x,v)$ by $B$-conjugate, we assume that 
$s \in \Ft, z \in \Fn$. Then the decomposition $V = V_1 \oplus \cdots \oplus V_a$
gives a decomposition $M_n = M_{n,1}\oplus\cdots\oplus M_{n,a}$. 
Here $M_n$ has a basis $\{ e_1, \dots, e_n\}$, and $M_{n,i}$ has a basis 
$\{ e_j \mid j \in J_i\}$, which gives a partition 
$[1,n] = \coprod_{1 \le i \le a}J_i$. 
$(x,v)$ determines $(z_i, v_i)$, and put $q_i = \dim W(z_i,v_i)$. Clearly 
$q_i \le \dim M_{n,i}$, and put $J_i' \subset J_i$ as the set of 
first $q_i$ letters in $J_i$.  We define $I_1 = \coprod_{1 \le i \le a}J_i'$.
Since $\dim W(x,v) = m_1$, we have $|I_1| = m_1$ by (6.12.1). 
Here $x' \in {X'}^0_{m_2}$ again by (6.12.1). 
Thus by 5.8, one can associate to $x'$ a subset $I_2$
of $[m_1 + 1,n]$ such that $|I_2| = m_2$. We have $I_1 \cap I_2 = \emptyset$, and 
$\BI = (I_1, I_2, I_3)$ gives an element in $\SI(\Bm)$ ($I_3$ is the complement 
of $I_1 \cup I_2$.)  The assignment $(x, v) \mapsto \BI$ does not depend on the
$B$-conjugates of $(x,v)$, Thus we obtain a well-defined map 
$(x,v, gB) \mapsto \BI$.     
\par
We define a subvariety $\wt\SX_{\BI}$ of $\wt\SX^+_{\Bm}$ by 
\begin{equation*}
\tag{6.12.2}
\wt\SX^+_{\Bm} = \{ (x,v, B) \in \wt\SX^+_{\Bm} \mid (x,v, gB) \mapsto \BI \}.
\end{equation*}

The following lemma can be proved in a similar way as Lemma 5.9 
(see also Lemma 2.4 in [Sh]). 

\begin{lem}  
$\wt\SX^+_{\Bm}$ is decomposed as
\begin{equation*}
\wt\SX^+_{\Bm} = \coprod_{\BI \in \SI(\Bm)}\wt\SX_{\BI}.
\end{equation*}
\end{lem}

\para{6.14.}
We fix $\Bm = (m_1, m_2, m_3) \in \SQ_{n,3}$. 
We consider the space $V_0 = M_{m_1}$ and 
$\ol V_0 = V_0^{\perp}/V_0$.  As in 5.11, we put 
$G_1 = GL(V_0), H_2 = Sp(\ol V_0)$. 
We consider $X'_{m_2} \subset X'$ with respect to $H_2$. 
Put $\Fg_1 = \Lie G_1$, and let $\Fg_1^0$ be the set of regular 
elements in $\Fg_1$. 
For $\xi = (x,v) \in \SX$, put $W_{\xi} = W(x,v)$.
As a variant of (5.11.1), (see also [Sh, 2.5]) we define a variety 
\begin{align*}
\tag{6.14.1}
\CK_{\Bm} = \{ (x,&v, \f_1, \f_2) \mid \xi = (x,v) \in \SX^0_{\Bm}, \\
            &\f_1 : W_{\xi} \isom V_0, 
     \vf_2 : W_{\xi}^{\perp}/W_{\xi} \isom \ol V_0  (\text{symplectic isom.})\}
\end{align*}
with morphisms
\begin{equation*}
\begin{aligned}
q : &\CK_{\Bm} \to \SX^0_{\Bm}, &\quad  &(x,v,\f_1, \f_2) \mapsto  (x,v), \\
\s : &\CK_{\Bm} \to \Fg_1^0 \times {X'}^0_{m_2}, &\quad
          &(x,v, \f_1, \f_2) \mapsto 
     (\f_1(x|_{W_{\xi}})\f_1\iv, \f_2(x|_{W _{\xi}^{\perp}/W_{\xi}})\f_2\iv)
\end{aligned}
\end{equation*}
\par
$H \times (G_1 \times H_2)$ acts on $\CK_{\Bm}$ by 
\begin{equation*}
(g, (h_1, h_2)) : (x,v,\f_1, \f_2) \mapsto (gx, gv, h_1\f_1g\iv, h_2\f_2g\iv).
\end{equation*}
Moreover, $\s$ is $H \times (G_1 \times H_2)$-equivariant with respect to the 
adjoint action of $G_1 \times H_2$ and the trivial action 
of $H$ on $\Fg_1^0 \times {X'}^0_{m_2}$. Similarly to (5.11.3), we have the following.
(The proof is similar to ([Sh, 2.5]). 
\par\medskip\noindent
(6.14.2) \ The map $q$ is a principal bundle with fibre isomorphic to $G_1 \times H_2$.
The map $\s$ is a locally trivial fibration with smooth connected fibre of dimension 
$\dim H + m_1$. 
\par\medskip  
We define a Borel subgroup $B_1$ of $G_1$ and $B_2$ of $H_2$ as in 5.13, by 
replacing $k$ by $m_1$, and define $\pi^1: \wt\Fg_1 \to \Fg_1$ similarly. 
Put $\wt\Fg^0_1 = (\pi^1)\iv(\Fg^0_1)$, and let 
$\vf^1 : \wt\Fg^0_1 \to \Fg^0_1$ be the restriction of $\pi^1$. 
We define $X'$ by using $H_2/B_2$, and let 
$\pi^2_{m_2} : {\wt X}'^+_{m_2} \to {X'}^0_{m_2}$ be the corresponding map.  
We define a variety
\begin{align*}
\wt\CZ^+_{\Bm} = \{ (\xi, gB, &\f_1, \f_2) \mid (\xi, gB) \in \wt\SX^+_{\Bm}, \\
      &\f_1 : W_{\xi} \isom V_0, \f_2 : W_{\xi}^{\perp}/W_{\xi} \isom \ol V_0 \}, 
\end{align*}
and define a map $\wt q: \wt\CZ^+_{\Bm} \to \wt\SX^+_{\Bm}$ by 
the natural projection.  We define a map 
$\wt\s : \wt\SZ^+_{\Bm} \to \wt\Fg^0_1 \times {\wt X}'^+_{m_2}$ as follows;
take $(x,v,B, \f_1, \f_2) \in \wt\CZ^+_{\Bm}$.  Since $\xi = (x,v) \in \SX^0_{\Bm}$, 
$W_{\xi}$ coincides with $g(M_{m_1})$.  Let $g_1B_1$ be the element corresponding to
the flag $\f_1(g(M_i))_{0 \le i \le m_1}$, and $g_2B_2$ the element corresponding to the
isotropic flag $\f_2(g(M_i)/g(M_{m_1}))_{m_1 \le i \le n}$.  Then 
\begin{equation*}
\wt\s : (x,v,gB, \f_1, \f_2) \mapsto 
               ((\f_1(x|_{W_{\xi}})\f_1\iv, g_1B_1), 
                (\f_2(x|_{W_{\xi}^{\perp}/W_{\xi}})\f_2\iv, g_2B_2).
\end{equation*}  
We also define a map $\wt\pi_{\Bm} : \wt\CZ^+_{\Bm} \to \CK_{\Bm}$ 
by $(x,v,gB, \f_1, \f_2) \mapsto (x,v,\f_1,\f_2)$. 
Then we have the following commutative diagram
(compare this with (5.13.2));
\begin{equation*}
\tag{6.14.3}
\begin{CD}
\wt\Fg^0_1 \times {\wt X}'^+_{m_2} @<\wt s<<  \wt\CZ^+_{\Bm} @>\wt q>> \wt\SX^+_{\Bm} \\
    @V\vf^1 \times \pi^2_{m_2}VV             @VV\wt\pi_{\Bm}V     @VV\pi_{\Bm}V  \\
\Fg_1^0 \times X'^0_{m_2}   @<\s<< \CK_{\Bm}   @>q>>     \SX^0_{\Bm}. 
\end{CD}
\end{equation*}

\para{6.15.}
By (5.14.1), $(\pi^1)_!\Ql$ can be written as 
$(\pi^1)_!\Ql \simeq \IC(\Fg_1, \SL)$ for a semisimple local system 
$\SL$ on $(\Fg_1)\reg$; the set of regular semisimple elements in $\Fg_1$.  
Since $\Fg_1^0$ is an open dense subset of $\Fg_1$
containing $(\Fg_1)\reg$, $(\vf^1)_!\Ql$ can be written as
\begin{equation*}
\tag{6.15.1}
(\vf^1)_!\Ql \simeq \bigoplus_{\r_1 \in S_{m_1}\wg}\r_1\otimes 
               \IC(\Fg_1^0, \SL^1_{\r_1}).
\end{equation*} 
By applying Proposition 5.16 to the map 
$\pi^2_{m_2} ; \wt X'^+_{m_2} \to X'^0_{m_2}$, we have
\begin{equation*}
\tag{6.15.2}
(\pi^2_{m_2})_!\Ql \simeq H^{\bullet}(\BP_1^{m_3})\otimes 
                    \bigoplus_{\r_2 \in (S_{m_2} \times S_{m_3})\wg}
                           \wh \r_2 \otimes \IC(X'^0_{m_2}, \SL'_{\r'}).
\end{equation*}

Put $A_{\r_1} = \IC(\Fg^0_1, \SL^1_{\r_1})[\dim \Fg_1]$ and 
$A_{\r_2} = \IC(X'^0_{m_2}, \SL'_{\r_2})[\dim X'_{m_2}]$. Then 
$A_{\r_1}\boxtimes A_{\r_2}$ is a $(G_1 \times H_2)$-equivariant 
simple perverse sheaf on $\Fg_1^0 \times X'^0_{m_2}$.
By a similar argument as in 5.14, thanks to (6.14.2), 
one can construct a simple perverse sheaf $A_{\r}$ on $\SX^0_{\Bm}$, 
where $\r = \r_1 \boxtimes \r_2 \in S_{\Bm}\wg$, satisfying the following property.
\begin{equation*}
q^*A_{\r}[\b_2] \simeq \s^*(A_{\r_1}\boxtimes A_{\r_2})[\b_1],
\end{equation*}
where $\b_1 = \dim H + m_1$ and $\b_2 = \dim (G_1 \times H_2)$. 
The following lemma can be proved in a similar way as [Sh, Lemma 2.7] 
(see also the proof of Lemma 5.15).

\begin{lem}  
Let $\SL_{\r}$ be a simple local system on $\SY^0_{\Bm}$ as given in (6.4.4). 
Then we have
\begin{equation*}
A_{\r} \simeq \IC(\SX^0_{\Bm}, \SL_{\r})[d_{\Bm}].
\end{equation*}
\end{lem}

By using Lemma 6.16, we can prove the following.

\begin{prop}  
Under the notation of Lemma 6.16, $(\pi_{\Bm})_!\Ql$ is decomposed as
\begin{equation*}
(\pi_{\Bm})_!\Ql \simeq H^{\bullet}(\BP_1^{m_3})\otimes 
               \bigoplus_{\r \in S_{\Bm}\wg}\wh\r \otimes \IC(\SX^0_{\Bm}, \SL_{\r}),
\end{equation*}
where $\wh\r$ is regarded as a vector space, ignoring the $W_{n,3}$-action. 
\end{prop}
\begin{proof}
We fix $\BI = (I_1, I_2, I_3) \in \SI(\Bm)$. Then the following commutative 
diagram is obtained from (6.14.3).

\begin{equation*}
\begin{CD}
\wt\Fg^0_1 \times {\wt X}'_{I_2} @<<<  \wt\CZ_{\BI} @>>> \wt\SX_{\BI} \\
    @V\vf^1 \times \pi^2_{I_2}VV             @VVV     @VV\pi_{\BI}V  \\
\Fg_1^0 \times X'^0_{m_2}   @<\s<< \CK_{\Bm}   @>q>>     \SX^0_{\Bm}, 
\end{CD}
\end{equation*}
where $\wt\CZ_{\BI} = \wt q\iv(\wt\SX_{\BI})$. 
Note that $\pi_{\BI}$ is proper, and the both squares are cartesian squares.
Then the proposition can be proved in a similar way as in the proof of Proposition 2.8 
in [Sh].  Also see the discussion in the proof of Lemma 5.18.  We omit the details.
\end{proof}

\remark{6.18.}
The proof of [Sh, Prop. 2.8] uses the induction on $r$, and depends on the Henderson's 
result [Hen] for the case where $r = 1$, which was proved by making use of 
the Fourier-Deligne  
transform on perverse sheaves on Lie algebras.  
In turn, in the proof of Lemma 5.18, we needed to assume 
that $k \ne 0$. But thanks to Proposition 5.16, we don't need any restriction in 
the proof of Proposition 6.17. Note that Proposition 5.16 was proved by making 
use of the $W_n$-actions on perverse sheaves on $\wt\Fh$. 

\para{6.19.}
Now the theorem is proved by a similar argument as in 2.10 in [Sh].
See also the discussion in 5.19. Note that in our situation, 
$\SN_0$ does not appear thanks to Proposition 5.16.  In 5.19, if $\SN_0 = 0$, 
one can construct a representation of $W_n$ on $(\ol\pi_m)_!\Ql$ by a similar 
method as in the case of $(\ol\psi_m)_!\Ql$ (see the discussion in 2.11). 
A similar argument works for $\ol\pi_{\Bm}$ also.  

\para{6.20.}
Recall the map $\pi^{(\Bm)}: \wt\SX_{\Bm} \to \SX_{\Bm}$ given in 6.2, and
$W\nat_{\Bm}$ given in 6.7. 
The following result is a variant of Theorem  6.10,  and is proved in a similar
way as Theorem 3.2 in [Sh].   Note that if $\Bm = (m_1, m_2, 0)$, then 
$W\nat_{\Bm} = S_{m_1} \times W_{m_2,2}$.  In the special case where 
$m_1 = 0, m_2 = n$, Theorem 6.21 coincides with Theorem 5.7. 

\begin{thm}  
Assume that $\Bm = (m_1, m_2, 0)\in \SQ^0_{n,3}$. $\pi^{(\Bm)}_!\Ql[d_{\Bm}]$ 
is a semisimple perverse sheaf on $\SX_{\Bm}$, equipped with $W\nat_{\Bm}$-action,
and is decomposed as
\begin{equation*}
\pi^{(\Bm)}_!\Ql[d_{\Bm}] \simeq \bigoplus_{0 \le k \le m_2}
     \bigoplus_{\r \in S_{\Bm(k)}\wg}\r\nat\otimes \IC(\SX_{\Bm(k)}, \SL_{\r})[d_{\Bm(k)}].
\end{equation*} 
\end{thm} 
\par\bigskip
\section{Nilpotent variety for $\Fg^{\th}$}

\para{7.1.}
For each $\Bm \in \SQ_{n,3}$ with $p_1 = m_1, p_2 = m_1 + m_2$, we define varieties
\begin{align*}
\wt\SX_{\Bm,\nilp} &= \{ (x,v, gB) \in \Fh \times V \times H/B 
            \mid g\iv x \in \Fn_s \oplus \SD_{p_2}, 
                                     g\iv v \in M_{p_1}  \}, \\
   \SX_{\Bm,\nilp} &= \bigcup_{g \in H}g\bigl((\Fn_s\oplus\FD_{p_2}) \times M_{p_1}\bigr).  
\end{align*}
We write $\SX_{\Bm,\nilp}$ as $\SX\nil$ if $\Bm = (n,0,0)$. 
Note that $\Fn \times M_n \simeq \Fn \oplus (M_n)_{\Fg}$ is contained in the nilpotent 
radical of a Borel subalgebra of $\Fg\Fl_{2n+1}$ (here $(M_n)_{\Fg}$ is the 
corresponding subspace of $V_{\Fg}$).  Hence 
\begin{equation*}
\tag{7.1.1}
\SX_{\Bm, \nilp} \subset \SX\nil \subset \Fg^{\th}\nil.
\end{equation*} 

In this paper, we are only concerned with $\SX\nil$, not with $\Fg^{\th}\nil$.
However it is likely that $\Fg^{\th}\nil = \SX\nil$. 
\par
We define a map $\pi_1^{(\Bm)}: \wt\SX_{\Bm, \nilp} \to \SX_{\Bm,\nilp}$ by 
$(x,v,gB) \mapsto (x,v)$. It is clear that $\wt\SX_{\Bm,\nilp}$ is 
smooth and irreducible, and $\pi_1^{(\Bm)}$ is a proper map 
onto $\SX_{\Bm,\nilp}$. 
Since
\begin{equation*}
\wt\SX_{\Bm,\nilp} \simeq H \times^B\bigl((\Fn_s\oplus \FD_{p_2}) \times M_{p_1}\bigr),
\end{equation*}
we have
\begin{align*}
\tag{7.1.2}
\dim \wt\SX_{\Bm, \nilp} &=  \dim H - (n + m_3) + m_1  \\
                         &= 2n^2 + m_1 - m_3.
\end{align*}

\para{7.2}
For an integer $r \ge 1$, let $\SP_{n,r}$ be the set of 
$r$-tuples of partitions $\Bla = (\la^{(1)}, \dots, \la^{(r)})$ such that 
$\sum_{1 \le i \le r}|\la^{(i)}| = n$.  
In the case where $r = 1$, we simply denote by $\SP_n$ the set of partitions of 
$n$. It is well-known that $S_n\wg$ is in bijection with $\SP_n$.  We denote by 
$\r_{\la}$ the irreducible representation of $S_n$ corresponding to $\la \in \SP_n$.
We normalize the parametrization so that the unit representation $1_{S_n}$ corresponds 
to $\la = (n)$.  We consider the complex reflection group $W_{n,r}$. 
As explained in 6.4, the irreducible representation of 
$W_{n,r}$ is constructed from irreducible representations of symmetric groups. 
By this construction, we have a natural parametrization between $W_{n,r}\wg$ and 
$\SP_{n,r}$.  We denote by $\r_{\Bla}$ the irreducible representation of $W_{n,r}$
corresponding to $\Bla \in \SP_{n,r}$. 
\par
We consider the nilpotent orbits in $\Fh = \Fs\Fp(V)$. By Theorem 3.6, nilpotent 
orbits $\SO$ are parametrized by $W\wg_{n,2}$, as $\SO = \SO_{\r}$ for 
$\r \in W\wg_{n,2}$.    
We denote by $\SO_{\Bla}$ the nilpotent orbit $\SO_{\r}$ in $\Fh$ corresponding to 
$\r = \r_{\Bla}$ with $\Bla \in \SP_{n,2}$. 

\para{7.3}
Let $V_1$ be an $n$-dimensional vector space over $\Bk$, and $G_1 = GL(V_1)$. 
Put $\Fg_1 = \Lie G_1$.  We consider the diagonal action of $G_1$ on $\Fg_1 \times V_1$.
$G_1$ acts on $(\Fg_1)\nil \times V_1$. It is known by [AH], [T]  that the set of 
$G_1$-orbits in $(\Fg_1)\nil \times V_1$ is parametrized by $\SP_{n,2}$.  We denote by 
$\OO_{\Bla}$ the $G_1$-orbit corresponding to $\Bla \in \SP_{n,2}$. According to [AH], 
the explicit correspondence is given as follows;
take $(x,v) \in (\Fg_1)\nil \times V_1$. Put $E^x = \{ y \in \End(V_1) \mid yx = xy \}$.
$E^x$ is a subalgebra of $\End(V_1)$, stable by the multiplication of $x$. 
If we put $W = E^xv$, $W$ is an $x$-stable subspace of $V_1$. We denote by $\la^{(1)}$ 
the Jordan type of $x|_W$ and by $\la^{(2)}$ the Jordan type of $x|_{V_1/W}$. 
Then the Jordan type of $x$ is $\la^{(1)} + \la^{(2)}$, and 
$\Bla = (\la^{(1)}, \la^{(2)}) \in \SP_{n,2}$. $\OO_{\Bla}$ is defined 
as the $G_1$-orbit
containing $(x,v)$.  This gives the required labelling of $G_1$-orbits in 
$(\Fg_1)\nil \times V_1$. Note that if $(x,v) \in \OO_{\Bla}$, the Jordan type of 
$x$ is $\la^{(1)} + \la^{(2)}$ for $\Bla = (\la^{(1)},\la^{(2)})$.   
\par 
For a partition $\la = (\la_1, \la_2, \cdots) \in \SP_n$, put 
$n(\la) = \sum_{i \ge 1}(i-1)\la_i$. Let $\SO_{\la}$ be the $G_1$-orbit in $(\Fg_1)\nil$
consisting of $x$ of Jordan type $\la$.  It is well-known that 
\begin{equation*}
\tag{7.3.1}
\dim \SO_{\la} = n^2 - n - 2n(\la). 
\end{equation*}
\par
Let $\OO_{\Bla}$ be as above.  The following formula was proved 
in [AH, Prop. 2.8].  Put $\nu = \la^{(1)} + \la^{(2)}$, 
and define $n(\Bla) = n(\la^{(1)}) + n(\la^{(2)})$.  Then 
\begin{align*}
\tag{7.3.2}
\dim \OO_{\Bla} &= \dim \SO_{\nu} + |\la^{(1)}| \\
                &= n^2 - n - 2n(\Bla) + |\la^{(1)}|.
\end{align*}

\para{7.4.}
For $\Bla \in \SP_{n,3}$, we shall define a variety $X_{\Bla} \subset \Fh \times V$.
Put $\Bla = (\la^{(1)}, \la^{(2)}, \la^{(3)})$, and $m_i = |\la^{(i)}|$ for 
$i = 1,2,3$. Let $P$ be the parabolic subgroup of $H$, which is the 
stabilizer of the subspace $V_0 = M_{m_1}$, 
and $L$ the Levi subgroup of $P$ containing $T$.  
Then $L \simeq GL(V_0) \times Sp(\ol V_0)$, where $\ol V_0 = V_0^{\perp}/V_0$.
Put $G_1 = GL(V_0)$ and $\Fg_1 = \Lie G_1$. 
$G_1$ acts on $\Fg_1 \times V_0$, as the restriction of the action of $H$ on $\Fh \times V$. 
Let $\SO_1$ 
be the $G_1$-orbit in $(\Fg_1)\nil \times V_0$ corresponding to $\OO_{\Bla}$ 
with $\Bla = (\la^{(1)}, \emptyset) \in \SP_{m_1,2}$.  
Put $H_2 = Sp(\ol V_0)$ and $\Fh_2 = \Lie H_2$. 
We denote by $\SO_2$ the $H_2$-orbit $\SO_{\Bla'}$ in $(\Fh_2)\nil$ 
corresponding to $\Bla' = (\la^{(2)}, \la^{(3)}) \in \SP_{m_2 + m_3, 2}$. 
Put $\Lie P = \Fp$.  
We define a set $\SM_{\Bla} \subset \Fp\nil \times V_0$ by 
\begin{equation*}
\tag{7.4.1}
\SM_{\Bla} = \{ (x,v) \in \Fp\nil \times V_0 \mid  
            (x|_{V_0} , v) \in \SO_1, x|_{\ol V_0} \in \SO_2 \},  
\end{equation*} 
and define $X_{\Bla} = \bigcup_{g \in H}g(\SM_{\Bla})$. 
Let $\Fn_P$ be the nilpotent radical of $\Fp$.  We define a variety 
\begin{equation*}
\tag{7.4.2}
\wt X_{\Bla} = H \times^P((\ol\SO_1 + \ol\SO_2) + \Fn_P)
\end{equation*}
and a map $\pi_{\Bla} : \wt X_{\Bla} \to \SX\nil$ by 
$g*x \mapsto gx$.  Then $\pi_{\Bla}$ is proper, and 
$\Im \pi_{\Bla} = \bigcup_{g \in H}g(\ol\SO_1 + \ol\SO_2 + \Fn_P)$ is closed in $\SX\nil$.

We show that 
\begin{lem}  
Under the notation above, 
\begin{enumerate}
\item
$X_{\Bla}$ is an $H$-stable,  smooth,  irreducible, 
locally closed subvariety of $\SX\nil$.
Moreover, $\Im \pi_{\Bla} = \ol X_{\Bla}$. 
\item  $\dim \wt X_{\Bla} = \dim X_{\Bla} = 2\dim U_P + \dim \SO_1 + \dim \SO_2$. 
\item  $X_{\Bla}$ gives a partition of $\SX\nil$, namely, 
\begin{equation*}
\SX\nil = \coprod_{\Bla \in \SP_{n,3}}X_{\Bla}.
\end{equation*}
\end{enumerate}
\end{lem}

\begin{proof}
Take $(x,v) \in L \times M_n$, with $x = (x_1, x_2)$ such that 
$(x_1,v) \in \SO_1$ and $x_2 \in \SO_2$. 
Then we can write as
\begin{equation*}
\tag{7.5.1}
X_{\Bla} \simeq H \times^{Z_L(x,v)U_P}\bigl((x + \Fn_P) \times \{ v\}\bigr).
\end{equation*}
Hence $X_{\Bla}$ is smooth and irreducible. 
We have 
\begin{align*}
\dim X_{\Bla} &= \dim H - \dim Z_L(x,v) \\
              &= \dim H - \dim Z_{G_1}(x_1,v) - \dim Z_{H_2}(x_2)  \\
              &= 2\dim U_P + \dim \SO_1 + \dim \SO_2.
\end{align*}
Thus by (7.4.2), we see that $\dim \wt X_{\Bla} = \dim X_{\Bla}$. 
(ii) is proved. 
\par
We have $X_{\Bla} \subset \Im \pi_{\Bla}$.  As a variant of $\pi_{\Bla}$, 
it is possible to define $\pi'_{\Bla}$ by replacing $\SO_i$ 
by any orbits $\SO_i' \subset \ol\SO_i$ for $i = 1,2$.  In that case, $\Im \pi'_{\Bla}$ 
is also closed. This implies that $X_{\Bla}$ is an open dense subset of $\Im \pi_{\Bla}$, 
hence $X_{\Bla}$ is locally closed in $\SX\nil$, and $\Im \pi_{\Bla} = \ol X_{\Bla}$.
This proves (i).  
\par
We show (iii).  Take $(x,v) \in \SX\nil$. Up to $H$-conjugate, we may assume 
that $(x,v) \in \Fh\nil \times M_n$. Let $x'$ be the restriction of $x$ on $M_n$.
Let $W = E^{x'}v$ for $V_1 = M_n$ 
under the notation in 7.3.  Let $\la^{(1)}$ be the Jordan type of $x'|_W$.
Then $(x'|_W,v)  \in \OO_{(\la^{(1)}, -)}$ in $W$. 
Let $P$ be the stabilizer of $W$.  Then $x'' = x|_{\ol W}$ gives an element 
in $\Fs\Fp(\ol W)\nil$.  Assume that $x'' \in \SO_{\Bla'}$ with 
$\Bla'' = (\la^{(2)}, \la^{(3)})$.  Then $(x,v) \in X_{\Bla}$ with 
$\Bla = (\la^{(1)}, \la^{(2)}, \la^{(3)}) \in \SP_{n,3}$.  This implies that 
$\SX\nil = \bigcup_{\Bla}X_{\Bla}$.  The above discussion 
shows that the $H$-conjugates of $(x,v)$determines $\Bla$ uniquely.  Hence 
they are disjoint, and (iii) holds.   
\end{proof}

\remark{7.6.}
It follows from the construction, $X_{\Bla}$ is a single
$H$-orbit if $m_1 = 0$. Since 
$\Fg^{\th}$ contains infinitely many $H$-orbits by Proposition 1.8, 
some of $X_{\Bla}$ contains infinitely many $H$-orbits. 

\para{7.7.}
For $\Bm \in \SQ_{n,3}$, we define $\Bla = \Bla(\Bm) \in \SP_{n,3}$ by 
$\la^{(i)} = (m_i)$ for $i = 1,2,3$.  Also we define a subset $\SP(\Bm)$ 
of $\SP_n$ as the set of $\Bla \in \SP_{n,3}$ such that 
$|\la^{(i)}| = m_i$ for $i = 1,2,3$. 
For $\Bm \in \SQ^0_{n,3}$, put
\begin{equation*}
\tag{7.7.1}
\wt\SP(\Bm) = \coprod_{0 \le k \le m_2}\SP(\Bm(k)).
\end{equation*}
We have the following result.

\begin{prop}  
Assume that $\Bm \in \SQ_{n,3}^0$.  Then we have
\begin{enumerate}
\item 
$\SX_{\Bm, \nilp} = \ol X_{\Bla(\Bm)}$. 
\item
$\dim \wt\SX_{\Bm,\nilp} = \dim \SX_{\Bm,\nilp} = 2n^2 + m_1$.
\item 
For $\Bmu \in \wt\SP(\Bm)$, we have $X_{\Bmu} \subset \SX_{\Bm, \nilp}$.  
\end{enumerate}
\end{prop}

\begin{proof}
We show (iii).  Since $m_3 = 0$, we have $\SD_{m_1 + m_2} = \SD$, 
hence we can write as 
\begin{equation*}
\SX_{\Bm, \nilp} = \bigcup_{g \in H}g(\Fn \times M_{m_1}).
\end{equation*} 
Take $(x, v) \in X_{\Bmu}$. Up to $H$-conjugate, we may assume
that $(x,v) \in \SM_{\Bmu}$, where $\SM_{\Bmu}$ is as in (7.4.1). 
Assume that $\Bmu \in \wt\SP(\Bm)$.  
Let $P$ be the stabilizer of $W = M_{m_1}$, and $L \simeq GL(W) \times Sp(\ol W)$. 
We may further assume that $(x,v)$ is of the form 
$(x_1 + x_2 + z, v) \in \Fh\nil \times V$, 
where $(x_1,v) \in \SO_1, x_2 \in \SO_2, z \in \Fn_P$ in the notation of 7.4.. 
Hence, up to $H$-conjugate,  we can take $v \in M_{m_1}$ and $x \in \Fn$.
It follows that $X_{\Bmu} \subset \SX_{\Bm\nilp}$. This proves (iii).
Now by (iii), $X_{\Bla(\Bm)} \subset \SX_{\Bm, \nilp}$. By Lemma 7.5 (ii), 
$\dim X_{\Bla(\Bm)} = 2\dim U_P + \dim \SO_1 + \dim \SO_2$. 
Since $\SO_1$ corresponds to $((m_1), -)$, we have 
$\dim \SO_1 = m_1^2$ by (7.3.2). On the other hand, since $\SO_2$ corresponds to 
$((m_2), -)$, which is the regular nilpotent orbit in $(\Fh_2)\nil$.  Thus 
$\dim \SO_2 = \dim H_2 - m_2$.  It follows that 
\begin{align*}
\dim X_{\Bla(\Bm)} &= (\dim H - \dim G_1 - \dim H_2) + m_1^2 + (\dim H_2 - m_2) \\ 
              &= 2n^2 + m_1. 
\end{align*}  
By the previous discussion, we have 
\begin{equation*}
\dim X_{\Bla(\Bm)} \le \dim \SX_{\Bm,\nilp} \le \dim\wt\SX_{\Bm,\nilp}.
\end{equation*}
By (7.1.2) (note that $m_3 = 0$), the above inequalities are actually equalities. 
Since $X_{\Bla(\Bm)}$ is irreducible, we conclude that 
$\ol X_{\Bla(\Bm)} = \SX_{\Bm,\nilp}$. Hence (i) holds. 
(ii) also follows from this.  
\end{proof}

\para{7.9.}
Let $P = LU_P$ be a parabolic subgroup of $H$, where $L$ is 
a Levi subgroup, and $U_P$ is the unipotent radical of $P$. 
Let $\pi_P : P \to L$ be the natural projection. 
Let $\SO'$ be an $L$-orbit in $(\Lie L)\nil$. 
Let $\SO$ be an $H$-orbit in $\Fh\nil$ such that 
$\SO \cap \pi_P\iv(\SO')$ is open dense in $\pi_P\iv(\SO')$.  
For any $x \in \SO$, consider the variety 
\begin{equation*}
\tag{7.9.1}
\SP_x = \{gP \in H/P \mid g\iv x \in \pi_P\iv(\SO') \}. 
\end{equation*}
Then clearly $\SP_x \ne \emptyset$.  
The following result can be proved in a similar way as [Sh, Prop. 6.3].
\begin{prop}  
Under the setting in 7.9, assume that $x \in \SO$. 
\begin{enumerate}
\item 
$\SP_x$ consists of one point.
\item 
$\dim Z_H(x) = \dim Z_L(x')$ for $x' \in \SO'$.
\item
Let $x_1 \in \pi_P\iv(\SO')$ be such that $\dim Z_H(x_1) = \dim Z_H(x)$.  
Then $x_1 \in \SO$. 
\item
Take $x_1 \in \SO \cap \pi_P\iv(\SO')$ and put $x' = \pi_P(x_1)$. 
Let $Q = Z_P(x')$.  Then $\dim Z_Q(x_1) = \dim Z_H(x_1)$.  In particular, 
\begin{equation*}
Z_H(x_1) = Z_P(x_1) = Z_Q(x_1).
\end{equation*} 
\item 
$P$ acts transitively on $\SO \cap \pi_P\iv(\SO')$, and $Q$ acts transitively 
on $\SO \cap \pi_P\iv(x')$. 
\end{enumerate}
\end{prop} 

\begin{proof}
For the sake of completeness, we give an outline of the proof below.
First we show that
\begin{equation*}
\tag{7.10.1}
\dim \SP_x = 0.
\end{equation*}
Replacing by $H$-conjugate, we may assume that $x \in \SO \cap \pi_P\iv(\SO')$.
Put $x' = \pi_P(x) \in \SO'$. We have $\dim \pi_P\iv(x) = \dim U_P$. 
Put $c = \dim \SO, c' = \dim \SO'$. By [L1, Prop.1.2] (actually a similar 
argument works also for the Lie algebra case, see [X2, Prop. 3.1]), we have
$\dim (\SO \cap \pi_P\iv(x')) \le (c - c')/2$.  
Since $\SO \cap \pi_P\iv(\SO')$ is open dense in $\pi\iv(\SO')$, 
$\SO \cap \pi_P\iv(x')$ is open dense in $\pi_P\iv(x')$.  It follows that 
\begin{equation*}
\dim U_P \le (c - c')/2.
\end{equation*}
On the other hand, by Proposition 3.1 (ii) in [X2] (or [L1, Prop. 1.2 (b)]), we have
\begin{align*}
\tag{7.10.2}
\dim \SP_x &\le (\dim U - c/2) - (\dim U_L - c'/2) \\
           &= \dim U_P - (c - c')/2,
\end{align*}
where $U_L = U \cap L$ is a maximal unipotent subgroup in $L$. 
Hence $\dim \SP_x \le 0$.  Since $\SP_x \ne \emptyset$, we obtain (7.10.1). 
We also have $c - c' = 2\dim U_P$. This implies that $\dim Z_H(x) = \dim Z_L(x')$.
Hence (ii) holds. Now consider (iv). 
Based on the above computation, in a similar way as in [Sh, Prop. 6.3], we can show 
$\dim Z_H(x_1) = \dim Z_Q(x_1)$. Since it is known that $Z_H(x_1)$ is connected 
([X2, 2.14]) in the case of characteristic 2, we have $Z_H(x_1) = Z_Q(x_1)$.
Hence $Z_H(x_1) = Z_P(x_1)$, which proves (iv). 
Put
\begin{equation*}
\SU = \{ (x_1, gP) \in \Fh\nil \times H/P \mid g\iv x_1 \in \pi_P\iv(\SO')\}.
\end{equation*} 
Then $\SU \simeq H \times^P\pi_P\iv(\SO')$ and so $\SU$ is irreducible. 
Let $f : \SU \to \Fh\nil$ be the first projection, and put $\SU_{\SO} = f\iv(\SO)$. 
Then $\SU_{\SO} \simeq H \times^P(\SO \cap \pi_P\iv(\SO'))$, and so $\SU_{\SO}$ 
is also irreducible.  
For any $x \in \SO$, $\dim f\iv(x) = 0$ by (7.10.1).  Thus $\dim \SU_{\SO} = \dim \SO$.
Since $f : \SU_{\SO} \to \SO$ is an $H$-equivariant surjective map, for any 
$\xi \in \SU_{\SO}$, the $H$-orbit $H\xi$ is open dense in $\SU_{\SO}$.  It follows that 
\par\medskip\noindent
(7.10.3) \ 
$H$ acts transitively on $\SU_{\SO}$. 
\par\medskip

Take $x, x_1 \in \SO \cap \pi_P\iv(\SO')$. Since 
$(x, P), (x_1, P) \in \SU_{\SO}$, there exists $g\in P$ such that $gx = x_1$.  This proves
the first statement of (v). Now assume that $x_1, x_2 \in \SO \cap \pi_P\iv(x')$.
Then there exists $g \in P$ such that $gx_1 = x_2$.  But since $\pi_P$ is $P$-equivariant, 
$g \in Z_P(x') = Q$. This proves the second statement of (v). 
\par
We show (i). We may assume that $x \in \SO \cap \pi_P\iv(\SO')$.  
Then $P \in \SP_x$. Assume that $gP \in \SP_x$.  Then $(x, P), (x, gP) \in \SU_{\SO}$, 
and so there exists $h \in Z_H(x)$ such that $gP = hP$ by (7.10.3).  
But by (iv), $h \in P$, and so (i) holds. 
(iii) is proved in the same way as in [Sh].  
The proposition is proved. 
\end{proof}

\para{7.11.}
We return to the original setting. 
For later application, we shall consider some 
open dense subvariety $X_{\Bla}^0$ of $X_{\Bla}$. 
As in 7.4, let $\SO_1$ be a $G_1$-orbit in $(\Fg_1)\nil \times W$ 
and $\SO_2$ an $H_2$-orbit in $(\Fh_2)\nil$.
We denote by $\SO_1'$ the $G_1$-orbit in $(\Fg_1)\nil$ which is the 
projection of $\SO_1$ to $(\Fg_1)\nil$, hence $\SO_1' = \SO_{\la^{(1)}}$ 
(see 7.3).
We define a subset $\SM_{\Bla}^0$ of $\SM_{\Bla}$ as the set of $(x,v)$
such that the orbit $\SO$ containing $x$ satsifies the 
property that $\SO \cap \pi_P\iv(\SO'_1 \times \SO_2)$ is open dense 
in $\pi_P\iv(\SO'_1\times \SO_2) = (\SO'_1 \times \SO_2) + \Fn_P$.
Clearly, $\SM_{\Bla}^0$ is open dense in $\SM_{\Bla}$.  We define
\begin{equation*}
\tag{7.11.1}
X_{\Bla}^0 = \bigcup_{g \in H}g(\SM_{\Bla}^0).
\end{equation*}  
Let $\pi_{\Bla} : \wt X_{\Bla} \to \ol X_{\Bla}$ be as in 7.4.
We define 
\begin{equation*}
\wt X^0_{\Bla} = H \times^P\SM^0_{\Bla}. 
\end{equation*}
Since $\wt X_{\Bla} \simeq H \times^P\ol\SM_{\Bla}$, and $\SM_{\Bla}$ is 
open dense in $\ol\SM_{\Bla}$, $\wt\SX^0_{\Bla}$ is open dense in 
$\wt X_{\Bla}$, and one can check that $\pi_{\Bla}\iv(X^0_{\Bla}) = \wt X^0_{\Bla}$. 
Since $\pi_{\Bla}$ is proper, 
$X_{\Bla}^0$ is open dense in $X_{\Bla}$. 
Let $\pi_{\Bla}^0: \wt X^0_{\Bla} \to X^0_{\Bla}$ be the restriction of 
$\pi_{\Bla}$ on $\wt X^0_{\Bla}$.  We have a lemma.

\begin{lem}  
$\pi_{\Bla}^0$ gives an isomorphism $\wt X^0_{\Bla} \isom X^0_{\Bla}$. 
\end{lem} 
\begin{proof}
$\pi^0_{\Bla}$ is $H$-equivariant and surjective. Take $(x,v) \in \SM^0_{\Bla}$, 
and put $x' = \pi_P(x) = (x_1, x_2)$ with $(x_1,v) \in \SO_1, x_2 \in \SO_2$. 
Furthermore, by our assumption, $\SO \cap \pi_P\iv(\SO')$ is open dense in 
$\pi_P\iv(\SO')$, where $\SO$ is the $H$-orbit containing $x$, and 
$\SO' = \SO_1' \times \SO_2$. 
Hence, under the notation of (7.9.1),  we have
\begin{equation*}
\pi_{\Bla}\iv(x,v) \simeq \{ gP \in H/P \mid g\iv(x,v)\in \SM_{\Bla}^0\}
                     \subset \SP_x.
\end{equation*}
By Proposition 7.10, we have $\SP_x = \{ P \}$. It follows that 
$\pi_{\Bla}\iv(x,v) = \{ P\}$. 
This shows that $\pi^0_{\Bla}$ gives a bijective morphism from 
$\wt X^0_{\Bla}$ onto $X^0_{\Bla}$. 
The map $g(x,v) \mapsto gP$ gives a well-defined morphism 
$X^0_{\Bla} \to \wt X^0_{\Bla}$, which is the inverse of $\pi_{\Bla}$. 
\end{proof}

\para{7.13.}
From now on, we fix $\Bm \in \SQ_{n,3}^0$. 
Put $\SB = H/B$.
For any $z = (x,v) \in \SX_{\Bm,\nilp}$, define 
\begin{align*}
\SB_z &= \{ gB \in \SB \mid g\iv x \in \Fn, g\iv v \in M_n\}, \\
\SB_z^{(\Bm)} &= \{ gB \in \SB \mid g\iv x \in \Fn, g\iv v \in M_{m_1} \}.
\end{align*}
Hence $\SB_z^{(\Bm)} \subset \SB_z$ are closed subvarieties of $\SB$.
\par
For each integer $d \ge 0$, 
we define a subset $X(d)$ of $\SX_{\Bm,\nilp}$ by 
\begin{equation*}
\tag{7.13.1}
X(d) = \{ z \in \SX_{\Bm,\nilp} \mid \dim \SB_z^{(\Bm)} = d \}.
\end{equation*}
Then $X(d)$ is a locally closed subvariety of $\SX_{\Bm,\nilp}$, and 
$\SX_{\Bm,\nilp} = \coprod_{d \ge 0}X(d)$. 
We consider the Steinberg varieties $\CZ^{(\Bm)}$ and $\CZ^{(\Bm)}_1$, which 
are a generalization of the Steinberg variety considered in 4.4.

\begin{align*}
\CZ^{(\Bm)} &= \{ (z, gB, g'B) \in \SX \times \SB \times \SB 
                     \mid (z, gB) \in \wt\SX_{\Bm}, (z, g'B) \in \wt\SX_{\Bm}\}  \\
\CZ^{(\Bm)}_1 &= \{ (z, gB, g'B) \in \SX\nil \times \SB \times \SB 
                 \mid (z, gB) \in \wt\SX_{\Bm,\nilp}, (z,g'B) \in \wt\SX_{\Bm,\nilp}\}.  
\end{align*}

Recall, for $\Bm = (m_1, m_2, 0) \in \SQ_{n,3}^0$, that 
$W\nat_{\Bm} = S_{m_1} \times W_{m_2}$.
We show the following lemma.

\begin{lem}  
Under the notation in 7.13, 
\begin{enumerate}
\item 
$\dim \CZ^{(\Bm)}_1 = \dim \SX_{\Bm \nilp}$.   The set of irreducible components
of $\CZ^{(\Bm)}_1$ with maximal dimension is parametrized by $w \in W\nat_{\Bm}$.
\item  
$\dim \CZ^{(\Bm)} = \dim \CZ^{(\Bm)}_1 + n$.  The set of irreducible components 
of $\CZ^{(\Bm)}$ with maximal dimension is parametrized by $w \in W\nat_{\Bm}$.
\item 
Assume that $X(d) \ne \emptyset$. For any $z \in X(d)$, we have
\begin{equation*}
\dim \SB^{(\Bm)}_z \le \frac{1}{2}(\dim \SX_{\Bm,\nilp} - \dim X(d)).
\end{equation*}
In particular, $\pi_1^{(\Bm)} :  \wt\SX_{\Bm,\nilp} \to \SX_{\Bm,\nilp}$ 
is semismall with respect to the stratification 
$\SX_{\Bm\nilp} = \coprod_{d}X(d)$. 
\end{enumerate}
\end{lem}  

\begin{proof}
Let $p_1 : \CZ^{(\Bm)}_1 \to \SB \times \SB$ be the projection to the second 
and third factors. For each $w \in W_n$, let $\SO_w$ be the $H$-orbit of $(B, wB)$ in 
$\SB \times \SB$. We have $\SB \times \SB = \coprod_{w \in W_n}\SO_w$. 
Put $Z_w = p_1\iv(\SO_w)$. Then $\CZ^{(\Bm)}_1 = \coprod_{w \in W_n}Z_w$. 
Here $Z_w$ is a vector bundle over 
$\SO_w \simeq H/(B \cap wBw\iv)$ with fibre isomorphic to 
\begin{equation*}
(\Fn \cap w(\Fn)) \times (M_{m_1} \cap w(M_{m_1})).
\end{equation*} 
(Note that $\ol\FN_{p_2} = \Fn_s \oplus \FD_{m_1 +m_2} = \Fn$ since $m_3 = 0$.)
We have 
\begin{equation*}
\dim Z_w = \dim H - \dim T + \dim (M_{m_1} \cap w(M_{m_1})).
\end{equation*}
Here $\dim (M_{m_1} \cap w(M_{m_1})) \le m_1$, and the equality holds 
if and only if $w(M_{m_1}) = M_{m_1}$, namely $w \in W\nat_{\Bm}$. 
It follows that $\dim Z_w \le 2n^2 + m_1$ and the equality holds if and only if
$w \in W\nat_{\Bm}$.   Hence by Proposition 7.8, 
$\dim Z_w = \dim \SX_{\Bm,\nilp}$ if $Z_w$ has the maximal dimension.   
This implies that $\{ \ol Z_w \mid w \in W\nat_{\Bm}\}$ gives the set of 
irreducible components of $\CZ^{(\Bm)}_1$ with maximal dimension.  This proves (i). 
For (ii), we consider $\wt Z_w = p\iv(\SO_w)$, where
 $p: \CZ^{(\Bm)} \to \SB \times \SB$ is the projection to the second and third factors. 
Then $\wt Z_w$ is a vector bundle over $\SO_w$, with fibre isomorphic to 
\begin{equation*}
\Ft \times (\Fn \cap w(\Fn))\times (M_{m_1} \cap w(M_{m_1})).
\end{equation*} 
Hence (ii) is proved in a similar way as (i). 
\par 
We show (iii).  Let $q_1 : \CZ^{(\Bm)}_1 \to \SX_{\Bm, \nilp}$ be the projection 
on the first factor. For each $z \in \SX_{\Bm,\nilp}$,
$q_1\iv(z) \simeq \SB_z^{(\Bm)} \times \SB_z^{(\Bm)}$.   
By (7.13.1), we have
\begin{equation*}
\dim q_1\iv(X(d)) = \dim X(d) + 2d.
\end{equation*}
Since $\dim q_1\iv(X(d)) \le \dim \CZ^{(\Bm)}_1 = \dim \SX_{\Bm, \nilp}$, 
we see that $2d \le \dim \SX_{\Bm,\nilp} - \dim X(d)$. 
This proves (iii).  The lemma is proved. 
\end{proof}

\par\bigskip

\section{ Springer correspondence for $\Fg^{\th}$}
\par\medskip

\para{8.1.}
In this section we shall prove the Springer correspondence for 
$\Fg^{\th}$.  In [Sh], the Springer correspondence was established 
for the exotic symmetric space of level $r$ for arbitrary $r \ge 1$.
Once Theorem 6.10 and Theorem 6.21 are proved, a similar 
discussion as in [Sh, 7] can be applied to our situation as the special 
case where $r = 3$.   

Assume that $\Bm \in \SQ_{n,3}^0$. 
We consider 
the variety $\CZ^{(\Bm)}$ as in 7.13.
We denote by $\vf: \CZ^{(\Bm)} \to \SX_{\Bm}$ 
the map $(z,gB,g'B) \mapsto z$. 
Let $\a: \CZ^{(\Bm)} \to \Ft$ be the map defined by 
$(x,v, gB, g'B) \mapsto p_1(g\iv x)$, similarly to 4.4.
As in 4.4, we have a commutative diagram 
\begin{equation*}
\begin{CD}
\CZ^{(\Bm)} @>\a>> \Ft  \\
@V\vf VV              @VV\w_1V   \\
\SX_{\Bm} @>\wt\w>>   \Xi,
\end{CD}
\end{equation*}
where $\wt\w$ is the composite of the projection $\SX_{\Bm} \to \Fh$ and $\w$.
Put $\s = \w_1\circ \a$, and $d'_{\Bm} = \dim \SX_{\Bm, \nilp} = 2n^2 + m_1$. 
We define a constructible sheaf $\ST$ on $\Xi$ by 
\begin{equation*}
\tag{8.1.1}
\ST = \SH^{2d'_{\Bm}}(\s_!\Ql) = R^{2d'_{\Bm}}\s_!\Ql. 
\end{equation*}

Note that if $m_1 = 0$, $\CZ^{(\Bm)}$ coincides with $Z$ in 4.4, and $\ST$ 
is nothing but the sheaf $\ST$ defined in (4.4.1). 
Also $\SF$ corresponds to the special case where $r = 3$ of the sheaf $\SF$ 
defined in [Sh, 7.1].  
The following properties of $\SF$ are obtained by similar 
arguments as in [Sh], so we omit the proof.  
The discussion in the proof
of Lemma 4.5 can be generalized to this case 
(see also [Sh, Lemma 7.2]), and we obtain 
\begin{lem}  
The sheaf $\ST$ is a perfect sheaf on $\Xi$. 
\end{lem}  
Let $p: \CZ^{(\Bm)} \to \SB \times \SB$ be the projection to the second and third
factors.  We put $\CZ^{(\Bm)}_w = p\iv(\SO_w)$ for $w \in W_n$, 
where $\SO_w$ is as in the proof of Lemma 4.5. 
Let $\s_w$ be the restriction of $\s$ on $\CZ^{(\Bm)}_w$, and put 
$\ST_w = \SH^{2d'_{\Bm}}((\s_w)_!\Ql)$.  In a similar way as in the proof of 
Proposition 4.6 (see also [Sh, Prop. 7.3], 
here the role of $W_n$ is replaced by $W\nat_{\Bm}$), we can prove
\begin{prop}  
$\ST \simeq \bigoplus_{w \in W\nat_{\Bm}}\ST_w$.
\end{prop}   

\para{8.4.}
By the K\"unneth formula, we have 
$\vf_!\Ql \simeq \pi^{(\Bm)}_!\Ql \otimes \pi^{(\Bm)}_!\Ql$. By Theorem 6.21, 
$\pi^{(\Bm)}_!\Ql$ has a natural structure of $W\nat_{\Bm}$-module. 
Hence $\vf_!\Ql$ has a structure of $W\nat_{\Bm} \times W\nat_{\Bm}$-module. 
It follows that 
$\ST = \SH^{2d'_{\Bm}}(\s_!\Ql) \simeq \SH^{2d'_{\Bm}}(\wt\w_!(\vf_!\Ql))$ has 
a natural action of $W\nat_{\Bm} \times W\nat_{\Bm}$. 
Under the decomposition of $\ST$ in Proposition 8.3, the action of 
$W\nat_{\Bm} \times W\nat_{\Bm}$ has the following property (see the discussion 
in 4.7). For each 
$w_1, w_2 \in W\nat_{\Bm}$, 
\begin{equation*}
\tag{8.4.1}
(w_1,w_2)\cdot \ST_w = \ST_{w_1ww_2\iv}.
\end{equation*} 
\par
Let $a_0$ be the element in $\Xi$ corresponding to the $S_n$-orbit of $0 \in \Ft$, 
and $\ST_{a_0}$ be the stalk of $\ST$ at $a_0 \in \Xi$. By Proposition 8.3, we have
a decomposition 
\begin{equation*}
\tag{8.4.2}
\ST_{a_0} = \bigoplus_{w \in W\nat_{\Bm}}(\ST_w)_{a_0}, 
\end{equation*}
where $(\ST_w)_{a_0}$ is the stalk of $\ST_w$ at $a_0$. 
$W\nat_{\Bm} \times W\nat_{\Bm}$ acts on $\ST_{a_0}$ following (8.4.1).
In a similar way as in (4.5.2), one can show that $\ST_w \simeq (\w_1)_!\Ql$.
Since $\w_1\iv(a_0) = \{ 0\} \in \Ft$, 
$(\ST_w)_{a_0} \simeq H_c^0(\w_1\iv(a_0),\Ql) \simeq \Ql$. 
Thus we have proved the following result, which is the stalk version of 
Proposition 4.8.
\begin{prop}  
$\ST_{a_0}$ has a structure of $W\nat_{\Bm} \times W\nat_{\Bm}$-module, which 
coincides with the two-sided regular representation of $W\nat_{\Bm}$. 
\end{prop}

\para{8.6.}
We consider the map $\pi_1^{(\Bm)} : \wt\SX_{\Bm\nilp} \to \SX_{\Bm,\nilp}$, 
and put $K_{\Bm,1} = (\pi_1^{(\Bm)})_!\Ql[d'_{\Bm}]$. 
By Lemma 7.14 (iii), $\pi_1^{(\Bm)}$ is semismall. Hence $K_{\Bm,1}$ is a 
semisimple perverse sheaf on $\SX_{\Bm,\nilp}$, and is decomposed as
\begin{equation*}
\tag{8.6.1}
K_{\Bm,\nilp} \simeq \bigoplus_{A}V_A\otimes A,
\end{equation*}
where $A$ is a simple perverse sheaf which is isomorphic to a direct 
summand of $K_{\Bm,1}$, and 
$V_A = \Hom (K_{\Bm,1}, A)$ is the multiplicity space for $A$. 
The following result is a counter part of Proposition 4.11 to the case of nilpotent 
variety (see also Proposition 7.8 in [Sh]). 
\begin{prop}  
Under the notation above, put $m_A = \dim V_A$ for each direct summand $A$.
Then we have
\begin{equation*}
\sum_{A}m_A^2 = |W\nat_{\Bm}|.
\end{equation*}
\end{prop}
\begin{proof}
By 8.4, we have
\begin{align*}
\ST_{a_0} &\simeq \BH^{2d'_{\Bm}}_c(\SX_{\Bm\nilp}, 
                    \pi^{(\Bm)}_!\Ql\otimes \pi^{(\Bm)}_!\Ql) \\
          &\simeq \BH^0_c(\SX_{\Bm,\nilp}, K_{\Bm,1}\otimes K_{\Bm,1}).
\end{align*}
By applying (8.6.1), we have 
\begin{equation*}
\dim \ST_{a_0} \simeq \sum_{A,A'}(m_Am_{A'})
          \dim \BH^0_c(\SX_{\Bm,\nilp}, A\otimes A'). 
\end{equation*}
Apply Lemma 4.9 for $X = \SX_{\Bm,\nilp}$. Then 
$\BH^0_c(\SX_{\Bm,\nilp}, A\otimes A') \ne 0$ only when $D(A) = A'$, in which case
$\dim \BH^0_c(\SX_{\Bm,\nilp}, A \otimes A') = 1$. But since $K_{\Bm,1}$ is self-dual, 
$m_A = m_{D(A)}$ for each $A$. It follows that $\dim \ST_{a_0} = \sum_A m_A^2$. 
On the other hand, by Proposition 8.5, we have $\dim \ST_{a_0} = |W\nat_{\Bm}|$. 
This proves the proposition.  
\end{proof}

\para{8.8.}
We consider the map $\pi^{(\Bm)} : \wt\SX_{\Bm} \to \SX_{\Bm}$. By Theorem 6.21, 
$\pi^{(\Bm)}_!\Ql[d_{\Bm}]$ is a semisimple perverse sheaf on $\SX_{\Bm}$, 
equipped with $W\nat_{\Bm}$-action, and is decomposed as
\begin{equation*}
\tag{8.8.1}
\pi^{(\Bm)}_!\Ql[d_{\Bm}] \simeq \bigoplus_{\r \in (W\nat_{\Bm})\wg}
                            \r \otimes K_{\r},
\end{equation*} 
where $K_{\r}$ is a simple perverse sheaf on $\SX_{\Bm}$, and 
$d_{\Bm} = \dim \SX_{\Bm}$. 
More precisely, 
it is given as
$K_{\r} = \IC(\SX_{\Bm(k)}, \SL_{\r_1})[d_{\Bm(k)}]$ if 
$\r = \r_1\nat$ for $\r_1 \in S_{\Bm(k)}\wg$.    
We consider the complex $(\pi_1^{(\Bm)})_!\Ql[d'_{\Bm}]$ for the map
$\pi_1^{(\Bm)}:\wt\SX_{\Bm, \nilp} \to \SX_{\Bm,\nilp}$, 
where $d'_{\Bm} = \dim \SX_{\Bm, \nilp}$.
The following result gives the Springer correspondence for 
$W\nat_{\Bm}$. (Compare Theorem 8.9 and Corollary 8.11 with Theorem 7.12 and 
Corollary 7.14 in [Sh]). 
\begin{thm}[Springer correspondence for $W\nat_{\Bm}$]  
Assume that $\Bm \in \SQ^0_{n,3}$. Then $(\pi^{(\Bm)}_1)_!\Ql[d'_{\Bm}]$ 
is a semisimple perverse sheaf on $\SX_{\Bm,\nilp}$, equipped with $W\nat_{\Bm}$-action, 
and is decomposed as 
\begin{equation*}
\tag{8.9.1}
(\pi^{(\Bm)}_1)_!\Ql[d'_{\Bm}] \simeq \bigoplus_{\r \in (W\nat_{\Bm})\wg}\r \otimes L_{\r},
\end{equation*}
where $L_{\r}$ is a simple perverse sheaf on $\SX_{\Bm,\nilp}$ such that
\begin{equation*}
\tag{8.9.2}
K_{\r}|_{\SX_{\Bm,\nilp}} \simeq L_{\r}[d_{\Bm} - d'_{\Bm}].
\end{equation*}
\end{thm}

\begin{proof}
As discussed in (8.6.1), $K_{\Bm,1} = (\pi_1^{(\Bm)})_!\Ql[d'_{\Bm}]$ is 
a semisimple perverse sheaf. Since $K_{\Bm,1}$ is the restriction of 
$\pi^{(\Bm)}_!\Ql$ on $\SX_{\Bm,\nilp}$, we have a natural homomorphism 
\begin{equation*}
\a : \Ql[W\nat_{\Bm}] \simeq \End (\pi^{(\Bm)}_!\Ql) \to \End K_{\Bm,1}.
\end{equation*}
In order to prove (8.9.1), it is enough to show that $\a$ is an isomorphism. 
By Proposition 8.7, we have $\dim \End K_{\Bm,1} = |W\nat_{\Bm}|$. Thus we have
only to show that $\a$ is injective. 
Note that $\ST_{a_0} = \BH^0_c(\SX_{\Bm,\nilp}, K_{\Bm,1}\otimes K_{\Bm,1})$,
and $K_{\Bm,1}$ is decomposed as 
$K_{\Bm,1} = \bigoplus_{\r}\r\otimes (K_{\r}|_{\SX_{\Bm,\nilp}})$ by (8.8.1).
This decomposition determines the $W\nat_{\Bm} \times W\nat_{\Bm}$-module structure 
of $\ST_{a_0}$. But by Proposition 8.5, $\ST_{a_0}$ is isomorphic to 
the two-sided regular 
representation of $W\nat_{\Bm}$.  This implies, in particular, 
$K_{\r}|_{\SX_{\Bm,\nilp}} \ne 0$ for any $\r \in W\nat_{\Bm}$.  
Hence $\a$ is injective, and so (8.9.1) holds.  Now (8.9.2) follows 
by comparing (8.8.1) and (8.9.1).  The theorem is proved. 
\end{proof}

\para{8.10.}
For each $\Bm \in \SQ^0_{n,3}$, we denote by $(W\wg_{n,3})_{\Bm}$ the set 
of irreducible representations $\wh\r \in W\wg_{n,3}$ corresponding to 
$\r \in S\wg_{\Bm(k)}$ for various $0 \le k \le m_2$. 
Then we have
\begin{equation*}
\tag{8.10.1}
W\wg_{n,3} = \coprod_{\Bm \in \SQ^0_{n,3}}(W\wg_{n,3})_{\Bm}.
\end{equation*}
For each $\r \in S_{\Bm(k)}\wg$, we can construct $\r\nat \in (W\nat_{\Bm})\wg$ 
as in 6.7, 
and the map $\r \mapsto \r\nat$ gives a bijective correspondence 
\begin{equation*}
\tag{8.10.2}
\coprod_{0 \le k \le m_2}S_{\Bm(k)}\wg \simeq (W\nat_{\Bm})\wg.
\end{equation*}
It follows that the correspondence $\wh\r \lra \r \lra \r\nat$ gives 
a bijective correspondence
\begin{equation*}
\tag{8.10.3}
(W\wg_{n,3})_{\Bm} \simeq (W\nat_{\Bm})\wg, \qquad \wh\r \lra \r\nat. 
\end{equation*}
\par
We consider the map $\ol\pi_{\Bm} : \pi\iv(\SX_{\Bm}) \to \SX_{\Bm}$. 
Then by Theorem 6.10, $(\ol\pi_{\Bm})_!\Ql[d_{\Bm}]$ is a semisimple perverse 
sheaf, equipped with $W_{n,3}$-action, and is decomposed as 
\begin{equation*}
(\ol\pi_{\Bm})_!\Ql[d_{\Bm}] \simeq \bigoplus_{\wh\r \in (W\wg_{n,3})_{\Bm}}
                  \wh\r \otimes K_{\r\nat},
\end{equation*} 
where $K_{\r\nat}$ is a simple perverse sheaf on $\SX_{\Bm}$ as defined 
in (8.8.1). 
Let $\ol\pi_{\Bm,1}: \pi\iv(\SX_{\Bm,\nilp}) \to \SX_{\Bm,\nilp}$ be the restriction 
of $\ol\pi_{\Bm}$ on $\SX_{\Bm,\nilp}$. By applying (8.9.2), we see that 
$(\ol\pi_{\Bm,1})_!\Ql[d'_{\Bm}]$ is a semisimple perverse sheaf.  As a corollary to 
Theorem 8.9, we obtain the Springer correspondence for $W_{n,3}$. 

\begin{cor}[Springer correspondence for $W_{n,3}$]   
Assume that $\Bm \in \SQ^0_{n,3}$.  Then $(\ol\pi_{\Bm,1})_!\Ql[d'_{\Bm}]$ 
is a semisimple perverse sheaf on $\SX_{\Bm,\nilp}$, equipped with $W_{n,3}$-action, 
and is decomposed as
\begin{equation*}
\tag{8.11.1}
(\ol\pi_{\Bm,1})_!\Ql[d'_{\Bm}] \simeq \bigoplus_{\wh\r \in (W\wg_{n,3})_{\Bm}}
                 \wh\r \otimes L_{\r\nat},
\end{equation*}
where $L_{\r\nat}$ is the simple perverse sheaf on $\SX_{\Bm,\nilp}$ 
as given in Theorem 8.9.  
\end{cor} 

\par\bigskip
\section {Determination of the Springer correspondence}

\para{9.1.}
In this section, we shall determine $L_{\r}$ appearing in the Springer 
correspondence explicitly. 
Let $\Bm = (m_1, m_2, 0) \in \SQ^0_{n,3}$.  We define a variety $\CG_{\Bm}$ by 
\begin{align*}
\CG_{\Bm} = \{ (x,&v, W_1)  \mid (x,v) \in \SX_{\Bm}, 
                   W_1 \text{ : isotropic, } \\
            &\dim W_1 = m_1, x(W_1) \subset W_1, v \in W_1 \}.
\end{align*}
Let $\z: \CG_{\Bm} \to \SX_{\Bm}$ be the projection to the first two factors. 
Then the map $\pi^{(\Bm)} : \wt\SX_{\Bm} \to \SX_{\Bm}$ is factored as
\begin{equation*}
\tag{9.1.1}
\begin{CD}
\pi^{(\Bm)} : \wt\SX_{\Bm} @>\vf>>  \CG_{\Bm} @>\z>> \SX_{\Bm}, 
\end{CD}
\end{equation*}
where $\vf$ is defined by $(x,v,gB) \mapsto (x,v, gM_{m_1})$. 
$\vf$ is surjective since there exists an $x$-stable maximal isotropic subspace 
containing $W_1$. 
$\z$ is also surjective since $\pi^{(\Bm)}$ is surjective. 
Since $\Bm \in \SQ^0_{n,3}$, we have 
$\dim \wt\SX_{\Bm} = \dim \SX_{\Bm}$. 
It follows that $\dim \CG_{\Bm} = \dim \SX_{\Bm}$. 
\par
In the case where $m_1 = 0$, $\pi^{(\Bm)} : \wt\SX_{\Bm} \to \SX_{\Bm}$ 
coincides with the map $\pi : \wt X \to X$ (since $m_2 = n$), and 
(9.1.1) can be written as $\pi : \wt X \to \CG = X = \Fh$, where 
$\pi = \vf, \z = \id$.  
\par
By modifying the definition of $\CK_{\Bm}$ in 6.14, we define a variety 
$\CH_{\Bm}$ by 

\begin{align*}
\CH_{\Bm} = \{ (x,v, &W_1, \f_1, \f_2) \mid (x,v, W_1) \in \CG_{\Bm}, \\
                 &\f_1 : W_1 \isom V_0, \f_2 : W_1^{\perp}/W_1 \isom \ol V_0
                      \text{ (symplectic isom.)} \},
\end{align*}
where $V_0 = M_{m_1}$ and $\ol V_0 = V_0^{\perp}/V_0$. 
We also define a variety $\wt\CZ_{\Bm}$ by
\begin{align*}
\wt\CZ_{\Bm} = \{(x,v, &gB, \f_1, \f_2) \mid (x,v,gB) \in \wt\SX_{\Bm}, \\
                 &\f_1 : gM_{m_1} \isom V_0, 
                 \f_2 : (gM_{m_1})^{\perp}/gM_{m_1} \isom \ol V_0 \}.  
\end{align*}

As in 6.14, we consider $G_1 = GL(V_0), H_2 = Sp(\ol V_0)$, and $\Fh_2 = \Lie H_2$.
The maps $\pi^2: \wt X' \to X' = \Fh_2$, 
$\pi^1: \wt\Fg_1 \to \Fg_1$ are given as in 6.14.   
We have the following commutative diagram

\begin{equation*}
\tag{9.1.3}
\begin{CD}
\wt\Fg_1 \times \wt X' @<\wt\s<< \wt\CZ_{\Bm} @>\wt q>> \wt\SX_{\Bm}  \\
       @V\pi^1 \times \pi^2VV      @VV\wt\vf V              @VV\vf V     \\
    \Fg_1 \times X'  @<\s<<  \CH_{\Bm}  @>q>> \CG_{\Bm}   \\
         @.        @.                  @VV\z V    \\
                        @.              @.     \SX_{\Bm},    
\end{CD}
\end{equation*}
where morphisms are defined as 
\begin{align*}
q: &(x,v,W_1, \f_1, \f_2) \mapsto (x,v, W_1), \\
\s : &(x,v, W_1, \f_1, \f_2) \mapsto (\f_1(x|_{W_1})\f_1\iv, 
                                       \f_2(x|_{W_1^{\perp}/W_1})\f_2\iv), \\
\wt\vf : &(x,v, gB, \f_1, \f_2) \mapsto (x, v, (gM_{m_1}), \f_1, \f_2).  
\end{align*}
$\wt\s, \wt q$ are defined naturally. 
\par
One can check that both squares are cartesian squares.  Moreover, it is 
easy to see that
\par\medskip\noindent
(9.1.4) \ $q$ is a principal bundle with fibre isomorphic to $G_1 \times H_2$, and 
$\s$ is a locally trivial fibration with smooth connected fibre of 
dimension  $\dim H + m_1$. 

\para{9.2.}
For a fixed $k$, we consider the variety 
$\wt\SY^+_{\Bm(k)} = (\psi^{(\Bm)})\iv(\SY^0_{\Bm(k)})$ as in 6.2, and 
let $\CG_{\Bm(k),\sr} = \z\iv(\SY^0_{\Bm(k)})$ be the locally closed subvariety 
of $\CG_{\Bm}$.  Here the varieties $Y'^0_k, \wt Y'^+_k$ are defined similarly 
as in Section 2 by replacing $X$ by $X'$. 
As the restriction of (9.1.3), we have the following commutative diagram
\begin{equation*}
\tag{9.2.1}
\begin{CD}
\wt\Fg_{1,\rg} \times \wt Y_k @<<<  \wt\CZ^+_{\Bm(k)} @>>> \wt\SY^+_{\Bm(k)} \\
      @VVV                            @VVV                  @VV\vf_0V   \\
\Fg_{1,\rg} \times Y^0_k    @<<<  \CH_{\Bm(k), \sr} @>>>  \CG_{\Bm(k), \sr}  \\ 
       @.                   @ .                   @VV\z_0V         \\
                            @ .             @ .          \SY^0_{\Bm(k)}, 
\end{CD}
\end{equation*}
where $\CH_{\Bm(k), \sr} = q\iv(\CG_{\Bm(k), \sr})$ and 
$\wt\CZ^+_{\Bm(k)} = \wt q\iv(\wt \SY^+_{\Bm(k)})$, and $\vf_0, \z_0$ 
are restrictions of $\vf, \z$, respectively.  
The following result can be proved in a similar way as [Sh, (8.2.4)].
\par\medskip
\noindent
(9.2.2) \  The map $\z_0$ gives an isomorphism 
$\CG_{\Bm(k), \sr} \simeq \SY^0_{\Bm(k)}$. 
\par\medskip

Take $\r \in (W\nat_{\Bm})\wg$.  Then by (8.10.2), there exists an integer 
$k$ and $\r_0 \in S_{\Bm(k)}\wg$ such that $\r = \r_0\nat$. Thus $K_{\r}$ in 
(8.8.1) is given by $K_{\r} = \IC(\SX_{\Bm(k)}, \SL_{\r_0})[d_{\Bm(k)}]$, 
where $\SL_{\r_0}$ is a simple local system on $\SY^0_{\Bm(k)}$. 
By (9.2.2), we can regard $\SL_{\r_0}$ as a simple local system on 
$\CG_{\Bm(k), \sr}$. 
We put $A_{\r} = \IC(\ol\CG_{\Bm(k),\sr}, \SL_{\r_0})[d_{\Bm(k)}]$.  Then  
$A_{\r}$ is an $H$-equivariant simple perverse sheaf on $\ol\CG_{\Bm(k),\sr}$, 
which we regard as a perverse sheaf on $\CG_{\Bm}$ by extension by zero. 
\par
We show the following result.
\begin{prop}  
Under the setting in 9.2, we have
\begin{enumerate}
\item 
$\vf_!\Ql[d_{\Bm}]$ is a semisimple perverse sheaf, equipped with $W\nat_{\Bm}$-action, 
and is decomposed as 
\begin{equation*}
\tag{9.3.1}
\vf_!\Ql[d_{\Bm}] \simeq \bigoplus_{\r \in (W\nat_{\Bm})\wg}\r \otimes A_{\r}.
\end{equation*}
\item
$\z_!A_{\r} \simeq K_{\r}$. 
\end{enumerate}
\end{prop} 

\begin{proof}
By (5.14.1), we can write as 
\begin{equation*}
\tag{9.3.2}
(\pi^1)_!\Ql[\dim \Fg_1] \simeq \bigoplus_{\r_1 \in S_{m_1}\wg}
                           \r_1 \otimes A_{\r_1}
\end{equation*}
where $A_{\r_1} = IC(\Fg_1, \SL^1_{\r_1})[\dim \Fg_1]$ is a simple perverse 
sheaf on $\Fg_1$.
On the other hand, by Theorem 5.7, we can write as 

\begin{equation*}
\tag{9.3.3}
(\pi^2)_!\Ql[\dim \Fh_2]\simeq 
       \bigoplus_{\r' \in W_{m_2, 2}\wg} \r'\otimes A_{\r'},   
\end{equation*}
where $A_{\r'} = \IC(X'_k, \SL^2_{\r_2})[d_k]$ for some $k$ and  
$\r_2 \in (S_k \times S_{m_2-k})\wg$ such that $\r' = \wh\r_2$, which is a simple 
perverse sheaf on $X'$.    
By applying a similar argument as in 6.14 to the diagram (9.1.3), together with 
(9.1.4), one can find an $H$-equivariant simple perverse sheaf $\wt A_{\r}$
on $\CG_{\Bm}$ such that 
\begin{equation*}
\tag{9.3.4}
q^*\wt A_{\r}[\b_2] \simeq \s^*(A_{\r_1} \boxtimes A_{\r'})[\b_1], 
\end{equation*}
where $\b_1 = \dim H + m_1$, and $\b_2 = \dim G_1 + \dim H_2$. 
Moreover, $\r \in W\nat_{\Bm}$ is given by 
$\r = \r_1 \boxtimes \r' \in (S_{m_1}\times W_{m_2,2})\wg$.
By using a similar argument as in [Sh, 8.2], based on the diagram (9.2.1), 
one can show that the restriction of $\wt A_{\r}$ on $\CG_{\Bm(k),\sr}$ 
coincides with $\SL_{\r_0}$.  Hence we have  $\wt A_{\r} = A_{\r}$.    
\par
Put $K_1 = (\pi^1)_!\Ql[\dim \Fg_1], K_2 = (\pi^2)_!\Ql[\dim \Fh_2]$, and also 
put $K = \vf_!\Ql[d_{\Bm}]$. Since both squares in (9.1.3) are cartesian, we have
\begin{equation*}
q^*K[\b_2] \simeq \s^*(K_1\boxtimes K_2)[\b_1].
\end{equation*} 
In particular, $K$ is a semisimple perverse sheaf. 
It follows from the discussion based on the diagram (9.2.1), we see that 
$K$ has a natural action of $W\nat_{\Bm}$.  Then by using 
(9.3.2), (9.3.3) and (9.3.4), we obtain  (9.3.1).  This proves (i).
\par
Next we show (ii).  Since $\z$ is proper, $\z_!A_{\r}$ is a semisimple complex
on $\SX_{\Bm}$.  Since $\z_!K = \pi^{(\Bm)}_!\Ql[d_{\Bm}]$ is a semisimple 
perverse sheaf, $\z_!A_{\r}$ is also a semisimple perverse sheaf by (i). 
By applying $\z_!$ on both sides of (9.3.1), we have 
\begin{equation*}
\tag{9.3.5}
\pi^{(\Bm)}_!\Ql[d_{\Bm}] \simeq \bigoplus_{\r \in (W\nat_{\Bm})\wg}
                    \r \otimes \z_!A_{\r}.
\end{equation*}
By using the diagram (9.2.1), one can show that the $W\nat_{\Bm}$-module structure 
of $\pi^{(\Bm)}_!\Ql[d_{\Bm}]$ induced from $\z_!$ coincides with the 
$W\nat_{\Bm}$-structure given in the formula (8.8.1). 
Thus by comparing (9.3.5) with (8.8.1), we obtain (ii). 
The proposition is proved.
\end{proof}

\para{9.4.}
For each $\Bm \in \SQ^0_{n,3}$, put $\CG_{\Bm, \nilp} = \z\iv(\SX_{\Bm, \nilp})$. 
Then the map $\pi_1^{(\Bm)}$ is factored as 
\begin{equation*}
\begin{CD}
\pi_1^{(\Bm)} : \wt\SX_{\Bm,\nilp} @>\vf_1>> \CG_{\Bm, \nilp} @>\z_1>> \SX_{\Bm,\nilp},
\end{CD}
\end{equation*}
where $\vf_1, \z_1$ are restrictions of $\vf, \z$. Since $\pi^{(\Bm)}$ is surjective, 
$\vf_1$ is surjective.  Put $\CH_{\Bm, \nilp} = q\iv(\CG_{\Bm,\nilp})$.  The inclusion 
map $\CG_{\Bm,\nilp} \hra \CG_{\Bm}$ is compatible with the diagram (9.1.3), 
and we have a commutative diagram

\begin{equation*}
\tag{9.4.1}
\begin{CD}
\Fg_1 \times X'  @<\s<< \CH_{\Bm} @>q>>  \CG_{\Bm} \\
       @AAA              @AAA           @AAA    \\
(\Fg_1)\nil \times X'\nil  @<\s_1<<    \CH_{\Bm,\nilp} @>q_1>>  \CG_{\Bm,\nilp},
\end{CD}
\end{equation*}
\par\medskip\noindent
where $\s_1, q_1$ are restrictions of $\s, q$, respectively, and vertical maps 
are natural inclusions. A similar property as (9.1.4) still holds for $\s_1, q_1$,
and both squares are cartesian squares.
\par
For each $\Bla \in \SP(\Bm(k))$, we define a subset $\CG_{\Bla}$ of 
$\CG_{\Bm,\nilp}$ as follows. Write $\Bla = (\la^{(1)}, \la^{(2)}, \la^{(3)})$, 
where $|\la^{(1)}| = m_1, |\la^{(2)}| = k, |\la^{(3)}| = m_2 - k$. 
Let $\SO'_1 = \SO_{\la^{(1)}}$ be the $G_1$-orbit in $(\Fg_1)\nil$ and 
$\SO_2 = \SO_{(\la^{(2)},\la^{(3)})}$ be the $H_2$-orbit in $X'\nil = (\Fh_2)\nil$
(see the notation in 7.2). 
Put $\CG_{\Bla} = q_1(\s_1\iv(\SO'_1 \times \SO_2))$. 
Then $\CG_{\Bla}$ is an $H$-stable, irreducible smooth subvariety of 
$\CG_{\Bm, \nilp}$. 
\par
Let $\ol\CG_{\Bla}$ be the closure of $\CG_{\Bla}$ in $\CG_{\Bm,\nilp}$.
Recall the map $\pi_{\Bla} : \wt X_{\Bla} \to \ol X_{\Bla}$ defined in 7.4, 
and let $\wt X^0_{\Bla}$ be as in 7.11.
It follows from the construction that $\wt X_{\Bla}$ is a closed subset of 
$\CG_{\Bm,\nilp}$.  We show a lemma.

\begin{lem}  
\begin{enumerate}
\item  
$\ol\CG_{\Bla}$ coincides with $\wt X_{\Bla}$. In particular, 
$\wt X^0_{\Bla}$ is an open dense subset of $\ol\CG_{\Bla}$. 
\item 
$\z_1(\ol\CG_{\Bla}) = \ol X_{\Bla}$, and $\z_1\iv(X^0_{\Bla}) = \wt X^0_{\Bla}$.
Hence the restriction of $\z_1$ on $\z_1\iv(X^0_{\Bla})$ gives an isomorphism 
$\z_1\iv(X^0_{\Bla}) \isom X^0_{\Bla}$. 
\end{enumerate}
\end{lem}

\begin{proof}
It follows from the construction that $\wt X^0_{\Bla} \subset \CG_{\Bla}$. 
Hence $\wt X_{\Bla} \subset \ol\CG_{\Bla}$. 
By Lemma 7.5, $\dim \wt X_{\Bla} = 2\dim  U_P + \dim \SO_1 + \dim \SO_2$. 
On the other hand, 
\begin{align*}
\dim \CG_{\Bla} &= \dim \SO'_1 + \dim \SO_2 + \b_1 - \b_2  \\
                &= 2\dim U_P + \dim \SO_1 + \dim \SO_2.
\end{align*}            
(Here $\b_1, \b_2$ are as in (9.3.4), and $\dim \SO_1 = \dim \SO_1' + m_1$
by (7.3.1).)
Since $\wt X_{\Bla}$ is irreducible and closed, we have $\wt X_{\Bla} = \ol\CG_{\Bla}$. 
This proves (i).  Then the restriction of $\z_1$ on $\ol\CG_{\Bla}$ coincides with 
the map $\pi_{\Bla} : \wt X_{\Bla} \to \ol X_{\Bla}$. Thus (ii) follows from Lemma 7.12.
The lemma is proved.
\end{proof}

\para{9.6.}
Recall the set $\wt\SP(\Bm)$ in (7.7.1) for each $\Bm \in \SQ^0_{n,3}$. 
By (8.10.2), the set $(W\nat_{\Bm})\wg$ is parametrized by $\wt\SP(\Bm)$. 
We denote by $\r_{\Bla}\nat$ the irreducible representation of $W\nat_{\Bm}$ 
corresponding to $\Bla \in \wt\SP(\Bm)$. 
On the other hand, we denote by $\wh\r_{\Bla}$ the irreducible representation of 
$W_{n,3}$ belonging to $(W\wg_{n,3})_{\Bm}$ under the correspondence (8.10.3). 
By (8.10.1), we have a parametrization 
\begin{equation*}
W\wg_{n,3} \simeq \coprod_{\Bm \in \SQ^0_{n,3}}\wt\SP(\Bm).
\end{equation*}
The following result determines the Springer correspondence explicitly
(compare with [Sh, Thm. 8.7]). 
\begin{thm}   Assume that $\Bm \in \SQ^0_{n,3}$.   
\begin{enumerate}
\item 
Let $L_{\r}$ be as in Theorem 8.9.  Assume that 
$\r = \r\nat_{\Bla} \in (W\nat_{\Bm})\wg$ for $\Bla \in \wt\SP(\Bm)$. Then 
we have 
\begin{equation*}
\tag{9.7.1}
L_{\r} \simeq \IC(\ol X_{\Bla}, \Ql)[\dim X_{\Bla}]. 
\end{equation*} 
\item 
$($Springer correspondence for $W\nat_{\Bm}$$)$ 
\begin{equation*}
(\pi^{(\Bm)}_1)_!\Ql[d'_{\Bm}] \simeq \bigoplus_{\Bla \in \wt\SP(\Bm)}
                   \r\nat_{\Bla}\otimes \IC(\ol X_{\Bla}, \Ql)[\dim X_{\Bla}].
\end{equation*}
\item
$($Springer correspondence for $W_{n,3}$$)$
\begin{equation*}
(\ol\pi_{\Bm})_!\Ql[d'_{\Bm}] \simeq \bigoplus_{\Bla \in \wt\SP(\Bm)}
                  \wh\r_{\Bla} \otimes \IC(\ol X_{\Bla}, \Ql)[\dim X_{\Bla}]. 
\end{equation*}
\end{enumerate}
\end{thm}

\begin{proof}
By Proposition 9.3, we know that $\z_!A_{\r} = K_{\r}$ under the notation in 9.2.
Hence by the base change theorem , 
$(\z_1)_!(A_{\r}|_{\CG_{\Bm\nil}}) \simeq K_{\r}|_{\CG_{\Bm,\nil}}$. 
For $\Bla \in \SP(\Bm(k))$, put $\r = \r\nat_{\Bla}$. 
We define a simple perverse sheaf $B_{\Bla}$ on 
$\CG_{\Bm,\nilp}$ as follows. Let $\SO_1', \SO_2$ be as in 9.4.  
Put $B_{\r_1} = \IC(\ol\SO_1', \Ql)[\dim \SO_1']$ for 
$\r_1 = \r_{\la^{(1)}} \in S_{m_1}\wg$, and $B_{\r'} = \IC(\ol\SO_2, \Ql)[\dim \SO_2]$
for $\r' = \wh\r_2 \in W_{m_2,2}\wg$ with $\r_2 \in (S_k \times S_{m_2-k})\wg$.
By a similar construction as in the proof of Proposition 9.3, there exists a unique 
simple perverse sheaf $B_{\r}$ on $\CG_{\Bm,\nilp}$ satisfying the relation 
\begin{equation*}
\tag{9.7.2}
q_1^*B_{\Bla}[\b_2] \simeq \s^*(B_{\r_1}\boxtimes B_{\r'})[\b_1].
\end{equation*}
We know that $A_{\r_1}|_{(\Fg_1)\nil} \simeq B_{\r_1}$, up to shift.
On the other hand, by Corollary 5.20, we have 
$A_{\r'}|_{(\Fh_2)\nil} \simeq B_{\r'}$, up to shift.  
Thus by comparing (9.7.2) and (9.3.4), we see that the restriction of $A_{\r}$
on $\CG_{\Bm,\nilp}$ coincides with $B_{\r}$, up to shift. Also by (9.7.2), 
the restriction of $B_{\r}$ on $\CG_{\Bla}$ is a constant sheaf $\Ql$. In particular, 
$\supp B_{\r} = \ol\CG_{\Bla}$. By Lemma 9.5, the support of $(\z_1)_!B_{\Bla}$ 
coincides with $\ol X_{\Bla}$. By Theorem 8.9, we know that the restriction of 
$K_{\r}$ on $\SX_{\Bm\nilp}$ is a simple perverse sheaf $L_{\r}$. Hence in order
to show (9.7.1), it is enough to see that $L_{\r}|_{X^0_{\Bla}}$ is a constant sheaf $\Ql$. 
By Lemma 9.5 (ii), $\z\iv(X_{\Bla}^0) = \wt X^0_{\Bla} \subset \CG_{\Bla}$, and 
$\z_1\iv(X_{\Bla}^0) \isom X_{\Bla}^0$. It follows that 
$(\z_1)_!B_{\la}|_{X_{\Bla}^0}$ coincides with $\Ql$, up to shift. 
This proves (9.7.1). (ii) and (iii) then follows from Theorem 8.9 and Corollary 8.11.
The theorem is proved.  
\end{proof}

\para{9.8.}
For each $z \in \SX_{\Bm, \nilp}$, we consider the Springer fibres 
$\SB_z = \pi\iv(z)$ and $\SB^{(\Bm)}_z = (\pi^{(\Bm)})\iv(z)$ as in 7.13.
We have $\SB_z^{(\Bm)} \subset \SB_z$.  The cohomology group $H^i(\SB^{(\Bm)}_z,\Ql)$
has a structure of $W\nat_{\Bm}$-module, and $H^i(\SB_z, \Ql)$ has a structure of 
$W_{n,3}$-module. For $\Bla \in \wt\SP(\Bm)$, put
\begin{equation*}
\tag{9.8.1}
d_{\Bla} = \frac{1}{2}(\dim \SX_{\Bm,\nilp} - \dim X_{\Bla})
\end{equation*}
We have a lemma.

\begin{lem}  
Assume that $\Bla \in \wt\SP(\Bm)$. 
\begin{enumerate}
\item
For any $z \in X_{\Bla}$, $\dim \SB_z^{(\Bm)} \ge d_{\Bla}$. The set of $z \in X_{\Bla}$
such that $\dim \SB_z^{(\Bm)} = d_{\Bla}$ forms an open dense subset of $X_{\Bla}$.
\item
For any $z \in X_{\Bla}$, $H^{2d_{\Bla}}(\SB_z^{(\Bm)},\Ql)$ contains an irreducible 
$W\nat_{\Bm}$-module $\r\nat_{\Bla}$. 
\end{enumerate}
\end{lem}

\begin{proof}
First we show (ii). For any $z \in \SX_{\Bm,\nilp}$, Theorem 9.7 (ii) implies that
\begin{equation*}
\tag{9.9.1}
H^i(\SB_z^{(\Bm)},\Ql) \simeq \bigoplus_{\Bmu \in \wt\SP(\Bm)}
                   \r\nat_{\Bmu}\otimes 
          \SH_z^{i - d'_{\Bm} + \dim X_{\Bmu}}\IC(\ol X_{\Bmu}, \Ql)
\end{equation*}
as $W\nat_{\Bm}$-modules. Assume that $z \in X_{\Bla}$ and put $i = 2d_{\Bla}$.
Since $\SH^0\IC(\ol X_{\Bla}, \Ql) = \Ql$, $H^{2d_{\Bla}}(\SB_z^{(\Bm)},\Ql)$ 
contains $\r\nat_{\Bla}$. This proves (ii). 
\par
(ii) implies, in particular, that $\dim \SB_z^{(\Bm)} \ge d_{\Bla}$. 
Put $d = \dim (\pi^{(\Bm)})\iv(X_{\Bla}) - \dim X_{\Bla}$.  Let $X(d)$ be as in 
(7.13.1).  Then $X(d) \cap X_{\Bla}$ is open dense in $X_{\Bla}$. Hence 
$\dim X_{\Bla} \le \dim X(d)$.  By Lemma 7.14 (iii), we have, for any 
$z \in X(d) \cap \SB_z^{(\Bm)}$, 
\begin{equation*}
\dim \SB_z^{(\Bm)} \le \frac{1}{2}(\dim \SX_{\Bm,\nilp} - \dim X(d))
                   \le \frac{1}{2}(\dim \SX_{\Bm,\nilp} - \dim X_{\Bla}) 
                   = d_{\Bla}.
\end{equation*} 
Hence $\dim \SB^{(\Bm)}_z = d_{\Bla}$ and $d = d_{\Bla}$. This proves (i). 
\end{proof}

We show the following result (compare with [Sh, Prop. 8.16]).

\begin{prop} 
Take $z \in X^0_{\Bla}$, and assume that $\Bla \in \wt\SP(\Bm)$. 
\begin{enumerate}
\item
$\dim \SB_z^{(\Bm)} = d_{\Bla}$, 
and 
$H^{2d_{\Bla}}(\SB_z^{(\Bm)},\Ql) \simeq \r\nat_{\Bla}$ as $W\nat_{\Bm}$-modules.
\item
$\dim \SB_z = d_{\Bla}$, and  $H^{2d_{\Bla}}(\SB_z, \Ql) \simeq \wh\r_{\Bla}$ as 
$W_{n,3}$-modules.  Hence the map $z \mapsto H^{2d_{\Bla}}(\SB_z, \Ql)$ 
gives a canonical bijection 
\begin{equation*}
\{ X^0_{\Bla} \mid \Bla \in \SP_{n,3}\} \isom W_{n,3}\wg.
\end{equation*}
\end{enumerate}
\end{prop}

\begin{proof}
We prove (i). We consider the diagram as in (9.1.3), restricted to the nilpotent variety
as in 9.4.  Write $\Bla = (\la^{(1)}, \Bla')$ with $\Bla' = (\la^{(2)}, \la^{(3)})$. 
Put 
\begin{align*}
d_{\la^{(1)}} &= (\dim (\Fg_1)\nil - \dim \SO_1')/2,  \\ 
d_{\Bla'} &= (\dim X'\nil - \dim \SO_2)/2.
\end{align*} 
We note that 
\begin{equation*}
\tag{9.10.1}
d_{\Bla} = d_{\la^{(1)}} + d_{\Bla'}.
\end{equation*}
In fact, by Proposition 7.8 and Lemma 7.5,
\begin{align*}
d_{\Bla} &= ((2n^2 + m_1) - (2\dim U_P + \dim \SO_1 + \dim \SO_2))/2 \\
         &= (\dim G_1 + \dim H_2 - n - \dim \SO_1' - \dim \SO_2)/2 \\
         &= (\dim (\Fg_1)\nil + \dim (\Fh_2)\nil - \dim \SO_1' - \dim \SO_2)/2. 
\end{align*}
Thus (9.10.1) holds. 
\par
Take $z \in X^0_{\Bla}$. Assume that $\Bla \in \wt\SP(\Bm)$. 
By Lemma 9.5 (ii), $\z_1$ gives an isomorphism 
$\z_1\iv(X^0_{\Bla}) \to X^0_{\Bla}$.  Hence there exists a unique 
$z_* \in \z_1\iv(X^0_{\Bla})$ such that $\z_1(z_*) = z$. 
Since $\z_1\iv(X_{\Bla}^0) \subset \CG_{\Bla}$, 
by using the diagram (9.1.3) and its restriction on the nilpotent variety 
(9.4.1), one can find $(x_1, x_2) \in \SO_1' \times \SO_2$ such that 
$\s_1\iv(x_1,x_2) = q_1\iv(z_*)$.  
Note that $\dim \SB^1_{x_1} = d_{\la^{(1)}}, \dim \SB^2_{x_2} = d_{\Bla'}$, 
where $\SB^1$ is the flag variety for $G_1$ and $\SB^2$ is the flag variety for 
$H_2$.  Thus by (9.1.3) together with (9.10.1), we have 
\begin{equation*}
\tag{9.10.2}
\bigl(R^{2d_{\la^{(1)}}}\pi^1_!\Ql\bigr)_{x_1}
       \otimes \bigl(R^{2d_{\Bla'}}\pi^2_!\Ql\bigr)_{x_2}
    \simeq \bigl(R^{2d_{\Bla}}\wt\vf_!\Ql\bigr)_{\xi} 
    \simeq \bigl(R^{2d_{\Bla}}(\vf_1)_!\Ql\bigr)_{z_*},
\end{equation*} 
where $\xi$ is an element in $\s_1\iv(x_1,x_2) = q_1\iv(z_*)$. 
Since $\z_1\iv(X^0_{\Bla}) \simeq X^0_{\Bla}$, we have
\begin{equation*}
\tag{9.10.3}
H^{2d_{\Bla}}(\SB^{(\Bm)}_z,\Ql) \simeq \bigl(R^{2d_{\Bla}}(\pi_1)_!\Ql\bigr)_z
     \simeq \bigl(R^{2d_{\Bla}}(\vf_1)_!\Ql)_{z_*}.
\end{equation*} 
We already know, from the Springer correspondence for $\Fg_1$ and $\Fh_2$, 
\begin{align*}
\dim \bigl(R^{2d_{\la^{(1)}}}\pi^1_!\Ql\bigr)_{x_1} 
       &= \dim H^{2d_{\la^{(1)}}}(\SB^1_{x_1},\Ql)  = \r_{\la^{(1)}}, \\
\dim \bigl(R^{2d_{\Bla'}}\pi^2_!\Ql\bigr)_{x_2} 
       &= \dim H^{2d_{\Bla'}}(\SB^2_{x_2},\Ql) = \wh\r_{\Bla'}. 
\end{align*} 
Then (9.10.2) and (9.10.3) show that 
$\dim H^{2d_{\Bla}}(\SB^{(\Bm)}_z, \Ql) = \dim \r_{\Bla}\nat$. 
By Lemma 9.9 (ii), 
$H^{2d_{\Bla}}(\SB^{(\Bm)}_z,\Ql)$ contains $\r\nat_{\Bla}$. 
It follows that 
$H^{2d_{\Bla}}(\SB^{(\Bm)}_z,\Ql) \simeq \r\nat_{\Bla}$ as 
$W\nat_{\Bm}$-modules. (9.10.2) also shows that $\dim \SB^{(\Bm)}_z = d_{\Bla}$.
This proves (i).
\par
Next we show (ii). 
We consider the decomposition of $H^i(\SB^{(\Bm)}_z,\Ql)$ in (9.9.1).
By Theorem 9.7 (iii), we have a similar decomposition 
\begin{equation*}
\tag{9.10.4}
H^i(\SB_z, \Ql) \simeq \bigoplus_{\Bmu \in \wt\SP(\Bm)}
                  \wh\r_{\Bmu} \otimes 
                \SH_z^{i - d'_{\Bm} + \dim X_{\Bmu}}\IC(\ol X_{\Bmu},\Ql). 
\end{equation*}
Since $H^i(\SB^{(\Bm)}_z,\Ql) = 0$ for $i > 2d_{\Bla}$, (9.9.1) implies 
that $\SH^{i-d'_{\Bm} + \dim X_{\Bmu}}\IC(\ol X_{\Bmu},\Ql) = 0$ for 
any choice of $\Bmu \in \wt\SP(\Bm)$ and of $i > 2d_{\Bla}$. 
It follows, by (9.10.4) that $H^i(\SB_z,\Ql) = 0$ for $i > 2d_{\Bla}$. 
Since $\SB^{(\Bm)}_z \subset \SB_z$, $\dim \SB_z \ge \dim \SB_z^{(\Bm)} = d_{\Bla}$.
Hence $\dim \SB_z = d_{\Bla}$.  
By (i) and (9.9.1), we see that 
$\SH^{2d_{\Bla} - d'_{\Bm} + \dim X_{\Bmu}}\IC(\ol X_{\Bmu},\Ql) = 0$ for any 
$\Bmu \ne \Bla$, and is equal to $\Ql$ for $\Bmu = \Bla$.
Hence by (9.10.4), we have 
$H^{2d_{\Bla}}(\SB_z, \Ql) \simeq \wh\r_{\Bla}$ as $W_{n,3}$-modules. 
This proves (ii).  The proposition is proved.
\end{proof}

\par

\par\vspace{1cm}
\noindent
J. Dong  \\
School of Mathematical Sciences, Tongji University \\ 
1239 Siping Road, Shanghai 200092, P. R. China  \\
E-mail: \verb|dongjunbin1990@126.com |

\newpage
\par\vspace{1cm}
\noindent
T. Shoji \\
School of Mathematical Sciences, Tongji University \\ 
1239 Siping Road, Shanghai 200092, P. R. China  \\
E-mail: \verb|shoji@tongji.edu.cn|

\par\vspace{1cm}
\noindent
G. Yang \\
School of Mathematical Sciences, Tongji University \\ 
1239 Siping Road, Shanghai 200092, P. R. China  \\
E-mail: \verb|yanggao_izumi@foxmail.com|


\bigskip

\begin{thebibliography}{[DJMu]}
\bibitem [AH] {AH} P. Achar and A. Henderson; Orbit closures in the enhanced 
nilpotent cone, Adv. in Math. {\bf 219} (2008), 27-62, Corrigendum , ibid {\bf 228}
(2011), 2984-2988.  
\par
\bibitem [CG] {CG} N. Chriss and V. Ginzburg; 
{\it Representation Theory and Complex Geometry,} Birkhauser, Boston, 
MA, 1997, x + 495pp. 
\par  
\bibitem [Hen] {H1} A. Henderson; Fourier transform, parabolic induction
and nilpotent orbits, Transf. Groups, {\bf 6} (2001), 353-370. 
\par
\bibitem [Hes] {H2} W.H. Hesselink; Nilpotency in classical groups over a field
of characteristic 2, Math. Z. {\bf 166} (1979), 165-181. 
\par
\bibitem [K1] {K1} S. Kato; An exotic Deligne-Langlands correspondence for 
symplectic groups, Duke Math. J. {\bf 148} (2009), 306 - 371. 
\par
\bibitem [K2] {K2} S. Kato; Deformations of nilpotent cones and Springer 
correspondence, Amer. J. Math. {\bf 133} (2011), 519 - 552.
\par 
\bibitem [L1] {L1} G. Lusztig; Intersection cohomology complexes 
on a reductive group, Invent. Math.{\bf 75} (1984), 205-272.
\par
\bibitem [L2] {L2} G. Lusztig; Character sheaves, II, Advances in Math. {\bf 57}, 
(1985), 226 - 265.
\par
\bibitem [R] {R} R.W. Richardson;  Orbits, invariants, and representations 
associated to involutions of reductive groups, Inv. Math. {\bf 66} (1982), 287-312.  
\par
\bibitem [Sh] {Sh} T. Shoji; exotic symmetric spaces of higher level - 
Springer correspondence for complex reflection groups, Transformation 
Groups {\bf 21} (2016), 197-264.  
\par
\bibitem [Spa] {Spa} N. Spaltenstein; Nilpotent classes and sheets of Lie algebras
in bad characteristic, Math. Z. {\bf 181} (1982), 31-48.  
\par
\bibitem [Spr] {Spr} T.A. Springer; Linear algebraic groups. Second edition.
Progress in Math. {\bf 9}, Birkhauser, Boston, Inc., MA, 1998. 
\par
\bibitem [SS] {SS} T. Shoji amd K. Sorlin; Exotic symmetric space over
a finite field, I, Transformation Groups, {\bf 18} (2013), 877 - 929.
\par
\bibitem [SY] {SY} T. Shoji and G. Yang; Generalized Springer correspondence 
for symmetric spaces associated to orthogonal groups, preprint. 
\par
\bibitem [T] {T} R. Travkin; Mirabolic Robinson-Schensted correspondence,
Selecta Math. {\bf 14}, (2009), 727 - 758. 
\par
\bibitem [V] {V} T. Vust; Op\'eration de groupes r\`eductifs dans un type de 
c\^ones presque homogenes, Bull. Soc. Mat. France {\bf 102} (1974), 317-334.  
\par 
\bibitem [X1] {X1} T. Xue; Nilpotent orbits in classical Lie algebras 
over finite field of characteristic 2 and the Springer correspondence, 
Represent. Theory {\bf 13} (2009), 371-393.
\par
\bibitem [X2] {X2} T. Xue; Combinatorics of the Springer correspondence 
for classical Lie algebras and their duals in characteristic 2, preprint. 
\end{thebibliography}
\end{document}